\documentclass[12pt, reqno]{amsart}
\usepackage{amsmath, amsthm, amscd, amsfonts, amssymb, graphicx, color, mathrsfs}
\usepackage[bookmarksnumbered, colorlinks, plainpages]{hyperref}
\usepackage[all]{xy} 
\usepackage{slashed}
\usepackage[export]{adjustbox}
\usepackage{wrapfig}

\newtheorem{theorem}{Theorem}[section]
\newtheorem{lemma}[theorem]{Lemma}

\newtheorem{proposition}[theorem]{Proposition}
\newtheorem{corollary}[theorem]{Corollary}

\theoremstyle{definition}
\newtheorem{definition}[theorem]{Definition}

\newtheorem{assumption}[theorem]{Assumption}

\theoremstyle{remark}
\newtheorem{remark}[theorem]{Remark}
\numberwithin{equation}{section}

\begin{document}
\setcounter{page}{1}

\title[Weak (1,1) boundedness of FIOs with complex phases]{The weak (1,1) boundedness of Fourier integral Operators with complex phases}

\author[D. Cardona]{Duv\'an Cardona}
\address{
  Duv\'an Cardona:
  \endgraf
  Department of Mathematics: Analysis, Logic and Discrete Mathematics
  \endgraf
  Ghent University, Belgium
  \endgraf
  {\it E-mail address} {\rm duvanc306@gmail.com, duvan.cardonasanchez@ugent.be}
  }

\author[M. Ruzhansky]{Michael Ruzhansky}
\address{
  Michael Ruzhansky:
  \endgraf
  Department of Mathematics: Analysis, Logic and Discrete Mathematics
  \endgraf
  Ghent University, Belgium
  \endgraf
 and
  \endgraf
  School of Mathematical Sciences
  \endgraf
  Queen Mary University of London
  \endgraf
  United Kingdom
  \endgraf
  {\it E-mail address} {\rm michael.ruzhansky@ugent.be, m.ruzhansky@qmul.ac.uk}
  }

\thanks{The authors were supported  by the FWO  Odysseus  1  grant  G.0H94.18N:  Analysis  and  Partial Differential Equations and  by the Methusalem programme of the Ghent University Special Research Fund (BOF)
(Grant number 01M01021). Duv\'an Cardona was supported by the Research Foundation-Flanders
(FWO) under the postdoctoral
grant No 1204824N.  Michael Ruzhansky is also supported  by EPSRC grant 
EP/R003025/2.
}


\begin{abstract} Let $T$ be a Fourier integral operator of order $-(n-1)/2$ associated with a canonical relation locally parametrised by a real-phase function. A fundamental result due to Seeger, Sogge, and Stein proved in the 90's, gives the boundedness of $T$  from the Hardy space $H^1$ into $L^1.$ Additionally,   it was shown by  T. Tao the weak (1,1) type of $T$.  In this work, we establish the weak (1,1) boundedness of a Fourier integral operator $T$ of order $-(n-1)/2$ when it has associated a canonical relation parametrised by a complex phase function. This result in the complex-valued setting, cannot be derived from its counterpart in the real-valued case.
\end{abstract} 

\maketitle

\allowdisplaybreaks

\tableofcontents

\section{Overview} Motivated by the developments about complex canonical relations, see e.g. Laptev, Safarov and Vassiliev \cite{Laptev:Safarov:Vassiliev}, Melin and Sj\"ostrand \cite{Melin:Sjostrand1975}, and Martini and M\"uller \cite{MMNG2023}, 
in this paper we establish the weak (1,1) boundedness of Fourier integral operators with complex phases and of order $-(n-1)/2.$ Our approach combines two methods. First, we adapt to our setting the factorisation procedure developed by Tao \cite{Tao} in the proof of the weak (1,1) inequality of these operators in the setting of real-valued phase functions, and the other one came from the global parametrisation of the canonical relation associated to the operator with an appropriate complex phase as observed by the second author in \cite{Ruzhansky:CWI-book}, when extending to Fourier integral operators with complex phases the $H^1$-$L^1$ boundedness result established by Seeger, Sogge and Stein in \cite{SSS}. To present our main Theorem \ref{Main:theorem} we are going to introduce some preliminaries. We consider $\mathbb{R}^n,$ with $n\geq 2,$ and we denote by $$f\mapsto \widehat{f},\,\widehat{f}(\xi):=\smallint_{\mathbb{R}^n} e^{-2\pi ix\cdot \xi}f(x)dx,\,\,\xi\in \mathbb{R}^n,$$ the Fourier transform of an integrable function $f\in L^1(\mathbb{R}^n).$

\subsection{Outline} When working with the real methods of the harmonic analysis one can find three substantial categories related to the boundedness properties of operators: maximal averages, singular integrals, and oscillatory integrals.  In the latter category, oscillatory integrals provide a necessary tool for exploiting the geometric properties related to the notions of {\it curvature} and of {\it orthogonality}. As it was pointed out by Stein in \cite{SteinBook1993}, this family of objects is not easily classified and appears in multiple forms: Bochner-Riesz operators, variants of pseudo-differential operators and of Fourier multipliers, and in the main object treated here, {\it Fourier integral operators}. 

Fourier integral operators appear as solutions to hyperbolic differential equations and then their boundedness properties have a direct relation with the validity of {\it a priori estimates} for these equations. As it was shown e.g. in \cite{Ruzhansky:CWI-book,SSS,Tao},  the geometry and the singularities of the canonical
relation of an operator are reflected in the boundedness properties of operators in $L^p$-
spaces. Due to the compatibility of the  space $H^1$ with the Littlewood-Paley theory, see Fefferman and Stein \cite{FeffermanStein1972}, and with the complex interpolation of operators, a main step in the proof of the $L^p$-boundedness properties of Fourier integral operators (and for other operators arising in harmonic analysis), is the proof of a suitable criterion for its $H^1$-$L^1$-boundedness which is usually  interpolated with its corresponding $L^2$-theory. However, as in the case of pseudo-differential operators, once proved the $H^1$-$L^1$-continuity criterion one expects that the corresponding class of  $H^1$-$L^1$-bounded Fourier integral operators must satisfy the $L^{1}$-$L^{1,\infty}$ boundedness property. This last property is called the {\it weak (1,1) inequality/boundedness} for these operators.    

The purpose of this paper is to prove the weak (1,1) boundedness for the class of Fourier integral operators associated to symbols of order $m=-(n-1)/2$ and of complex phases (our main result is Theorem \ref{Main:theorem} below).  This order is expected from the weak (1,1) inequality for Fourier integral operators with real-valued phase functions proved by Tao in \cite{Tao}, where it was applied a `second decomposition' on the symbol of the operator adapted to a suitable family of ellipsoids with a smooth variation of eccentricity and orientation that avoid Kakeya-type covering lemmas, see Theorem \ref{Tao:theorem}. This approach was combined with a factorisation strategy. The order $m=-(n-1)/2$ also provides bounded Fourier integral operators with real-valued phase functions from the Hardy space $H^{1}$ into $L^1$, see Segger, Sogge, and Stein \cite{SSS}, as well as in the case of complex phase functions as proved by the second author in \cite{Ruzhansky:CWI-book}, see Theorems \ref{SSS:th} and \ref{Ruzhansky}, respectively.

The extension of the theory of Fourier integral operators due to H\"ormander and Duistermaat to the case of complex-valued phase functions was systematically developed by Melin and Sj\"ostrand in \cite{Melin:Sjostrand1975} motivated by the problem of the construction of parametrices, or fundamental solutions for operators of principal type with non-real principal symbols, see for instance \cite{Melin:Sjostrand1976}.

The algebraic properties for the theory of Fourier integral operators with complex phases are well developed but their regularity has
not been much studied, with the exception, e.g. of \cite{Ruzhansky:CWI-book}. In some sense, the use of complex phases provides a more natural
approach to Fourier integral operators. Indeed,  when working with the complex phase approach, the geometric obstructions of the global theory with real phase functions can be avoided and it is a remarkable fact that every Fourier integral operator with a real phase can be globally parametrized
by a single complex phase, see e.g.  \cite{Ruzhansky:CWI-book}. We also refer to Laptev, Safarov, and Vassiliev \cite{Laptev:Safarov:Vassiliev} for this construction.

\subsection{Fourier integral operators with real-phase functions}
In order to present our main result, we record the definition of Fourier integral operators based on the real-phase function approach. So, we present the following simplified version which is suitable for our purposes. 
\begin{definition} A Fourier integral operator $T$  of order $m\in \mathbb{R}$ and with real phase function $\Phi,$  is a continuous linear operator $T:C^\infty_0(\mathbb{R}^n)\rightarrow C^\infty(\mathbb{R}^n) $  of  the form
\begin{equation}\label{FIO:Quantisation}
    Tf(x)=\smallint\limits_{\mathbb{R}^n}e^{i\Phi(x,\xi)}a(x,\xi)\widehat{f}(\xi)d\xi,
\end{equation} where $\widehat{f}$ is the Fourier transform of $f.$ The symbol $a:=a(x,\xi)$ is of order $m\in \mathbb{R}^n,$ that is $a$ satisfies the symbol inequalities
\begin{equation}
    \forall\alpha\in \mathbb{N}_0\,\forall\beta \in \mathbb{N}_0,\,\,|\partial_{x}^\beta\partial_\xi^\alpha a(x,\xi)|\leq C_{\alpha,\beta}(1+|\xi|)^{m-|\alpha|}.
\end{equation} The symbol $a$ has compact support in $x.$ We assume that in an open neighbourhood $V$ of $\pi_1(\textnormal{supp}(a))=\{x:(x,\xi)\in \textnormal{supp}(a)\},$ (with $\overline{V}$ being a compact set) the real-valued phase function $\Phi$ is homogeneous of degree $1$ in $\xi\neq 0,$ smooth in $(x,\xi)\in {V}\times (\mathbb{R}^n\setminus \{0\}),$  and satisfies the non-degeneracy condition
\begin{equation}
  \forall (x,\xi)\in \overline{V}\times (\mathbb{R}^n\setminus \{0\}),\,  \textnormal{det}\left(\frac{\partial^{2}\Phi(x,\xi)}{\partial {x_i}\partial {\xi_j}}\right)_{1\leq i,j\leq n }\neq 0.
\end{equation}   
\end{definition} The following result provides the weak(1,1) inequality for a Fourier integral operator of order $m=-(n-1)/2,$ or equivalently its boundedness from $L^{1}$ into (the weak $L^1$-space) $L^{1,\infty}.$
  
\begin{theorem}[Tao \cite{Tao}, 2004]\label{Tao:theorem} Let $T$ be a Fourier integral operator of order $m=-(n-1)/2$ with a real phase function. Then $T$ is of weak $(1,1)$ type.    
\end{theorem}

The corresponding $L^p$-theory for these operators as developed in \cite{SSS}, can be summarised as follows.
\begin{theorem}[Seeger, Sogge and Stein \cite{SSS}, 1991]\label{SSS:th} Let $1<p<\infty.$ Let $T$ be a Fourier integral operator of order $m=-(n-1)|\frac{1}{p}-\frac{1}{2}|,$ with a real phase function. Then $T$ is bounded from $L^p$ to $L^p.$ When $p=1,$ then $T$ is bounded from $H^1$ to $L^1.$
\end{theorem}
The case $p=2$ in Theorem \ref{SSS:th} was proved by Duistermaat and H\"ormander  \cite{Duistermaat-Hormander:FIOs-2}\footnote{See also Èskin \cite{Eskin1970}.} using the $T^*T$-method and extended independently to the setting of Fourier integral operators with complex phases by H\"ormander in \cite{Ho2} and by Melin and Sj\"ostrand in \cite{Melin:Sjostrand1976}. One of the fundamental differences when working on the boundedness of operators in the setting of complex-valued phase functions, compared with the real-valued setting is that one has to deal with the boundedness properties of pseudo-differential operators in the class $S^{0}_{1/2,1/2}.$ One reason for this is the analysis of the oscillatory term $e^{\pm (1+i\tau)\textnormal{Im}(\Phi)},$ see Melin and Sj\"ostrand \cite[Page 394]{Melin:Sjostrand1976}.

\subsection{Fourier integral operators with complex-phase functions} Let $X$ and $Y$ be smooth manifolds of dimension $n$ and let us denote by $\widetilde{ T^*(X\times Y)\setminus 0}$  the almost analytic continuation of  $ T^*(X\times Y)\setminus 0.$ For this notion see H\"ormander \cite{Hormander1969} and Nirenberg \cite{Nirenberg71}, see also Subsection \ref{2:2}. Let   $$  C\equiv C_{\Phi}=\{(x,d_x\Phi,y, d_y\Phi):d_\theta \Phi=0\}\subset\widetilde{ T^*(X\times Y)\setminus 0}$$  be a smooth positive homogeneous canonical relation parametrised by a phase function of {\it positive type}, namely, it satisfies $\textnormal{Im}(\Phi)\geq 0,$ see Subsections \ref{2:2} and \ref{2:4} for details. 

The class  $I^\mu_\rho(X,Y;C)$ of Fourier integral operators is determined modulo $C^{\infty}$ by those integral operators $T$ that locally have an integral  kernel of the form
\begin{align}\label{Kernel:FIO:C}
    A(x,y)=\smallint\limits_{\mathbb{R}^N}e^{i\Phi(x,y,\theta)}a(x,y,\theta)d\theta,
\end{align}where $a(x,y,\theta)$ is a symbol of order $\mu+\frac{N-n}{2}.$ 
By the equivalence-of-phase-function theorem, we can always microlocally assume that $N=n,$ and that the symbol $a$ satisfies the Kohn-Nirenberg type estimates
\begin{equation}
    |\partial_{x,y}^{\beta}\partial_\theta^\alpha a(x,y,\theta)|\leq C_{\alpha,\beta}(1+|\theta|)^{\mu-|\alpha|},
\end{equation}locally uniformly in $(x,y).$  According to Lemma 2.1 of \cite{Melin:Sjostrand1976}, one can always re-write the kernel in \eqref{Kernel:FIO:C} as 
\begin{align}\label{Kernel:FIO:C:2}
    A(x,y)=\smallint\limits_{\mathbb{R}^N}e^{2\pi i(\Phi(x,\theta)-y\cdot\theta )}a(x,y,\theta)d\theta,
\end{align}where the function $\Phi$ is of positive type. Note that the amplitude $a$ in \eqref{Kernel:FIO:C:2} is probably different from the one in \eqref{Kernel:FIO:C}, however, we keep the same notation for simplicity. To be able to use the Fourier inversion formula we will assume that $a(x,y,\theta)$ is independent of $y\in \mathbb{R}^n,$ and compactly supported in $x.$ So, we will work with symbols instead of amplitudes, in order to recover the simplified form
\begin{equation}
    Tf(x)=\smallint\limits_{\mathbb{R}^n}e^{i\Phi(x,\xi)}a(x,\xi)\widehat{f}(\xi)d\xi.
\end{equation} 
To analyse the boundedness of a Fourier integral operator associated to the phase $\Phi$ we require the following {\it local graph condition}.

\begin{assumption}\label{assumption:local:graph}
 There exists $\tau\in \mathbb{C},$ such that 
    \begin{equation}\label{ass}
     \Phi_\tau:=   \textnormal{Re}(\Phi)+\tau\textnormal{Im}(\Phi) \textnormal{ defines }\textnormal{ a } \textnormal{ local } \textnormal{ graph } \textnormal{ in }T^{*}(\widetilde{X\times Y})\setminus 0.
    \end{equation}   Locally, this means that 
\begin{equation}\label{non:deg:con:Ruz}
        \textnormal{det}\partial_x \partial_\xi(\textnormal{Re}(\Phi(x,\xi))+\tau\textnormal{Im}(\Phi(x,\xi)))\neq 0,
\end{equation}
on the support of the symbol $a,$  for all $\xi\neq 0.$
\end{assumption}
\begin{remark}
    According to Lemma 1.8.1 of \cite{Ruzhansky:CWI-book}, the local graph condition \eqref{ass} in Assumption \ref{assumption:local:graph} is equivalent to the same assumption for almost every $\tau\in \mathbb{C}.$ Moreover, the determinant in  \eqref{non:deg:con:Ruz} is a polynomial function in $\tau$ which means that if   \eqref{ass} holds for some $\tau\in \mathbb{C},$ then it holds for all but finitely many $\tau \in \mathbb{C}.$ To summarise this discussion, Assumption \ref{assumption:local:graph} is equivalent to the existence of the same condition for some, and then for almost every $\tau\in \mathbb{C}.$
\end{remark} 
Before presenting our main result we record the $L^p$-theory of Fourier integral operators with complex phases as developed by the second author in  \cite[Theorems 1.7.1 and 1.7.2]{Ruzhansky:CWI-book} including the endpoint case $p=1.$ 
\begin{theorem}\label{Ruzhansky}  Let $T$ be a Fourier integral operator of order $\mu=-(n-1)|\frac{1}{p}-\frac{1}{2}|$ associated with a canonical relation parametrised by a complex phase $\Phi$  of positive type and satisfying the local graph condition \eqref{ass}. If $p=1$ then $T$ is bounded from $H^1_{comp}$ to $L^1_{loc},$ and for $1<p<\infty,$  $T$ is bounded from $L^p_{comp}$ to $L^p_{loc}.$    
\end{theorem}

\subsection{Main result}

The following is the main result of this paper. We will assume that the complex phase function is of {\it positive type}, in such a way that $\textnormal{Im}(\Phi(x,\xi))> 0,$ when $|\xi|\neq 0.$
\begin{theorem}\label{Main:theorem} Let $T$ be a Fourier integral operator of order $\mu=-(n-1)/2$ associated with a canonical relation parametrised by a complex phase $\Phi$  of positive type in such a way that $\textnormal{Im}[\Phi(x,\xi)]>0$ when $|\xi|\neq 0,$ and satisfying the local graph condition \eqref{ass}. Then $T$ is of weak $(1,1)$ type.
\end{theorem}
\begin{remark}
    We observe that the condition on the phase function in Theorem \ref{Main:theorem}, that $\textnormal{Im}[\Phi(x,\xi)]>0$ when $|\xi|\neq 0,$ is not restrictive,  see Laptev, Safarov, and Vassiliev \cite{Laptev:Safarov:Vassiliev}.
\end{remark}
Below we briefly discuss some aspects of the main differences between the regularity properties of Fourier integral operators in the setting of complex-valued phase functions in comparison with the setting of real-valued phases.
\begin{remark} There is a simple argument to show that our main Theorem \ref{Main:theorem} does not follow from its real counterpart  in Theorem \ref{Tao:theorem}. Indeed, any Fourier integral operator $T\in I^\mu(X,Y;C)$ in the form
\begin{equation}
  Tf(x)= \smallint\limits_{\mathbb{R}^n}e^{i(\textnormal{Re}(\Phi(x,\xi))+\tau \textnormal{Im}(\Phi(x,\xi) )}a(x,\xi)e^{-(1+i\tau)\textnormal{Im}(\Phi(x,\xi) )}\widehat{f}(\xi)d\xi ,
\end{equation}with real $\tau\neq0$ and with real phase function $$\Phi_{\tau}(x,\xi):=\textnormal{Re}(\Phi(x,\xi))+\tau \textnormal{Im}(\Phi(x,\xi) )$$ has the symbol
$$ a_\tau(x,\xi)= a(x,\xi)e^{-(1+i\tau)\textnormal{Im}(\Phi(x,\xi) )}$$
in the class $S^{\mu}_{1/2,1/2}$ if $a\in S^\mu,$ see Melin and Sj\"ostrand \cite[Page 394]{Melin:Sjostrand1976}.  Note that $T$ is bounded on $L^p(\mathbb{R}^n),$ with $1<p<\infty$ if, for example,  $$\mu\leq -\left(n-\frac{1}{2}\right)|1/p-1/2|.$$
However when $p\rightarrow 1^{+}$ this order is not good enough as the one in Theorem \ref{Tao:theorem}, or compared with the order $-(n-1)/2$ for the $H^1$-$L^1$-boundedness  proved by Seeger, Sogge and Stein in \cite{SSS} for real phase functions or even as proved in \cite{Ruzhansky:CWI-book} for complex phase functions.
In order to improve this order in Theorem  \ref{Main:theorem} we need to avoid the reduction to a real phase function, making some strategic modifications to the argument given in Tao \cite{Tao}, in particular taking into account that we have to deal with dyadic components of the symbol $a_\tau(x,\xi)= a(x,\xi)e^{-(1+i\tau)\textnormal{Im}(\Phi(x,\xi) )}.$ 
\end{remark}
\begin{remark}
We remark that in Tao \cite{Tao} it has been introduced a {\it factorisation approach} for the proof of the weak (1,1) inequality of Fourier integral operators, see Section \ref{Section:Main:theorem:2} for details. When $T$ has a real-valued phase function, according to such an approach a portion of the operator, the non-degenerate part $T_{\textnormal{non-deg}},$ is factorised as $T_{\textnormal{non-deg}}=SA$ modulo an error operator $E=T_{\textnormal{non-deg}}-SA.$ Here, $S$ is a pseudo-differential operator of order zero and its symbol is essentially $|\xi_n|^{(n-1)/2}a(x,\xi),$ when $|\xi_n|\gg 1,$ where $a$ is the symbol of $T.$ Note that the weak (1,1) boundedness of $S$ is a consequence of the Calder\'on-Zygmund theory.
 However, compared with the case of Fourier integral operators with complex phases, one has to redefine the operator $S,$ because in this setting  $|\xi_n|^{(n-1)/2}a_\tau(x,\xi)\in S^{0}_{1/2,1/2},$ for $|\xi_n|\gg 1.$ Nevertheless, according to the $L^1$-theory for the H\"ormander classes $S^{m}_{1/2,1/2},$ $m\in \mathbb{R},$ see \'Alvarez and Hounie \cite{Hounie}, a pseudo-differential operator with a symbol in these classes is of weak (1,1) type when $m\leq -n/4.$ In order to avoid the loss of regularity caused by the order $-n/4$, we will keep the symbol of the operator $a$ in the pseudo-differential part $S,$ and we will absorb the oscillating term $e_{\Phi}=e^{-(1+i\tau)\textnormal{Im}(\Phi(x,\xi) )}$ into the average operator $A. $ Essentially, this term will modify the argument in \cite{Tao}, where the $L^1$-boundedness of the averaging operator $A$ depends on the uniformly $L^1$-boundedness of the Littlewood-Paley operators, but even, perturbating this part with the oscillating symbol $e_{\Phi}$ we are able to preserve the $L^1$-boundedness of $A.$ In order to follow this strategy, we will make a fundamental reduction of our Theorem \ref{Main:theorem} to the case of phase functions of {\it strictly positive type}  in Section \ref{reduction}, see Theorem \ref{Main:theorem:2}.  We present the proof of Theorem \ref{Main:theorem} in Section \ref{reduction}.  The authors thank Terence Tao for suggesting the factorisation approach used in this manuscript.
\end{remark}
In the next Section \ref{Section:pre}, we present the preliminaries and the notation used in this work. We will continue the discussion about the regularity properties of Fourier integral operator in Section \ref{reduction}, its main result is Theorem \ref{Main:theorem:2} and its proof is presented in Section \ref{Section:Main:theorem:2}. We finish this manuscript with Section \ref{Final:rem} where we present an application of our main theorem and also we provide a bibliographic discussion.

\section{Preliminaries}\label{Section:pre}
Let $M,$ $X$ and $Y$ be paracompact smooth real manifolds. We assume that $X$ and $Y$  are manifolds of dimension $n.$  Here, we can take $M=X,Y$ or $M=X\times Y.$ When $Y\subset \mathbb{R}^n,$  $L^{1}(Y)$ is defined in the usual form, namely,  
 $ 
    f\in L^{1}(Y) \Longleftrightarrow  \Vert f\Vert_{L^1(Y)}:=\smallint_{Y}|f(x)|dx<\infty,
$ where $dx$ denotes the Lebesgue measure on $\mathbb{R}^n.$ The measure of a Borel set $A\subset \mathbb{R}^n$ is denoted by $|A|,$ ans its characteristic function is denoted by $1_A$ or by $\chi_A.$  We also consider the weak $L^1$-space $L^{1,\infty}(Y)$ defined by the seminorm
$  \Vert f\Vert_{L^{1,\infty}(Y) }:=\sup_{t>0}t|\{y\in Y:|f(y)|>t\}|. $ 

If $X$ is a compact manifold without boundary we define the spaces $L^{1}(X)$ and $L^{1,\infty}(X)$ by using partitions of unity (subordinated to local coordinate systems). We say that an operator $T:L^{1}(Y)\rightarrow L^{1,\infty}(Y)$ is of {\it weak (1,1) type} if there exists a constant $C>0$ such that for all $f\in L^{1}(Y),$ the inequality $\Vert Tf\Vert_{L^{1,\infty}}(Y)\leq C\Vert f\Vert_{L^{1}(Y)}$ holds. A prototype of operators of weak (1,1) type is the family of Calder\'on-Zygmund singular integral operators, see Calder\'on and Zygmund \cite{CalderonZygmund1952}.

Below we precise the notation and the fundamental notions related to the theory of Fourier integral operators of importance for this work. We start with the properies of {\it conic Lagrangian manifolds} in the next subsection.
\subsection{Basics on symplectic geometry} Now, let us follow  \cite[Chapter I]{Ruzhansky:CWI-book}.
 We recall that a $2$-form $\omega$ is called {\it symplectic}  on $M$ if $d\omega =0,$ and for all $x\in M,$ the bilinear form $\omega_x$  is antisymmetric and non-degenerate on $T_{x}M$. The canonical symplectic form $\sigma_M$ on $M$ is defined as follows. Let $\pi:=\pi_M:T^*M\rightarrow M,$ be the canonical projection. For any $(x,\xi)\in T^*M,$ let us consider the linear mappings
\begin{equation*}
    d\pi_{(x,\xi)}:T_{(x,\xi)}(T^*M)\rightarrow T_{x}M\textnormal{   and,   }\xi:T_{x}M\rightarrow \mathbb{R}.
\end{equation*}The composition $\alpha_{(x,\xi)}:=\xi\circ d\pi_{(x,\xi)}\in T^{*}_{(x,\xi)}(T^*M),$ that is
$
   \xi\circ d\pi_{(x,\xi)}:T_{(x,\xi)}(T^*M)\rightarrow  \mathbb{R},
$ defines  a 1-form $\alpha $ on $T^*M.$ With the notation above, the canonical symplectic form $\sigma_M$ on $M$ is defined by 
\begin{equation}\label{sigmaM}
    \sigma_M:=d\alpha.
\end{equation}Since $\sigma_M$ is an exact form, it follows that $d\sigma_M=0$ and then  that $\sigma_M$ is symplectic. If $M=X\times Y,$ it follows that  $\sigma_{X\times Y}=\sigma_X\oplus -\sigma_Y .$ Now we record the family of submanifolds that are necessary when one defines the canonical relations.
\begin{itemize}
    \item Let $M$ be of dimension $n.$ A submanifold $\Lambda\subset T^*M$ of dimension $n$ is called {\it Lagrangian} if 
\begin{equation*}
    T_{(x,\xi)}\Lambda=( T_{(x,\xi)}\Lambda)^{\sigma}:=\{v\in T_{(x,\xi)}(T^*M):\sigma_M(v,v')=0,\,\forall v'\in T_{(x,\xi)}\Lambda \}.
\end{equation*}
\item We say that  $\Lambda\subset T^*M\setminus 0$ is {\it conic} if $(x,\xi)\in \Lambda,$ implies that $(x,t\xi)\in \Lambda,$ for all $t>0.$ 
\item Let $\Sigma\subset X$ be a smooth submanifold of $X$ of dimension $k.$ Its conormal bundle in $T^*X$ is defined by
\begin{equation}
    N^{*}\Sigma:=\{(x,\xi)\in T^*X: \,x\in \Sigma,\,\,\xi(\delta)=0,\,\forall \delta\in T_{x}\Sigma  \}.
\end{equation}
\end{itemize}
The following facts characterise the conic Lagrangian submanifolds of $T^*M.$
\begin{itemize}
    \item Let $\Lambda\subset T^*M\setminus 0,$ be a closed sub-manifold of dimension $n.$ Then $\Lambda$ is a conic Lagrangian manifold if and only if the 1-form $\alpha$ in \eqref{sigmaM} vanishes on $\Lambda.$
    \item Let $\Sigma\subset X,$ be a submanifold of dimension $k.$ Then its conormal bundle $N^*\Sigma$ is a conic Lagrangian manifold. 
\end{itemize}
The Lagrangian manifolds have the following property.
\begin{itemize}
    \item Let $\Lambda\subset T^{*}M\setminus 0,$ be a conic Lagrangian manifold and let 
    \begin{equation}
        d\pi_{(x,\xi)}:T_{(x,\xi)}\Lambda\rightarrow T_{x}M,
    \end{equation}have constant rank equal to $k,$ for all $(x,\xi)\in \Lambda.$ Then, each $(x,\xi)\in \Lambda$ has a conic neighborhood $\Gamma$ such that 
    \begin{itemize}
        \item[1.] $\Sigma=\pi(\Gamma\cap \Lambda)$ is a smooth manifold of dimension $k.$
        \item[2.] $\Gamma\cap \Lambda$ is an open subset of $N^*\Sigma.$
    \end{itemize}
\end{itemize}

The Lagrangian manifolds have a local representation defined in terms of {\it phase functions} that can be defined as follows. For this, let us consider a local trivialisation $M\times (\mathbb{R}^n\setminus 0), $ where we can assume that $M$ is an open subset of $\mathbb{R}^n.$

\begin{definition}[Real-valued phase functions]\label{Real:defin} Let $\Gamma$ be a cone in $M\times (\mathbb{R}^N\setminus 0). $ A smooth function $\phi:M\times (\mathbb{R}^N\setminus 0)\rightarrow \mathbb{R}, $ $(x,\theta)\mapsto \phi(x,\theta),$ is a real {\it phase function} if, it is homogeneous of degree one in $\theta$ and has no critical points as a function of $(x,\theta),$ that is 
\begin{equation}
 \forall t>0,\,   \phi(x,t\theta)=t\phi(x,\theta), \textnormal{   and   }d_{(x,\theta)}\phi(x,\theta)\neq 0,\forall (x,\theta)\in M\times (\mathbb{R}^N\setminus 0).
\end{equation}Additionally, we say that $\phi$ is a  {\it non-degenerate phase function} in $\Gamma,$ if for any $(x,\theta)\in \Gamma$ such that $d_{\theta}\phi(x,\theta)=0,$ one has that 
\begin{equation}
    d_{(x,\theta)}\frac{\partial \phi}{\partial\theta_{j}}(x,\theta),\,\,1\leq j\leq N,
\end{equation}is a system of linearly independent vectors on $\mathbb{R}$.
\end{definition}
The following facts describe locally a Lagrangian manifold in terms of a phase function.
\begin{itemize}
    \item Let $\Gamma$ be a cone in $M\times (\mathbb{R}^N\setminus 0), $ and let $\phi$ be a non-degenerate phase function in $\Gamma.$ Then, there exists an open cone $\tilde{\Gamma}$ containing $\Gamma$ such that the set
    \begin{equation}
        U_{\phi}=\{(x,\theta)\in \tilde{\Gamma}:d_\theta \phi(x,\theta)=0\},
    \end{equation}is a smooth conic  sub-manifold of $M\times (\mathbb{R}^N\setminus 0)$ of dimension $n.$ The mapping 
    \begin{equation}
        L_{\phi}: U_\phi\rightarrow T^*M\setminus 0,\,\, L_{\phi}(x,\theta)=(x,d_x\phi(x,\theta)),
    \end{equation}is an immersion. Let us denote $\Lambda_\phi=L_{\phi}(U_\phi).$
    \item Let $\Lambda\subset T^*M\setminus 0$ be a sub-manifold of dimension $n.$ Then,  $\Lambda$ is a conical Lagrangian manifold if and only if  every $(x,\xi)\in \Lambda$ has a conic neighborhood $\Gamma$ such that $\Gamma\cap \Lambda=\Lambda_\phi, $ for some non-degenerate phase function $\phi.$ 
\end{itemize}
\begin{remark}
The cone condition on $\Lambda$ corresponds to the homogeneity of the phase function.
\end{remark}
\begin{remark} Although we have presented the previous Definition \ref{Real:defin} of non-degenerate real phase function in the case of a real function of  $(x,\theta),$ the same can be defined if one considers functions of $(x,y,\theta).$ Indeed, a real-valued phase function $\phi(x,y,\theta)$ homogeneous of order 1 at $\theta\neq 0$  that satisfies the following two conditions
 \begin{equation}
        \textnormal{det}\partial_x \partial_\theta(\phi(x,y,\theta))\neq 0,\,\, \textnormal{det}\partial_y \partial_\theta(\phi(x,y,\theta))\neq 0,\,\,\theta\neq 0,
    \end{equation} is called  non-degenerate. 
\end{remark}
\begin{remark}
 For a symplectic manifold $M$ of dimension $2n,$ we will denote by $\widetilde{M}$ its almost analytic continuation in $\mathbb{C}^{2n}$ (see   \cite[Page 10]{Ruzhansky:CWI-book} for details about the construction of $\widetilde{M},$ as well as many details in \cite{Melin:Sjostrand1975}). For completeness, we present such a notion in the following subsection.
\end{remark}
\subsection{Almost analytic continuation of real manifolds}\label{2:2} To define the almost analytic continuation of a manifold we require some preliminaries. We record that a function $f:U\to \mathbb{C}$ defined on an open subset $U\subset \mathbb{C}^n$ is called {\it almost analytic} in $U_{\mathbb{R}}:=U\cap \mathbb{R}^n,$ if $f$ satisfies the Cauchy-Riemann equations on $U_{\mathbb{R}},$ that is, if $\overline{\partial}f$ and all its derivatives vanish in $U_\mathbb{R}.$ Here, as usual, $\overline{\partial}=1/2(\partial_x+i\partial_y).$

One defines an {\it almost analytic extension} of a manifold $M$ requiring that the corresponding coordinate functions are almost analytic in $M.$ Here, we are going to present this notion carefully. 

Let $\Omega\subset \mathbb{R}^n$ and let $\rho:\Omega\rightarrow \mathbb{R}$ be a non-negative and Lipschitz function. A function $f:\Omega\rightarrow \mathbb{R}$ is called $\rho$-flat in $\Omega$ if for every compact set $K\subset \Omega,$ and for every integer $N\geq 0,$ one has the growth estimate $|f(x)|\lesssim_{K,N} \rho(x)^{N},$ for all $x\in K.$ This notion defines an equivalence relation on the space of mapping on $\Omega:$
$$f\textnormal{ and }g \textnormal{ are }\rho\textnormal{-equivalent if }f-g\textnormal{ is }\rho\textnormal{-flat on }\Omega. $$ If $K_0\subset \Omega$ is a compact set, we will say that $f$ is {\it flat } on $K_0,$ if it is $\rho$-flat with $\rho(x):=\textnormal{dist}(x,K_0).$ One has the following property:
\begin{itemize}
    \item let $f\in C^{\infty}(\Omega)$ be $\rho$-flat. Then all its derivatives are $\rho$-flat and $f$ is flat on $K_0$ if and only if $D^{\alpha}f=0$ for all $x\in K_0$ and all $\alpha.$ 
\end{itemize}
Let $G$ be an open subset of $\mathbb{C}^n$ and let $K$ be a closed subset of $G.$ A  function $f\in C^\infty(G)$ is {\it almost analytic} on $K$ if the functions $\partial_j f$ are flat on $K,$ for all $1\leq j\leq n.$ For an open set  $\Omega\subset \mathbb{R}^n,$ we will denote
\begin{equation}
   \tilde{\Omega}:=\Omega+i\mathbb{R}^n.
\end{equation}One can make the identification $\Omega\cong \tilde{\Omega}\cap\{z:\textnormal{Im}(z)=0 \}.$ Note that
every function $f\in C^{\infty}(G_{\mathbb{R}})$ defines an equivalence class of almost analytic functions on $G_\mathbb{R},$ which consist of functions in $C^\infty(G)$ which are almost analytic in $G_\mathbb{R},$ modulo functions being flat on $G_{\mathbb{R}}.$ Any representative of this class is called an {\it almost analytic continuation of }$f$ in $G.$

Let $O$ be an open subset of $\mathbb{C}^n.$ Let $M$ be a smooth submanifold of codimension $2k$ of $O,$ and let $K$ be a closed subset of $O.$ Then, $M$ is called {\it almost analytic }on $K,$ if for any point $z_0\in K,$ there exists an open set $U\subset O,$ and $k$-functions $f_{j},$ such that every $f_j$ is almost analytic on $K\cap U,$
and such that in $U,$ $M$ is defined by the zero-level sets $f_j(z)=0,$ with the differentials $df_{j},$ $1\leq j\leq k,$ being linearly independent over $\mathbb{C}.$

Two almost analytic submanifolds $M_1$ and $M_2$ of $O$ are {\it equivalent} if they have the same dimension and the same intersection with $\mathbb{R}^n,$ namely, that $$M_{1,\mathbb{R}}=M_1\cap \mathbb{R}^n=M_2\cap \mathbb{R}^n=M_{2,\mathbb{R}}=:M_{\mathbb{R}},$$ and locally $f_j-g_j$ are flat functions on $M_{\mathbb{R}}$ where $f_j$ and $g_j$ are the functions that define $M_1$ and $M_2,$ respectively. 

A real manifold $\Omega$ defines an equivalence class of almost analytic manifolds in $\tilde{\Omega}.$ A representative of this class is called an {\it almost analytic continuation } of $\Omega$ in $\mathbb{C}^n.$
Now, let us consider:
\begin{itemize}
    \item a real symplectic manifold $M$ of dimension $d=2n,$ and let $\tilde{M}$ be (modulo its equivalence class) its almost analytic continuation in $\mathbb{C}^{2n}.$
    \item Let $\Lambda\subset \tilde{M}$ be an almost analytic sub-manifold containing the real point $\rho_0\in M,$ and let $(x,\xi)$ be its real symplectic coordinates in a neighbourhood $W$ of $\rho_0.$ 
    \item Let $(\tilde{x},\tilde{\xi})$ be almost analytic continuations of the coordinates $(x,\xi)$ in $W,$ in such a way that $(\tilde{x},\tilde{\xi})$ maps $ \tilde{W}$  diffeomorphically on an open subset of $\mathbb{C}^{2n}.$
    \item Let $g$ be an almost analytic function such that $\textnormal{Im}(g)\geq 0,$ in $\mathbb{R}^n,$ and such that $\Lambda$ is defined in a neighbourhood of $\rho_0$ by the equations $\tilde{\xi}=\partial_{\tilde{x}}g(x),$ $\tilde{x}\in \mathbb{C}^n.$
\end{itemize}
Then, an almost analytic manifold $\Lambda$ satisfying this property, in some real symplectic coordinate system at every point is called a {\it positive Lagrangian manifold}. We conclude this discussion with the following result,  see e.g. Theorem 1.2.1 in \cite{Ruzhansky:CWI-book}.
\begin{proposition}
    Let $M,$ $\Lambda$ and $W$ be as before. If $(\tilde y, \tilde \eta)$ is another almost analytic continuation of coordinates in $\tilde{W}$ and $\Lambda$ is defined by the equation $\tilde\eta=H(\tilde y),$ in a neighbourhood of $\rho_0,$ then $\Lambda$ is locally equivalent to the manifold $\tilde{\eta}=\partial_{\tilde y}h,$ $\tilde{y}\in \mathbb{C}^n,$ where $h$ is an almost analytic function and $\textnormal{Im}(h)\geq 0.$
\end{proposition}

\subsection{Fourier integral operators with real-valued phases}

We can assume that  $X, Y$ are open sets in $\mathbb{R}^n$. One defines the class of
  Fourier integral operators $T\in I^\mu_{\rho}(X, Y;\Lambda)$ by
  the (microlocal) formula
   \begin{equation}\label{EQ:FIO}
     Tf(x)=\smallint\limits_Y\smallint\limits_{\mathbb{R}^N}
   e^{i\Psi(x,y,\theta)} a(x,y,\theta) f(y)d\theta\;dy,
   \end{equation} 
   where the symbol $a$ is a smooth function locally  in the class $S^\mu_{\rho,1-\rho}(X\times Y\times  (\mathbb{R}^n\setminus 0) ),$ with $1/2\leq \rho\leq 1.$ This means that  $a$ satisfies the symbol inequalities
   $$ |\partial_{x,y}^{\beta}\partial_\theta^\alpha
   a(x,y,\theta)|
    \leq C_{\alpha\beta}(1+|\theta|)^{\mu-\rho|\alpha|+(1-\rho)|\beta|},$$
   for $(x,y)$ in any compact subset $K$ of $X\times Y,$ and $\theta \in \mathbb{R}^N\setminus 0,$ while the real-valued phase function $\Psi$ satisfies 
     the following properties:
\begin{itemize}
    \item[1.] $\Psi(x,y,\lambda\theta)=
      \lambda\Psi(x,y,\theta),$ for all $\lambda>0$;
    \item[2.] $d\Psi\not=0$;
    \item[3.] $\{d_\theta\Psi=0\}$ is smooth 
      (e.g. $d_\theta\Psi=0$ implies
      $d_{(x,y,\theta)}\frac{\partial\Psi}{\partial\theta_j}$ are linearly
      independent).
\end{itemize}
Here $\Lambda \subset T^*(X\times Y)\setminus 0$ is a Lagrangian manifold locally parametrised by the phase function $\Psi,$
 $$ \Lambda=\Lambda_\Psi=\{(x,d_x\Psi,y,d_y\Psi): d_\theta\Psi=0\}.$$
\begin{remark}
  The   canonical relation associated with $T$
  is the conic Lagrangian manifold in 
  $T^*(X\times Y)\backslash 0$,  defined by
  $\Lambda'=\{(x,\xi,y,-\eta): (x,\xi,y,\eta)\in \Lambda\}.$
In view of the H\"ormander equivalence-of-phase-functions theorem (see e.g. Theorem 1.1.3 in \cite[Page 9]{Ruzhansky:CWI-book}), the notion of Fourier integral operator becomes independent of the choice of a particular phase function associated to a Lagrangian manifold $\Lambda.$  Because of the diffeomorphism $\Lambda\cong \Lambda',$ we do not distinguish between $\Lambda $ and $\Lambda'$ by saying also that $\Lambda$ is the canonical relation associated with $T.$  
\end{remark}

\subsection{Fourier integral operators with complex-valued phases}\label{2:4}
Now we record the following result for canonical relations, see e.g. Theorem 1.2.1 in \cite{Ruzhansky:CWI-book}.
\begin{proposition}
    Let $\Psi$ be a phase function of the positive type that is defined in a conical neighbourhood. Let $\tilde{\Psi}$ be an almost analytic homogeneous continuation of $\Psi$ in a canonical neighbourhood in $\mathbb{C}^n\times \mathbb{C}^n\times ({\mathbb{C}^{n}}\setminus \{0\}).$ Let \begin{equation}
        C_{\tilde \Psi}=\{(\tilde x, \tilde y, \tilde \theta)\in \mathbb{C}^n\times \mathbb{C}^n\times ({\mathbb{C}^{n}}\setminus \{0\}):\partial_{\theta}\tilde\Psi(\tilde x, \tilde y, \tilde \theta)=0 \}.
    \end{equation}Then, the image of the set $C_{\tilde \Psi}$ under the mapping 
    \begin{equation}
      C_{\tilde \Psi}  \ni(\tilde x, \tilde y, \tilde \theta)\mapsto (\tilde x, \partial_{\tilde x}\tilde\Psi,\partial_{\tilde y}\tilde\Psi )|_{(\tilde x, \tilde y, \tilde \theta)}\in \mathbb{C}^n\times ({\mathbb{C}^{n}}\setminus \{0\})\times ({\mathbb{C}^{n}}\setminus \{0\})
    \end{equation} is a local positive Lagrangian manifold.
\end{proposition}
Let $X$ and $Y$ be smooth manifolds of dimension $n$ and let us denote by $\widetilde{ T^*(X\times Y)\setminus 0}$  the  almost analytic continuation  of  $ T^*(X\times Y)\setminus 0.$ Let   $$  C\equiv C_{\Phi}=\{(x,d_x\Phi,y, d_y\Phi):d_\theta \Phi=0\}\subset\widetilde{ T^*(X\times Y)\setminus 0}$$  be a {\it smooth positive homogeneous canonical relation.} This means that $C$ is locally parametrised by a complex phase function $\Phi(x,y,\theta),$ satisfying the following properties:
\begin{itemize}
    \item[1.] $\Phi$ is homogeneous of order one in $\theta:$  $\Phi(x,y,t\theta)=t\Phi(x,y,\theta),$ for all $t>0,$
    \item[2.] $\Phi$ has no critical points on its domain, i.e. $d\Phi\neq 0;$
    \item[3.] $\{d_\theta \Phi(x,y,\theta)=0\}$ is smooth (e.g. $d_\theta \Phi=0$ implies that $d\frac{\partial \Phi}{\partial\theta_j}$ are independent over $\mathbb{C}$),
    \item[4.] $\textnormal{Im}(\Phi)\geq 0.$
\end{itemize}When the last condition is satisfied one says that the complex phase function $\Phi$ is of {\it positive type}. Then, the class  $I^\mu_\rho(X,Y;C)$ of Fourier integral operators $T$ is determined modulo $C^{\infty}$ by those integral operators $T$ that locally have an integral  kernel of the form
\begin{align*}
    A(x,y)=\smallint\limits_{\mathbb{R}^N}e^{i\Phi(x,y,\theta)}a(x,y,\theta)d\theta,
\end{align*}where $a(x,y,\theta)$ is a symbol of order $\mu+\frac{N-n}{2}.$ 
By the equivalence-of-phase-function theorem, we can always assume that $N=n,$ and that the symbol $a$ satisfies the type $(\rho,1-\rho)$-estimates
\begin{equation}
    |\partial_{x,y}^{\beta}\partial_\theta^\alpha a(x,y,\theta)|\leq C_{\alpha,\beta}(1+|\theta|)^{\mu-\rho|\alpha|+(1-\rho)|\beta|},\,\,\frac{1}{2}\leq \rho\leq 1,
\end{equation}locally uniformly in $(x,y).$  

\subsection{Asymptotics for oscillatory integrals}
Now, we recall two principles used in this work: the principle of non-stationary phase and the principle of stationary phase, respectively.

\begin{proposition}[\cite{SteinBook1993}, Page 342]\label{Non-stationary:phase}
    Suppose that $a $ is smooth and supported in the unit ball, also let $\phi$ be real-valued so that for some multi-index $\alpha$ with $k=|\alpha|>0,$ we have that 
    $$|\partial_x^\alpha \phi|\geq 1 $$
    throughout  the support of $a.$ Then,
    $$\left|\smallint_{\mathbb{R}^n}e^{i\lambda \phi(\xi)}a(\xi)d\xi\right|\leq C_{k,\phi}\lambda^{-\frac{1}{k}}(\Vert a\Vert_{L^\infty}+\Vert\nabla a\Vert_{L^1}).  $$ The constant $C_{k,\phi}$ is independent of $\lambda>0,$ and of $a,$ and remains bounded as long as the $C^{k+1}$
norm of $a$ remains bounded.
\end{proposition}

\begin{proposition}[\cite{Wolf2003}, Proposition 6.4]\label{Stationar:phase} Let $\phi$ be a smooth function and assume that $\nabla \phi(p)=0.$ Assume also that the Hessian $H_{\phi}(p)=(\partial_{ij}^2\phi(\xi))|_{\xi=p}$ of $\phi$ at $\xi=p,$   is an invertible matrix. Let $\mu$ be the signature of $H_{\phi}(p),$ and let $\Delta=2^{-n}\det(H_{\phi}(p)). $ Let $a:=a(\xi)\in C^\infty_0(\mathbb{R}^n)$ be supported in a small neighbourhood of $\xi=p.$ Then for any $N\in \mathbb{N},$ the integral
$$I(\lambda)=\smallint_{\mathbb{R}^n} e^{i\pi \lambda\phi(\xi)}a(\xi)d\xi,$$ satisfies the asymptotic expansion
    \begin{equation}
        I(\lambda)=e^{\pi i\lambda \phi(p)}e^{-\frac{\mu \pi i }{4}}\Delta^{-\frac{1}{2}}\lambda^{-\frac{n}{2}}\left(a(p)+\sum_{j=1}^N\lambda^{-j}D_{2j}a(p)+O(\lambda^{-(N+1)})\right),
    \end{equation}where $D_{2j}$ are differential operators of order $2j,$ with smooth coefficients depending on $\phi,$ and bounds for finitely many derivatives of $a.$
\end{proposition}

\section{Reduction to phases of strictly positive type}\label{reduction}

 We say that a complex phase function is of {\it strictly positive type}, if $\textnormal{Im}(\Phi(x,\theta))>0,$ when $|\theta|\sim 1.$
 The aim of this section is to show that our main Theorem \ref{Main:theorem} follows from the following reduced version.

\begin{theorem}\label{Main:theorem:2} Let $T$ be a Fourier integral operator of order $-(n-1)/2$ associated with a canonical relation parametrised by a complex phase $\Phi$  of strictly positive type, that is,  there exists $\varkappa_0>0$ such that for all $(x,\theta)\in \overline{V}\times \overline{\{\theta: |\theta|\sim 1\}},$ one has the lower bound $$\textnormal{Im}(\Phi(x,\theta))>\varkappa_0,$$ 
and additionally satisfying the local graph condition \eqref{ass}. Then $T$ is of weak type (1,1).
\end{theorem}
 The proof of Theorem \ref{Main:theorem:2} will be addressed in Section \ref{Section:Main:theorem:2}. Assuming this statement we give a proof of our main theorem.
 \begin{proof}[Proof of Theorem \ref{Main:theorem}] Since $\textnormal{Im}(\Phi(x,\xi))> 0,\,\,|\xi|\neq 0,$ the compactness of the closed annulus $\overline{\{\theta: |\theta|\sim 1\}},$ and of $\overline{V},$ imply that there exists $\varkappa_0>0$ such that for all $(x,\theta)\in \overline{V}\times \overline{\{\theta: |\theta|\sim 1\}},$ one has the lower bound $$\textnormal{Im}(\Phi(x,\theta))>\varkappa_0.$$ In view of Theorem \ref{Main:theorem:2} we end the proof. 
 \end{proof}

\subsubsection*{Notation} We consider the dimension $n\geq 2$ and the parameter $0<\varepsilon\ll 1.$ We will write $2^{C\varepsilon k},$ and the constant $C>0$ will be independent of $\varepsilon,$ and probably it is not the same in different parts of the manuscript. Always we will keep in mind that the constant $C\varepsilon$ can be chosen small enough by allowing $\varepsilon$ to be arbitrarily small.     

Since $n\geq 2,$ we can define the projective system of co-ordinates $(\lambda,\omega),$ via
\begin{equation}\label{Proj:coor:nota}
    \lambda:=\xi_n,\,\omega:=\overline{\xi}/\xi_n,\,\textnormal{ where }\,\xi=(\overline{\xi},\xi_n),\,\xi_n\in \mathbb{R},\,\overline{\xi}\in \mathbb{R}^{n-1}.
\end{equation}In the same way we will denote 
\begin{equation}
    \lambda'=\xi_n',\,\omega'=\overline{\xi'}/\xi'_n,\,\,\xi'=(\overline{\xi}',\xi_n'),\,\xi'_n\in \mathbb{R},\,\overline{\xi'}\in \mathbb{R}^{n-1},
\end{equation}when $\xi\neq \xi'.$

We will use the standard notation for Littlewood-Paley decompositions. So, we fix a  non-negative radial test function $\phi(\xi)=\phi_0(\xi),$ supported on the ball $B(0,2)=\{\xi:|\xi|\leq 2\},$ and for any $k\in \mathbb{R},$ we define the functions
\begin{equation}\label{Littlewood:Paley:components}
    \phi_{k}(\xi):=\phi(\xi/2^k),\,\,\eta_{k}(\xi)=\phi_{k}(\xi)-\phi_{k-1}(\xi).
\end{equation}Even, when these functions are defined on $\mathbb{R}^{d},$ $d=1,n-1$ we keep the same notation. 

We also will denote by $[x,y]:=\{(1-t)x+ty:0\leq t\leq 1\}$ the segment joining two arbitrary points $x,y\in \mathbb{R}^n.$

\section{Proof of Theorem  \ref{Main:theorem:2}}\label{Section:Main:theorem:2}
The aim of this section is to prove the reduced version of our main Theorem \ref{Main:theorem}, namely, the one in Theorem \ref{Main:theorem:2} for phase functions of strictly positive type. To do this we will combine the following two approaches:
\begin{itemize}
    \item The first one, is a strategy traced back to the works of Melin and Sj\"ostrand  \cite{Melin:Sjostrand1975,Melin:Sjostrand1976}, where any Fourier integral operator $T\in I^\mu(X,Y;C)$ can be written in the form
\begin{equation}
  Tf(x)= \smallint\limits_{\mathbb{R}^n}e^{i(\textnormal{Re}(\Phi(x,\xi))+\tau \textnormal{Im}(\Phi(x,\xi) )}a(x,\xi)e^{-(1+i\tau)\textnormal{Im}(\Phi(x,\xi) )}\widehat{f}(\xi)d\xi ,
\end{equation}with the real parameter $\tau\neq0$ as in the local graph condition \eqref{ass}, the real phase function $\Phi_{\tau}$ is given by 
\begin{equation}\label{Phi:tau:rem}
    \Phi_{\tau}(x,\xi):=\textnormal{Re}(\Phi(x,\xi))+\tau \textnormal{Im}(\Phi(x,\xi) ),
\end{equation}
and  the symbol
$$ a_\tau(x,\xi)= a(x,\xi)e^{-(1+i\tau)\textnormal{Im}(\Phi(x,\xi) )}$$
belongs to the class $S^{-(n-1)/2}_{1/2,1/2}.$ This formula for the operator has shown to be effective when dealing with its $H^1$-$L^1$-boundedness, see \cite{Ruzhansky:CWI-book}.
\item The second one, introduced by Tao \cite{Tao} where for the proof of the weak (1,1) inequality one decomposes the operator $T$ into its degenerate and non-degenerate components, $T_{\textnormal{deg}}$ and $T_{\textnormal{non-deg}}.$  Indeed, this is done using the {\it curvature} notion, which measures the extent to which the phase function fails to be linear. The portion where the curvature is small on each dyadic component of the operator defines the degenerate component $T_{\textnormal{deg}}$ of $T$, and its `complement' $T-T_{\textnormal{deg}},$ defines the non-degenerate component $T_{\textnormal{non-deg}}.$ Among other things, the following are the key points in the factorisation approach in  \cite{Tao}: (i) to prove that  $T_{\textnormal{deg}}$ is bounded on $L^1$; (ii) to construct a pseudo-differential operator $S$ of order zero, and then of weak (1,1) type, according to the Calder\'on-Zygmund theory and an average operator $A,$ which is bounded on $L^1,$ and essentially has the same phase function that $T_{\textnormal{non-deg}},$ in such a way that the error operator $E=T_{\textnormal{nondeg}}-SA$ is bounded on $L^1.$
\end{itemize}
The aim of the next subsections is to show that the two previous approaches are compatible, allowing to factorise the operator $T$ with a complex phase satisfying the conditions in Theorem  \ref{Main:theorem:2}, into more manageable operators. 

\subsection{Degenerate and non-degenerate components}\label{Degenerate:non:degenerate:components}
Let us use the notation $\phi_{k}$ in \eqref{Littlewood:Paley:components} for the Littlewood-Paley partition of unity. Typically, $k$ runs over the set of integers, but also we are going to consider the functions $\phi_{-\varepsilon k}$ where $k$ runs over the set of integers and $\varepsilon>0$ is fixed, and then $ \phi_{-\varepsilon k}(\xi):=\phi(2^{\varepsilon k}\xi).$

In terms of the projective co-ordinates $(\lambda,\omega)$ in \eqref{Proj:coor:nota}, in view of the homogeneity of the phase function $\Phi_{\tau}$ as defined in \eqref{Phi:tau:rem}, we can write $\Phi_\tau(x,\xi)=\lambda\Phi_\tau(x,(\omega,1)),$ when $\xi_n=\lambda\neq 0.$ In such a case we are going to simplify the notation by writing $\Phi_\tau(x,\omega):=\Phi_\tau(x,(\omega,1)).$

Using the Littlewood-Paley decomposition for the Fourier integral operator $T,$ we have
\begin{equation*}
    Tf(x)=\sum_{k\gg 1}\smallint\limits_{\mathbb{R}^n}e^{2\pi i\Phi(x,\xi)}a(x,\xi)\eta_{k}(\xi)\widehat{f}(\xi)d\xi,
\end{equation*}where the inequality $k\gg 1,$ with $k\in \mathbb{N},$ is justified because the symbol $a(x,\xi)$ is supported on $|\xi|\gg 1.$ Now, we decompose $T=T_{\textnormal{deg}}+T_{\textnormal{nondeg}}$ with
\begin{equation*}
    T_{\textnormal{deg}} f(x)= \sum_{k\gg 1}\smallint\limits_{\mathbb{R}^n}e^{2\pi i\Phi_\tau(x,\xi)}a(x,\xi)e^{-(1+i\tau)\textnormal{Im}(\Phi(x,\xi))}\phi_{-\varepsilon k}(J_{\tau}(x,\omega))\eta_{k}(\xi)\widehat{f}(\xi)d\xi,
\end{equation*} where, as above,
\begin{equation}
    \Phi_{\tau}(x,\xi):=\textnormal{Re}(\Phi(x,\xi))+\tau\textnormal{Im}(\Phi(x,\xi)),
\end{equation} and 
\begin{equation}
    J_{\tau}(x,\omega)=\textnormal{det}(\nabla_{\omega}^2(\Phi_\tau(x,\omega))),\, \Phi_\tau(x,\omega):=\Phi_{\tau}(x,(\omega,1)),
\end{equation} where $\Phi_{\tau}(x,\xi)=\lambda\Phi_{\tau}(x,(\omega,1)).$ Here, $J_{\tau}(x,\omega)$ is the {\it ``curvature''} associated to the real-valued phase function $\Phi_{\tau}(x,\omega).$ Since $\tau\neq 0,$ sometimes we also write
$$J(x,\omega)=J_{\tau}(x,\omega)$$ for simplicity.
Note then that $T_{\textnormal{nondeg}}$ is given  by
\begin{equation}\label{non:de:c}
    T_{\textnormal{nondeg}} f(x)= \sum_{k\gg 1}\smallint\limits_{\mathbb{R}^n}e^{2\pi i\Phi_\tau(x,\xi)}a(x,\xi)e^{-(1+i\tau)\textnormal{Im}(\Phi(x,\xi))}(1-\phi_{-\varepsilon k}(J_{\tau}(x,\omega)))\eta_{k}(\xi)\widehat{f}(\xi)d\xi.
\end{equation}

\subsubsection{Symbol classes in projective co-ordinates}

The aim of this subsection is to analyse the behaviour of symbols in the frequency variable with respect to the projective co-ordinates $(\lambda,\omega),$ defined via
\begin{equation}
    \lambda=\xi_n,\,\omega=\overline{\xi}/\xi_n,\,\,\xi:=(\overline{\xi},\xi_n),\,\xi_n\in \mathbb{R},\,\overline{\xi}\in \mathbb{R}^{n-1}.
\end{equation}Although $\lambda$ and $\xi_n$ are equal in value, the radial derivative $\partial_\lambda$ keeps $\omega$ fixed, and the vertical derivative $\partial_{\xi_n}$ keeps $\overline{\xi}$ fixed. In the next lemma we analyse the derivatives of a symbol in terms of the radial and of the angular variables. We note that $\omega\in \mathbb{R}^{n-1},$ and we write $\omega^{J}=\omega_1^{J_1}\cdots\omega_{n-1}^{J_{n-1}},$ when $J\in \mathbb{N}_0^{n-1},$ and also $\omega^{T}=(\omega_1,\cdots,\omega_{n-1})$ for a row vector. 
\begin{lemma}\label{Lemma:1:1} Let us assume that the domain of $a$ contains an open neighborhood of the region $\Omega_M=\{(x,\xi)\in T^{*}\mathbb{R}^n:|\xi_n|>M\},$ where $M>0$ is fixed. Assuming $a:=a(x,\xi)$ smooth in $\Omega_M,$ we have that
\begin{equation}\label{lemma:1}
     \partial^\ell_\lambda a=\sum_{|J|\leq \ell }C_{J}\partial_{\overline{\xi}}^{J}\partial_{\xi_n}^{\ell-|J|}a\cdot \omega^{J},
\end{equation}for some family of coefficients $C_{J}$. Also, 
\begin{equation}\label{lemma:2}
   \partial_\omega^{\beta}a=\partial_{\overline{\xi}}^{\beta}a\cdot  \lambda^{|\beta|},
\end{equation}on $\Omega_M.$
    
\end{lemma}
\begin{proof} Let us prove \eqref{lemma:1}. When $\ell=1,$
    the chain rule gets
\begin{align}\label{case:ell:1}
    \partial_\lambda a=\sum_{j=1}^{n-1}\partial_{\xi_j}a\times \partial_{\xi_j}/\partial{\lambda}+\partial_{\xi_n}a\times \partial{\xi_n}/\partial\lambda= \sum_{j=1}^{n-1}\omega_j\partial_{\xi_j}a+\partial_{\xi_n}a=\omega^T\cdot \nabla_{\overline{\xi}}a+\partial_{\xi_{n}}a.
\end{align}When $\ell=2,$  observe that
\begin{align*}
    \partial_\lambda^2 a &=\partial_\lambda(\omega^T\cdot \nabla_{\overline{\xi}}a+\partial_{\xi_{n}}a)=\omega^T\cdot\partial_\lambda \nabla_{\overline{\xi}}a+\omega^{T}\cdot \nabla_{\overline{\xi}}(\partial_{\xi_n}a)+\partial_{\xi_{n}}^2a\\
    &=\omega^{T}\nabla^{2}_{\overline{\xi}}a\cdot \partial{\overline{\xi}}/\partial\lambda+\omega^{T}\cdot \nabla_{\overline{\xi}}(\partial_{\xi_n}a)+\partial_{\xi_{n}}^2a\\
    &=\omega^{T}\cdot[ \nabla^{2}_{\overline{\xi}}a\cdot \omega]+\omega^{T}\cdot \nabla_{\overline{\xi}}(\partial_{\xi_n}a)+ \partial_{\xi_{n}}^2a.
\end{align*} Mathematical induction on $\ell \in \mathbb{N},$
implies that $\partial^\ell_\lambda a$ can be written in terms of some coefficients $C_{J},$ $J\in \mathbb{N}_0^k$ as
\begin{equation}\label{Identity}
  \partial^\ell_\lambda a=\sum_{|J|\leq \ell }C_{J}\partial_{\overline{\xi}}^{J}\partial_{\xi_n}^{\ell-|J|}a\cdot \omega^{J}.
\end{equation} For the proof of \eqref{Identity}
observe that assuming this identity valid for $\ell>2,$  we have
\begin{align*}
    \partial^{\ell+1}_\lambda a=\partial_{\lambda}\partial^{\ell}_\lambda a=\sum_{|J|\leq \ell }C_{J}\partial_{\lambda}(\partial_{\overline{\xi}}^{J}\partial_{\xi_n}^{\ell-|J|}a)\cdot \omega^{J}.
\end{align*}With $b_{J}=\partial_{\overline{\xi}}^{J}\partial_{\xi_n}^{\ell-|J|}a,$ we apply \eqref{case:ell:1} to deduce that
\begin{align*}
  \sum_{|J|\leq \ell }C_{J}\partial_{\lambda}(\partial_{\overline{\xi}}^{J}\partial_{\xi_n}^{\ell-|J|}a)\cdot \omega^{J}&=\sum_{|J|\leq \ell }C_{J}\partial_{\lambda}(b_J)\cdot \omega^{J}  =\sum_{|J|\leq \ell }\left(\sum_{j=1}^{n-1}C_{J}\partial_{\xi_j}b_J\omega_j\omega^{J}+\partial_{\xi_n}b_J\right)\\
  &=\sum_{|J|\leq \ell }\left(\sum_{j=1}^{n-1}C_{J}\partial_{\xi_j}\partial_{\overline{\xi}}^{J}\partial_{\xi_n}^{\ell-|J|}a\omega_j\omega^{J}+\partial_{\xi_n}\partial_{\overline{\xi}}^{J}\partial_{\xi_n}^{\ell-|J|}a\right)\\
  &=\sum_{|J|\leq \ell }\left(\sum_{j=1}^{n-1}C_{J}\partial_{\overline{\xi}}^{J+e_j}\partial_{\xi_n}^{\ell+1-|J+e_j|}a\omega^{J+e_j}+\partial_{\overline{\xi}}^{J}\partial_{\xi_n}^{\ell+1-|J|}a\right)\\
  &=\sum_{|\tilde{J}|\leq \ell+1 }C_{\tilde{J}}\partial_{\overline{\xi}}^{\tilde{J}}\partial_{\xi_n}^{\ell+1-|\tilde{J|}}a\cdot \omega^{\tilde{J}},
\end{align*}for some family of coefficients $C_{\tilde{J}}.$ The proof of \eqref{lemma:1} is complete. For the proof of \eqref{lemma:2} note that for any $j=1,2,\cdots, n-1,$
\begin{align*}
    \partial_{\omega_j}a=\partial_{\xi_j}a\cdot \partial\xi_{j}/\partial \omega_{j}=\partial_{\xi_j}a\cdot \lambda,
\end{align*}which verifies \eqref{lemma:2} when $|\beta|=1.$ Assuming  \eqref{lemma:2} with $\beta$ fixed, namely, that $  \partial_\omega^{\beta}=\partial_{\overline{\xi}}^{\beta}a\cdot  \lambda^{|\beta|},$ note that for any $1\leq j\leq n-1,$
\begin{equation*}
     \partial_\omega^{\beta+e_j} a= \partial_{\omega_j}\partial_{\omega}^{\beta}a= \partial_{\xi_j}(\partial_{\omega}^{\beta}a)\cdot \lambda=\partial_{\xi_j}\partial_{\overline{\xi}}^{\beta}a\cdot  \lambda^{|\beta|}\cdot \lambda=\partial_{\overline{\xi}}^{\beta+e_j}a\cdot  \lambda^{|\beta+e_j|},
\end{equation*}from where the mathematical induction proves $ \eqref{lemma:2} $ for all $\beta\in \mathbb{N}_0^n.$
\end{proof}

For $1/2\leq \rho\leq 1,$ let $S^{m}_{\rho,\mathscr{C}}$ be the class of symbols  that consists of all functions $a:=a(x,\xi)$ which are Borel measurable in $x\in \mathbb{R}^n,$ supported  in the cone bundle
\begin{equation}
    \mathscr{C}=\{(x,\xi)\in T^{*}\mathbb{R}^n:\xi_{n}\gg 1,\,|\overline{\xi}|\leq \lambda\},
\end{equation} and satisfying the H\"ormander type estimates 
\begin{equation}\label{rho:estimates}
    |\partial_\xi^\alpha a(x,\xi)|\leq C_\alpha|\xi|^{m-\rho|\alpha|},\,\,(x,\xi)\in \textnormal{supp}(a).
\end{equation}
\begin{lemma}\label{Lemma:Order:m}   Let $k\gg 1,$ $C>0$ and let $\varepsilon>0.$ Let $a\in S^{m}_{\rho,\mathscr{C}}.$ Consider the symbol 
\begin{equation}
    a_{k}(x,\xi)= a(x,\xi)\phi_{-\varepsilon k}(J_{\tau}(x,\omega))\eta_{k}(\xi),
\end{equation}  supported in the  region
$$\mathscr{C}_\lambda=\{(x,\xi)\in \textnormal{supp}(a): \lambda\sim 2^k; \,|J_{\tau}(x,\xi)|\lesssim 2^{-\varepsilon k}\}.$$ Then, we have the following symbol estimates
\begin{equation}\label{symbol:radial:angular:type:order:m}
    |\partial_\omega^\beta \partial_{\lambda}^\gamma a_{k}(x,\xi) |\lesssim_{\beta,\gamma,C} 2^{mk}2^{-\rho|\gamma|k+(1-\rho)|\beta|k} 2^{C\varepsilon k |\beta|},
\end{equation} for $(x,\xi)$ in the support of $a:=a(x,\xi).$
\end{lemma}
\begin{proof} Using that
 $a$ satisfies the estimates in \eqref{rho:estimates} and the identity \eqref{lemma:1}  we have that
\begin{align*}
 \forall \ell\geq 1,\,\,   |\partial_\lambda^\ell a|\lesssim \sup_{|J|\leq \ell}|\partial_{\overline{\xi}}^{J}\partial_{\xi_n}^{\ell-|J|}a|\lesssim |\xi|^{m-\rho \ell}.
\end{align*}Observing that, on the region $\mathscr{C}_\lambda,$ $|\xi|\sim |\lambda|\sim 2^{k}$ we have proved that
\begin{equation}
    |\partial_\lambda^\ell a|\lesssim 2^{k(m-\rho \ell) }.
\end{equation}
By the Leibniz rule we have that
\begin{align}\label{Leibniz}
    \partial_\omega^\beta a_k=\sum_{\beta_1+\beta_2+\beta_3=\beta}C_{\beta_1,\beta_2,\beta_2}  \partial_\omega^{\beta_1} a\times \partial_\omega^{\beta_2}(\phi_{-\varepsilon k}(J_{\tau}(x,\omega)))\times \partial_\omega^{\beta_3} \eta_k.
\end{align}In view of \eqref{lemma:2},
we can estimate
\begin{equation}
    |\partial_\omega^{\beta_1} a|=|\partial_{\overline{\xi}}^{\beta_1} a||\lambda|^{|\beta_1|}\lesssim |\xi|^{m-\rho|\beta_1|}|\lambda|^{|\beta_1|}\sim  |\lambda|^{m-\rho|\beta_1|}|\lambda|^{|\beta_1|}\sim 2^{km}2^{k(1-\rho)|\beta_1|}
\end{equation}Since the sequence $\eta_k$ belongs uniformly  in $k$ to $S^0,$ we have that 
\begin{equation}
    |\partial_\omega^{\beta_3} \eta_k|=|\partial_{\overline{\xi}}^{\beta_3} \eta_k||\lambda|^{|\beta_3|}\lesssim |\xi|^{-|\beta_3|}|\lambda|^{|\beta_3|}\sim 1.
\end{equation} 

Now, let us analyse the behaviour of the cutt-off $\psi_{-\varepsilon k}$ at the ``curvature'' term $J_\tau(x,\omega).$ Observe that for any $j=1,\cdots, n-1,$
\begin{align*}
\partial_{\omega_j}(\phi_{-\varepsilon k}(J_\tau(x,\omega)))&=\phi_{-\varepsilon k}'(J_\tau(x,\omega)))\times \partial_{\omega_j}J_\tau(x,\omega).
\end{align*} Then, we have the equalities
\begin{align*}
    |\partial_{\omega_j}(\phi_{-\varepsilon k}(J_{\tau}(x,\omega)))| &=|\phi'_{-\varepsilon k}(t)|_{t=J_\tau(x,\omega)}\times \partial_{\omega_j}J_\tau(x,\omega)|\\
    &=|\partial_{t}[\phi(t/2^{-\varepsilon k})]|_{t=J(x,\omega)}\times \partial_{\omega_j}J(x,\omega)|\\
    &=|(\partial_{t}\phi)(t/2^{-\varepsilon k})2^{k\varepsilon}|_{t=J_\tau(x,\omega)}\times \partial_{\omega_j}J_\tau(x,\omega)|.    
\end{align*}There is $C_\phi>0$ such that
\begin{equation*}
    |\partial_{t}\phi(\lambda)|\leq C_{\phi}.
\end{equation*}Consequently, 
\begin{align*}
    |(\partial_{t}\phi)(t/2^{-\varepsilon k})|_{t=J_\tau(x,\omega)}2^{k\varepsilon}\lesssim  2^{k\varepsilon }.
\end{align*}Thus, the mathematical induction gives the estimate
\begin{equation}
    |\partial_\omega^{\beta_2}(\phi_{-\varepsilon k}(J_{\tau}(x,\omega)))|\lesssim 2^{k\varepsilon|\beta_2|}.
\end{equation}We have proved the estimates
$$  |\partial_\lambda^\gamma a|\lesssim 2^{km}2^{-k\gamma} ,\,  |\partial_\omega^{\beta_1} a|\lesssim 2^{km}2^{k(1-\rho)|\beta_1|},$$
and 
\begin{equation}\label{deribatives:J:lemma}
    |\partial_\omega^{\beta_3} \eta_k|\lesssim 1,\, |\partial_\omega^{\beta_2}(\phi_{-\varepsilon k}(J_{\tau}(x,\omega)))|\lesssim 2^{k\varepsilon|\beta|}.
\end{equation}
Note that when $\gamma=\beta_1=0,$ we recover the information that $a$ has order $m.$ Indeed,  let $a_{\beta_1}:= \partial_\omega^{\beta_1} a.$ Since $\lambda\sim 2^{k},$ note that the order of $a_{\beta_1}$ with respect to $\lambda$ is $m_{\beta_1}:=m+(1-\rho)|\beta_1|.$ If we replace the analysis above with the symbol  $a_{\beta_1}$ instead of $a,$  then $m_{\beta_1}$ takes the place of $m,$  and we get the estimate
$$  |\partial_\lambda^\gamma a_{\beta_1}|\lesssim 2^{km_{\beta_1}}2^{-k\gamma} =2^{km+k(1-\rho)|\beta_1|-k\gamma}. $$
Therefore we have proved that 
$$  |\partial_\lambda^\gamma \partial_\omega^{\beta_1} a|\lesssim 2^{km_{\beta_1}}2^{-k\gamma} =2^{km+k(1-\rho)|\beta_1|-k\gamma}. $$
The estimates for the derivatives of $a$ in $\lambda$ and in $\omega$ above, the inequality in \eqref{deribatives:J:lemma} and the Leibniz rule \eqref{Leibniz} conclude the proof of \eqref{symbol:radial:angular:type:order:m}.
\end{proof}

\subsection{Boundedness of the degenerate component}\label{Boundedness of the degenerate component}
In this section we prove that the  operator 

\begin{equation}\label{deg:section:l1}
    T_{\textnormal{deg}} f(x)= \sum_{k\gg 1}\smallint\limits_{\mathbb{R}^n}e^{2\pi i\Phi_\tau(x,\xi)}a(x,\xi)e^{-(1+i\tau)\textnormal{Im}(\Phi(x,\xi))}\phi_{-\varepsilon k}(J_{\tau}(x,\omega))\eta_{k}(\xi)\widehat{f}(\xi)d\xi,
\end{equation} is bounded on $L^1(\mathbb{R}^n).$ 
Before continuing with the proof of the $L^1$-boundedness of $T_{\textnormal{deg}} ,$ we discuss the second dyadic partition that will be employed in our further analysis.
\begin{remark}[About discrete vs continuous `second' dyadic decompositions]
In \cite{SSS}, a discrete `second dyadic decomposition' was introduced where the variable $\omega$ is decomposed in a finite family of disks $\{D\}_{D\in \mathcal{I}},$ see Figure \ref{Fig1}. At first glance,  for the analysis of the degenerate component  $T_{\textnormal{deg}},$ one could apply this discrete `second dyadic partition' on the symbol:
$$  a_{k}(x,\xi)= a(x,\xi)\phi_{-\varepsilon k}(J_{\tau}(x,\omega))\eta_{k}(\xi),$$ generating the new family of symbols $a_{k,D},$ in such a way that $a_k=\sum_{D}a_{k,D}.$ However, as it was explained in \cite{Tao}, this discrete decomposition provides the estimate
\begin{equation}\label{Remark:D}
    \left\Vert \smallint\limits_{\mathbb{R}^n}e^{2\pi i\Phi_\tau(x,\xi)}a_k(x,\xi)e^{-2\pi i y\cdot \xi}   d\xi \right\Vert_{L^1(\mathbb{R}^n_x)}\lesssim 1,
\end{equation} for the  $L^1$-norm of the kernel associated with each dyadic part $a_k.$ This estimate is not enough to deduce the $L^1$ boundeness of $T_{\textnormal{deg}}.$ The reason is that one needs a more accurate decay of this $L^1$-norm. Indeed, this is proved in Lemma \ref{Fundamental:deg:lemma} improving the $O(1)$ term in  \eqref{Remark:D}, by the smaller one $2^{-\varepsilon k}.$ In order to make this improvement, we will use the degeneracy condition $|J(x,\omega)|\lesssim 2^{-\varepsilon k}.$ To have the control of lower terms in the Taylor expansion of the phase $\Phi_\tau-y\cdot \xi$ at $\omega_0=\omega_D,$ especially if $\omega-\omega_D$ is in the direction where $\nabla_\omega^2\Phi_\tau$ degenerates, the disks $D$ will be replaced by a suitable family of ellipsoids, that allows us to decompose the $(x,\xi)$-variable in a family of tubular regions where we gain the decay $2^{-\varepsilon k}.$ 

\begin{figure}[h]
\includegraphics[width=8cm]{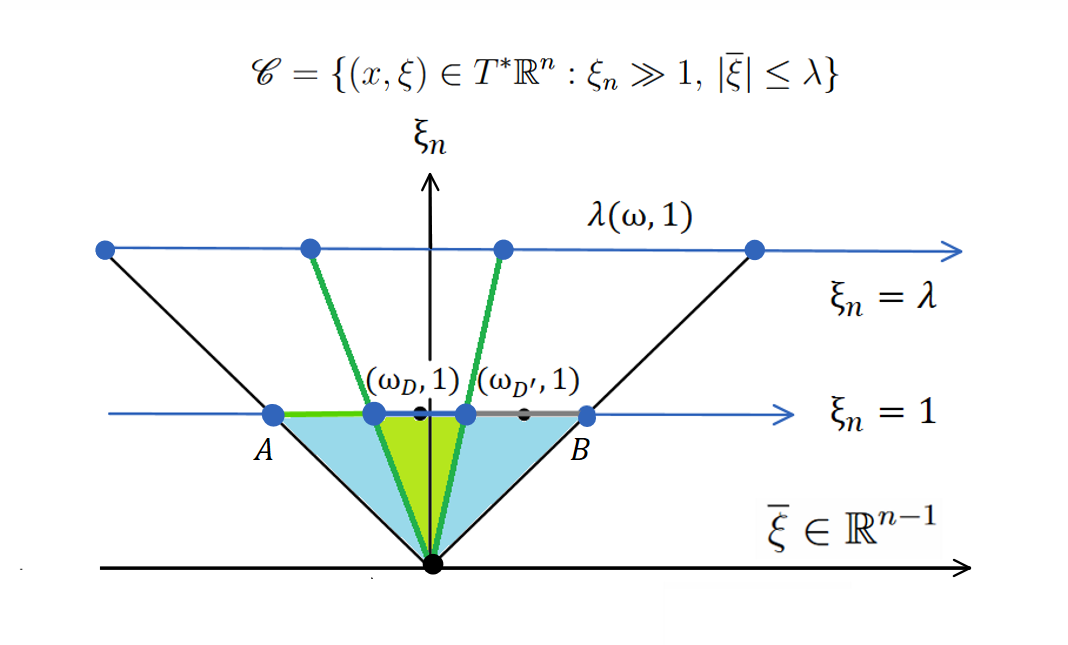}\includegraphics[width=8cm]{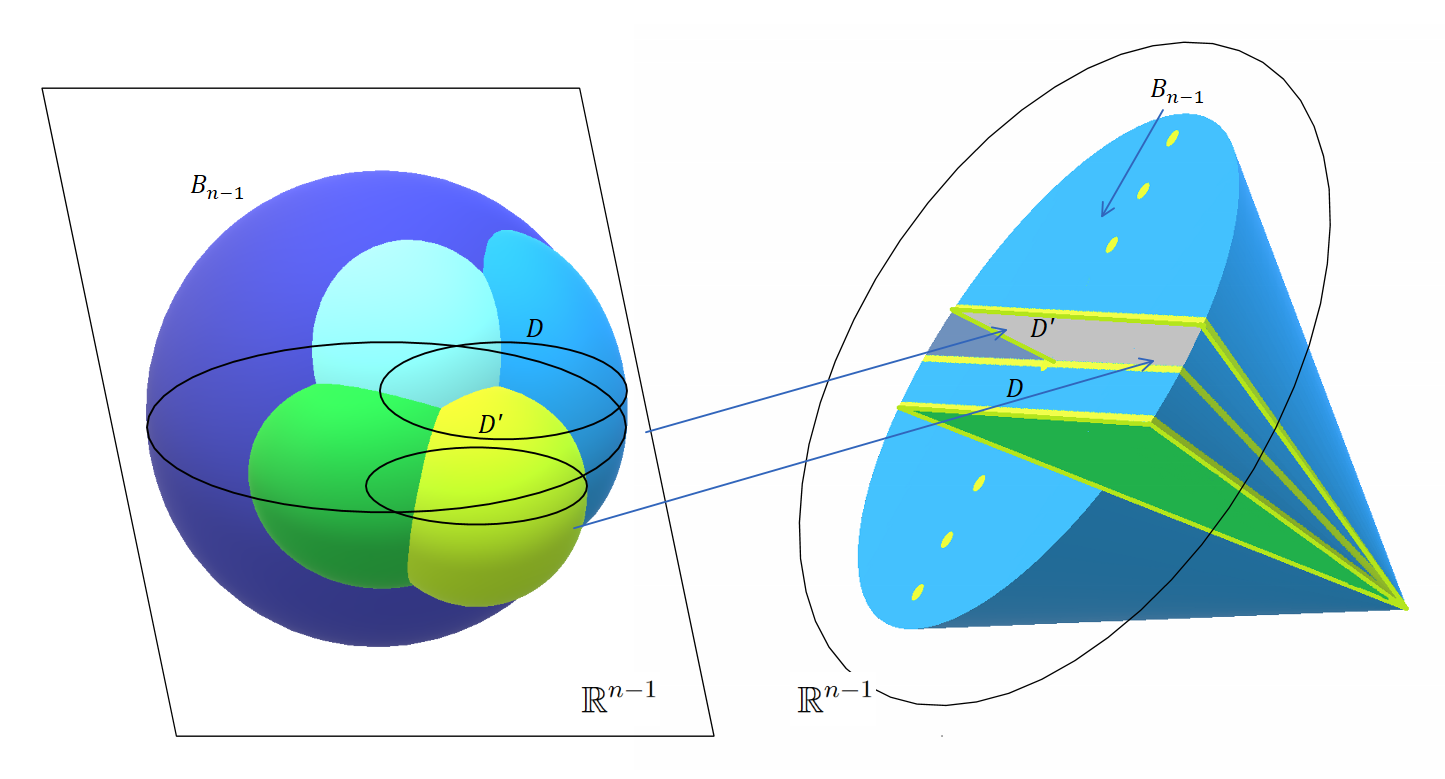}\\\includegraphics[width=11cm]{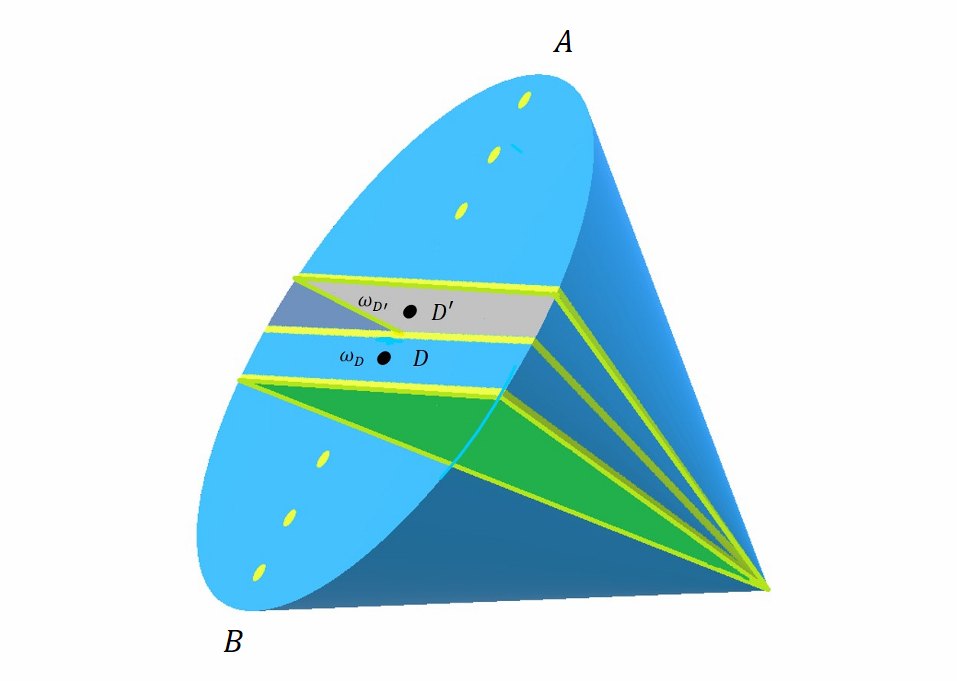}
\caption{ Here we illustrate the partition of the $\omega$ variable smoothly into about $2^{(n-1)k/2}$ disks $D$ is radius $2^{-k/2}$ when $|\omega|\lesssim 1.$  First, in the system of coordinates $(\overline{\xi},\xi_n),$ the projective coordinates $(\lambda,\omega)$ determine each point $(\overline{\xi},\xi_n)\in \mathscr{C}.$ The unit ball $B_{n-1}=\{|\omega|\leq 1\}$ on $\mathbb{R}^{n-1}$ is saturated to the `segment' $\overline{AB}.$ Each partition of the ball $B_{n-1}$ in a family of disks $\{D\}_{D\in \mathcal{I}}$ is saturated to a partition of the segment $\overline{AB}.$ In particular if the radii of the disks are proportional to $r=2^{-k/2},$ one can estimate $|\mathcal{I}|\sim 2^{(n-1)k/2}.$ }
 \label{Fig1}
\centering
\end{figure}
As it was pointed out in \cite{Tao}, there is an apparent difficulty in the case when $n\geq 3,$ where the ellipsoid around $\omega_0=\omega_D,$ depends on the spectrum of $\nabla_\omega^2\Phi_\tau(x,\omega_D).$ Then, the eccentricity and the orientation of this ellipsoid vary in function of $\omega_D.$ However, by adapting to the setting of the complex phases, the approach in \cite{Tao}, we will define a continuous `second dyadic partition' in $\omega,$ using a family of functions  $\psi_{x,\omega_D}$ with a family of corresponding supports determined by a  family of ellipsoids with a smooth variation of eccentricity and orientation, centred at the points $\omega_D$  that avoid Kakeya-type covering lemmas. According to this partition the part of the support of each symbol $a_{k,D}$ on the hyperplane $\xi_n=\lambda$ is illustrated in Figure \ref{Fig2}.
\begin{figure}[h]
\includegraphics[width=11cm]{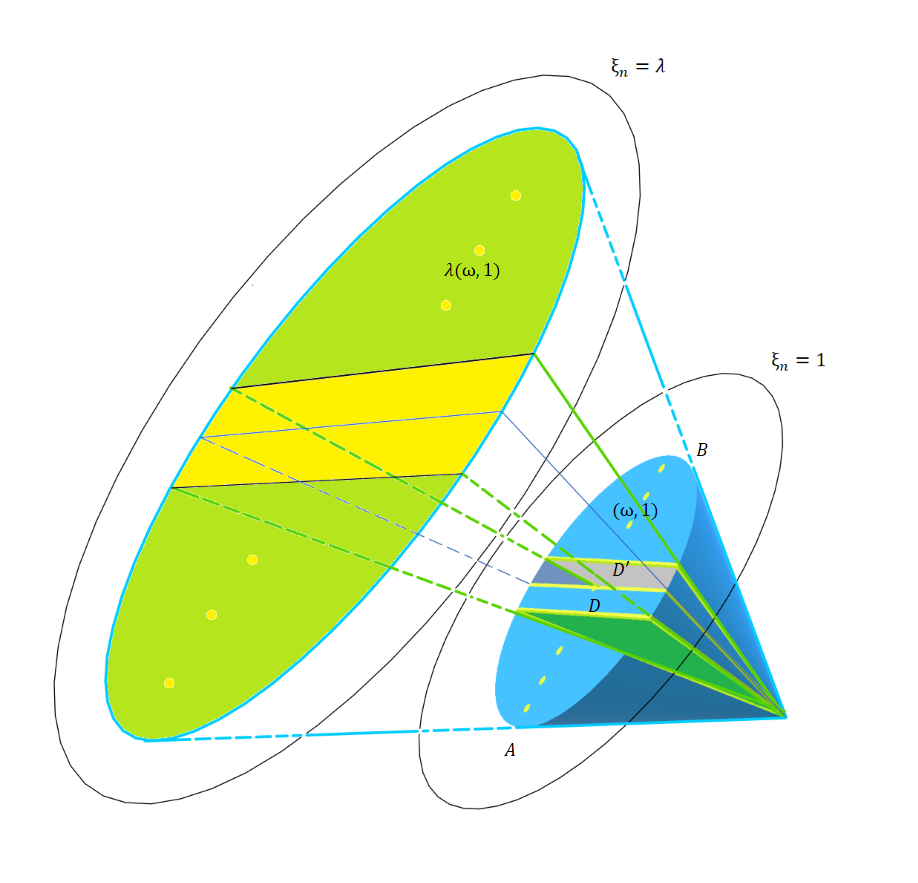}
\caption{Note that the dyadic partition of the variable $\xi_n=\lambda$ together with the family of `ellipsoids' $\{D\}_{D\in \mathcal{I}}$ decomposing the variable $\omega,$ provide a partition of the cone bundle $\mathscr{C}.$}
 \label{Fig2}
\centering
\end{figure} 
\end{remark}
The following is the main result of this subsection.
\begin{proposition}\label{Proposition:deg} Let $S^{-(n-1)/2}_{\rho,\mathscr{C}}$ be the symbol of the Fourier integral operator $T$ of order $m=-(n-1)/2$ and let $\Phi$ be its complex phase function.  Then, the degenerate component $T_{\textnormal{deg}} $ in \eqref{deg:section:l1} extends to a bounded operator from $L^1(\mathbb{R}^n)$ into $L^1(\mathbb{R}^n).$    
\end{proposition}
\begin{remark}\label{remark:final:part:proof}
    For the proof of Proposition \ref{Proposition:deg}, 
it is suffices to prove the estimate
\begin{equation}\label{L1:B:Deg:Part}
    \exists C>0,\forall f\in C^\infty_0(\mathbb{R}^n),\, \Vert T_{\textnormal{deg}} f \Vert_{L^1(\mathbb{R}^n)}\leq C\Vert f\Vert_{L^1(\mathbb{R}^n)}.
\end{equation}Note that the kernel of $T_{\textnormal{deg}}$ is given by
\begin{equation*}
    K_{\textnormal{deg}}(x,y):=\sum_{k\gg 1}\smallint\limits_{\mathbb{R}^n}e^{2\pi i\Phi_\tau(x,\xi)}a(x,\xi)e^{-(1+i\tau)\textnormal{Im}(\Phi(x,\xi))}\phi_{-\varepsilon k}(J_{\tau}(x,\omega))\eta_{k}(\xi)e^{-2\pi i y\cdot \xi}d\xi.
\end{equation*}By the triangle inequality, for the proof of \eqref{L1:B:Deg:Part}, it suffices to prove that  
\begin{equation}\label{First:Reduction}
    \left\Vert \smallint\limits_{\mathbb{R}^n}e^{2\pi i\Phi_\tau(x,\xi)}a(x,\xi)e^{-(1+i\tau)\textnormal{Im}(\Phi(x,\xi))}\phi_{-\varepsilon k}(J_{\tau}(x,\omega))\eta_{k}(\xi)e^{-2\pi i y\cdot \xi}   d\xi \right\Vert_{L^1(\mathbb{R}^n_x)}\lesssim 2^{-\varepsilon k},
\end{equation} in view of Minkowski integral inequality. Indeed, from \eqref{First:Reduction} one deduces that
\begin{equation}\label{The:constant:C}
  C:=\sup_{y\in \mathbb{R}^n}\Vert  K_{\textnormal{deg}}(x,y) \Vert_{L^1(\mathbb{R}^n_x)}\lesssim\sum_{k\gg 1}2^{-\varepsilon k}<\infty.   
\end{equation}
Consequently,
$$
 \Vert T_{\textnormal{deg}} f \Vert_{L^1(\mathbb{R}^n)}=   \Vert \smallint K_{\textnormal{deg}}(x,y)f(y)dy\Vert_{L^1(\mathbb{R}^n_x)}=\smallint\left|\smallint K_{\textnormal{deg}}(x,y)f(y)dy\right|dx$$
    $$\leq \smallint\smallint |K_{\textnormal{deg}}(x,y)|dx|f(y)|dy\leq \sup_{y\in \mathbb{R}^n} \smallint |K_{\textnormal{deg}}(x,y)|dx\smallint|f(y)|dy $$
     $$ = \sup_{y\in \mathbb{R}^n} \smallint |K_{\textnormal{deg}}(x,y)|dx\|f\|_{L^1(\mathbb{R}^n)},$$ from where one deduces \eqref{L1:B:Deg:Part} with $C>0$ defined in \eqref{The:constant:C}.
\end{remark}
For the proof of Proposition \ref{Proposition:deg} we require some preliminary results. We start with the following estimates for radial and angular derivatives of symbols in the class $S^{-(n-1)/2}_{\rho,\mathscr{C}}$ and with an estimate of the $L^1$-norm of the kernel of $T_{\textnormal{deg}}.$ This is our first step in the proof of \eqref{First:Reduction}.

\begin{lemma}\label{Fundamental:deg:lemma} Let $k\gg 1,$ $\frac{1}{2}\leq \rho\leq 1,$ $C>0,$ and let $\varepsilon>0.$ Let $a\in S^{-(n-1)/2}_{\rho,\mathscr{C}}.$ Consider the symbol 
\begin{equation}\label{symbol:a:k}
    a_{k}(x,\xi)= a(x,\xi)\phi_{-\varepsilon k}(J_{\tau}(x,\omega))\eta_{k}(\xi).
\end{equation} Then, we have the following symbol estimates
\begin{equation}\label{symbol:radial:angular:type}
    |\partial_\omega^\beta \partial_{\lambda}^\gamma a_{k}(x,\xi) |\lesssim_{\beta,\gamma} 2^{-\frac{(n-1)k}{2}}2^{-\rho|\gamma|k+(1-\rho)|\beta|k} 2^{C\varepsilon k |\beta|},
\end{equation} for $(x,\xi)$ in the support of $a.$ Moreover, any function $a_k:=a_{k}(x,\xi)\in S^{-(n-1)/2}_{\rho,\mathscr{C}}$ supported in the region
\begin{equation}\label{region:support:a:k}
  \{(x,\xi)\in \textnormal{supp}(a): \lambda\sim 2^k; \,|J_{\tau}(x,\xi)|\lesssim 2^{-\varepsilon k}\}  
\end{equation}
satisfying the symbol estimates in \eqref{symbol:radial:angular:type} defines a Fourier integral operator whose  kernel satisfies the inequality
\begin{equation}\label{First:Reduction:lemma}
    \left\Vert \smallint\limits_{\mathbb{R}^n}e^{2\pi i\Phi_\tau(x,\xi)}a_k(x,\xi)e^{-2\pi i y\cdot \xi}  d\xi  \right\Vert_{L^1(\mathbb{R}^n_x)}\lesssim_\rho 2^{-\varepsilon k},
\end{equation} uniformly in $y\in \mathbb{R}^n.$
    
\end{lemma}
\begin{proof} The estimate in \eqref{symbol:radial:angular:type} for $a_k$ as defined in \eqref{symbol:a:k}  is a consequence of Lemma \ref{Lemma:Order:m} applied to the order $m=-(n-1)/2.$

Now, let us consider a symbol $a_{k}:=a_k(x,\xi)$ satisfying the estimates in \eqref{symbol:radial:angular:type} whose support is contained in \eqref{region:support:a:k}. For the proof of \eqref{First:Reduction:lemma} we need a preliminary construction. Indeed, let us define a positive-definite replacement $Q(x,\omega)$ of $\nabla_\omega^2\Phi_\tau(x,\omega).$ Note that  $\nabla_\omega^2\Phi_\tau(x,\omega)$   is the associated  matrix to the quadratic form
$$\zeta\mapsto \zeta^{T}\nabla_\omega^2\Phi_\tau(x,\omega)\zeta=:\nabla_\omega^2\Phi_\tau(x,\omega)(\zeta,\zeta).$$
 Let $$\lambda_{1}(x,\omega),\cdots, \lambda_{n-1}(x,\omega)$$ be the family of eigenvalues of $\nabla_\omega^2\Phi_\tau(x,\omega).$
In an appropriate bases $B$ of $\mathbb{R}^{n-1}$ we can write
\begin{align*}
    \nabla_\omega^2\Phi_\tau(x,\omega)(\zeta,\zeta)=\sum_{j=1}^{n-1}\lambda_{j}(x,\omega)|\zeta_j|^{2},\,\,\zeta\in \mathbb{R}^{n-1}.
\end{align*}
Define
\begin{equation}
    Q(x,\omega):=(2^{-2\varepsilon k}I+( \nabla_\omega^2\Phi_\tau(x,\omega))^2)^{\frac{1}{2}},
\end{equation}where $I=I_{n-1}$ is the identity matrix of size $n-1.$ 
In the same basis $B,$ we can write
\begin{align*}
    Q(x,\omega)=\textnormal{diag}[2^{-2\varepsilon k}+\lambda_j(x,\omega)^2]_{1\leq j\leq n-1}^{\frac{1}{2}}.
\end{align*}
Observe that
\begin{equation}
    \det(Q(x,\omega))=\prod_{j=1}^{n-1}[2^{-2\varepsilon k}+\lambda_j(x,\omega)^2]^{\frac{1}{2}}\geq 2^{-Ck\varepsilon}, \,C=n-1.
\end{equation}On the other hand
\begin{align*}
     \det(Q(x,\omega)) &=\prod_{j=1}^{n-1}[2^{-2\varepsilon k}+\lambda_j(x,\omega)^2]^{\frac{1}{2}}\leq \prod_{j=1}^{n-1}[2^{- \varepsilon k}+|\lambda_j(x,\omega)|]\\
     &=2^{-(n-1)\varepsilon k}+\prod_{j=1}^{n-1}|\lambda_j(x,\omega)|+O(2^{-\varepsilon k})\\
     &\lesssim 2^{-\varepsilon k}+|J_\tau(x,\omega)|.
\end{align*}The matrix $Q(x,\omega)$ has determinant satisfying the estimates
\begin{align}\label{lower:bounde:det:Q}
    2^{-Ck\varepsilon}\leq \det(Q(x,\omega))\lesssim  2^{-\varepsilon k}+|J_\tau(x,\omega)|.
\end{align}
For any $(x,\omega_D)\in \mathbb{R}^{n}\times\mathbb{R}^{n-1},$ with the property that when $\lambda\sim 2^{k},$ 
$$(x,\lambda(\omega_D,1))=(x,\overline{\xi}_D,\lambda)\in \mathscr{C}, \,\,\overline{\xi}_D:=\omega_D\lambda ,$$
 define the function $\psi_{x,\omega_D}$ by
\begin{align}\label{phi:x:omega:D}
    \psi_{x,\omega_D}(\omega):=\det(Q(x,\omega_D))^{\frac{1}{2}}{ 2^{\frac{(n-1)k}{2}}\psi_{-k}(Q(x,\omega_D)(\omega-\omega_D,\omega-\omega_D))   }.
\end{align} Since on $\mathscr{C},$ $|\overline{\xi}_D|\lesssim \lambda$ we necessarily have that $|\omega_D|\lesssim 1$ and so $\omega_D$ is contained in a compact subset $\Omega$ of $\mathbb{R}^{n-1}.$ 
Note that
\begin{align*}
 Q(x,\omega_D)(\omega-\omega_D,\omega-\omega_D)\leq \sup_{x\in \pi_{1}[\textnormal{supp}(a)],\,\omega'\in \Omega }\Vert Q(x,\omega')\Vert_{\textnormal{op}}\Vert \omega-\omega_D\Vert^{2}\lesssim_{\textnormal{supp}(a),\Omega}   \Vert \omega-\omega_D\Vert^{2}.
\end{align*}In consequence, if 
$$\Vert \omega-\omega_D\Vert\lesssim 2^{-(1-\frac{\varepsilon}{2})k},$$ then 
$$Q(x,\omega)(\omega-\omega_D,\omega-\omega_D)\leq 2^{-(2-\varepsilon)k}\lesssim 2^{-k},$$ and consequently $\textnormal{supp}(\psi_{x,\omega_D})$ contains a small disk $\{\omega=\omega_D+z:|z|\lesssim 2^{-(1-\frac{\varepsilon}{2})k} \},$ namely, we have the inclusion
\begin{align}\label{A1}
    \{\omega=\omega_D+z:|z|\lesssim 2^{-(1-\frac{\varepsilon}{2})k} \}\subset \textnormal{supp}(\psi_{x,\omega_D}).
\end{align}
 By writing $Q(x,\omega_D)$ in the basis $B,$ we have
\begin{align*}
    Q(x,\omega_D)(\omega-\omega_D,\omega-\omega_D)=\sum_{j=1}^{n-1}(2^{-2\varepsilon k}+\lambda_j(x,\omega_D)^{2})^{\frac{1}{2}}(\omega_j-\omega_{D,j})^{2}
\end{align*} from where one can deduce that $\psi_{x,\omega_D}(\omega)$ is supported in an ellipsoid centred at $\omega=\omega_D.$ 
Observe that, since 
\begin{align}\label{aux:disc:quadratic:form}
    Q(x,\omega_D)(\omega-\omega_D,\omega-\omega_D)\geq \sum_{j=1}^{n-1}2^{-\varepsilon k}(\omega_j-\omega_{D,j})^{2},
\end{align} the support of $\psi_{x,\omega_D}$ contains the disk $\{\omega=\omega_D+O(2^{- k/2})\}.$ Indeed, on the support of $\psi_{x,\omega_D},$  
 $Q(x,\omega_D)(\omega-\omega_D,\omega-\omega_D)\in \textnormal{supp}(\psi_{-k})$ which says that
$$ Q(x,\omega_D)(\omega-\omega_D,\omega-\omega_D)\lesssim 2^{-k}.$$ In view of \eqref{aux:disc:quadratic:form} the inequality
$\sum_{j=1}^{n-1}2^{-\varepsilon k}(\omega_j-\omega_{D,j})^{2}\lesssim  2^{-k} $
implies that 
$$\sum_{j=1}^{n-1}(\omega_j-\omega_{D,j})^{2}\lesssim  2^{-k} 2^{\varepsilon k}= 2^{-(1-\varepsilon)k}. $$ 
In consequence, 
\begin{align}\label{A2}
\textnormal{supp}(\psi_{x,\omega_D})\subset \{\omega=\omega_D+z:|z|\lesssim2^{-(1-\varepsilon)k}\}.
\end{align} In view of the inclusions \eqref{A1}   and \eqref{A2} for any $\omega\in \textnormal{supp}(\psi_{x,\omega_D})$ we have
\begin{align}
  2^{-(1-\frac{\varepsilon}{2})k }\lesssim  |\omega-\omega_D|\lesssim 2^{-(1-\varepsilon)k } .
\end{align} Now, define the averaged  function $\psi_{x}(\omega)$ by
\begin{equation}
    \psi_{x}(\omega)=\smallint_{\omega_D\in \mathbb{R}^{n-1}}\psi_{x,\omega_D}(\omega)d\omega_{D}.
\end{equation}

We review the properties of the averaged function $\psi_{x}$ in the following lemma.
\begin{lemma}\label{function:psi:x}
    The function $\psi_x$ is positive, and the following estimates are valid:
    \begin{itemize}
        \item $\psi_{x}(\omega)\sim 1$
        \item  $\forall\beta,\,\,|\partial^\beta_\omega \psi_{x}(\omega) |\lesssim_{\beta}2^{C\varepsilon k|\beta|}. $
    \end{itemize}
\end{lemma} \begin{proof}In a suitable basis we can write the quadratic form $Q(x,\omega_D)$ as follows
 \begin{align*}
    Q(x,\omega)(\zeta,\zeta)=\sum_{j=1}^{n-1}\lambda_{Q,j}^{\omega}|\zeta_j|^2,\,\lambda_{Q,j}^{\omega}:=(2^{-2\varepsilon k}+\lambda_j(x,\omega)^{2})^{\frac{1}{2}}.
\end{align*}
Observe that
\begin{align*}
    &\psi_{x}(\omega) \\
    &=\smallint_{\omega_D\in \mathbb{R}^{n-1}}\psi_{x,\omega_D}(\omega)d\omega_{D}=\smallint_{\textnormal{supp}(\psi_{x,\omega_D})}\psi_{x,\omega_D}(\omega)d\omega_{D}\\
    &\sim \smallint_{ \omega=\omega_D+z:  2^{-(1-\frac{\varepsilon}{2})k }\lesssim  |z|\lesssim 2^{-(1-\varepsilon)k }   }\psi_{x,\omega_D}(\omega)d\omega_{D}\\
    &= \smallint_{ \omega=\omega_D+z: 2^{-(1-\frac{\varepsilon}{2})k }\lesssim  |z|\lesssim 2^{-(1-\varepsilon)k }    }{ 2^{\frac{(n-1)k}{2}}\psi_{-k}(Q(x,\omega_D)(\omega-\omega_D,\omega-\omega_D))   }{\det(Q(x,\omega_D))^{\frac{1}{2}}}d\omega_{D}\\
    &= \smallint_{z:\, 2^{-(1-\frac{\varepsilon}{2})k }\lesssim  |z|\lesssim 2^{-(1-\varepsilon)k } }{ 2^{\frac{(n-1)k}{2}}\psi_{-k}(Q(x,\omega-z)(z,z))   }{\det(Q(x,\omega-z))^{\frac{1}{2}}}dz.
  \end{align*}
 By making Taylor expansion of the function $\psi_{-k}(Q(x,\omega-z')(z,z))|_{z'=z}$ at $z'=0$ we have that
 \begin{align*}
         \psi_{x}(\omega)   &= \smallint_{ z:\,2^{-(1-\frac{\varepsilon}{2})k }\lesssim  |z|\lesssim 2^{-(1-\varepsilon)k } }{ 2^{\frac{(n-1)k}{2}}\psi_{-k}(Q(x,\omega-z)(z,z))   }\det(Q(x,\omega-z))^{\frac{1}{2}}dz.\\            
         &=\smallint_{ z:\,2^{-(1-\frac{\varepsilon}{2})k }\lesssim  |z|\lesssim 2^{-(1-\varepsilon)k } }{ 2^{\frac{(n-1)k}{2}}\psi_{-k}( Q(x,\omega)(z,z))   }{\det(Q(x,\omega))^{\frac{1}{2}}}dz+O(2^{C\varepsilon k}2^{-\varepsilon k})\\
          &=\smallint_{ z:\,2^{-(1-\frac{\varepsilon}{2})k }\lesssim  |z|\lesssim 2^{-(1-\varepsilon)k } }
         { 2^{\frac{(n-1)k}{2}}\psi_{-k}\left(  \sum_{j=1}^{n-1}\lambda_{Q,j}^\omega|z_j|^2  \right)   }{\det(Q(x,\omega_D))^{\frac{1}{2}}}dz+O(2^{C\varepsilon k}2^{-\varepsilon k}).
\end{align*}The change of variables $z=(\lambda_{j,Q}^{-\frac{1}{2}}\zeta_j)_{1\leq j\leq n-1}$ has the new volume element
\begin{equation}
    dz=\prod_{j=1}^{n-1}(\lambda_{j,Q}^\omega)^{-\frac{1}{2}}\cdot d\zeta=\det(Q(x,\omega))^{-\frac{1}{2}}d\zeta
\end{equation} from where we have that
\begin{align*}
     \psi_{x}(\omega)   &\sim \smallint
         { 2^{\frac{(n-1)k}{2}}\psi_{-k} ((\zeta,\zeta))  }d\zeta=\smallint
         { 2^{\frac{(n-1)k}{2}}\psi(|\zeta|^{2}/2^{-k})  }d\zeta=\smallint\psi_0(|\overline{\xi}|^2)d\overline{\xi}\sim 1.
\end{align*}A similar analysis can be done for the proof of the estimate $|\partial^\beta_\omega \psi_{x}(\omega) |\lesssim_{\beta}2^{C\varepsilon k|\beta|}. $ On the other hand,  let $|\beta|\geq 1.$  Note that when applying derivatives $\partial_{\omega}^{\beta}$ to $\psi_x(\omega)$ the main term $\smallint\psi_0(|\overline{\xi}|^2)d\overline{\xi}$ vanishes  and the error term contribute in $O(2^{C\varepsilon k|\beta|}).$ 
\end{proof}

In order to prove \eqref{First:Reduction:lemma}
we use the continuous partition of unity taking averages with respect to ellipsoids centered at $\omega_D$ as follows (and using Fubini's theorem)
\begin{align*}
    \smallint\limits_{\mathbb{R}^n}e^{2\pi i(\Phi_\tau(x,\xi)-y\cdot \xi)}a_k(x,\xi)d\xi &= \smallint\limits_{\mathbb{R}^n}\smallint\limits_{\Omega\subset \mathbb{R}^{n-1} }e^{2\pi i(\Phi_\tau(x,\xi)-y\cdot \xi)} \frac{ a_k(x,\xi) \psi_{x,\omega_D}(\omega)  }{\psi_{x}(\omega)} d\omega_D d\xi\\
    &=\smallint\limits_{\Omega\subset \mathbb{R}^{n-1} } \smallint\limits_{\mathbb{R}^n} e^{2\pi i(\Phi_\tau(x,\xi)-y\cdot \xi)} \frac{ a_k(x,\xi) \psi_{x,\omega_D}(\omega)  }{\psi_{x}(\omega)}  d\xi d\omega_D.
\end{align*}The variable $\omega_D$ runs over a compact set $\Omega,$ when $(x,\xi)\in \mathscr{C}$ as discussed above. Since $\lambda\sim 2^k,$ on the support of $a_{k},$ Figure \ref{Fig2}  illustrates the points of the support of $a_k$ on the hyperplane $\xi_n=\lambda.$

Hence, for the proof of \eqref{First:Reduction:lemma} is enough to show that 
 \begin{align}\label{pre:final:reduction}
     \sup_{\omega_D\in  \Omega}\left\Vert   \smallint\limits_{\mathbb{R}^n} e^{2\pi i(\Phi_\tau(x,\xi)-y\cdot \xi)} \frac{ a_k(x,\xi) \psi_{x,\omega_D}(\omega)  }{\psi_{x}(\omega)}  d\xi \right\Vert\lesssim 2^{-\varepsilon k}.
 \end{align}Indeed, observe that \eqref{pre:final:reduction} implies the estimate
 \begin{align*}
    &  \left\Vert \smallint\limits_{\mathbb{R}^n}e^{2\pi i\Phi_\tau(x,\xi)}a_k(x,\xi)e^{-2\pi i y\cdot \xi}    \right\Vert_{L^1(\mathbb{R}^n_x)}\\
    &= \left\Vert \smallint\limits_{\Omega } \smallint\limits_{\mathbb{R}^n} e^{2\pi i(\Phi_\tau(x,\xi)-y\cdot \xi)} \frac{ a_k(x,\xi) \psi_{x,\omega_D}(\omega)  }{\psi_{x}(\omega)}  d\xi d\omega_D    \right\Vert_{L^1(\mathbb{R}^n_x)}\\
    &\leq |\Omega| \sup_{\omega_D\in  \Omega}\left\Vert   \smallint\limits_{\mathbb{R}^n} e^{2\pi i(\Phi_\tau(x,\xi)-y\cdot \xi)} \frac{ a_k(x,\xi) \psi_{x,\omega_D}(\omega)  }{\psi_{x}(\omega)}  d\xi \right\Vert\\
    &\lesssim 2^{-\varepsilon k}.
 \end{align*}For the proof of \eqref{pre:final:reduction} we write the frequency variable $ \xi$ in terms of the co-variables $(\lambda,\omega),$ and of the new volume element
 \begin{align*}
     d\xi=d\overline{\xi}d\xi_n=\lambda^{n-1}d\omega d\lambda.
 \end{align*}Then, \eqref{pre:final:reduction} can be written as 
 \begin{align}\label{to:prove:lambda}
     \left\Vert   \smallint\limits_{\mathbb{R}^{n-1}}  \smallint\limits_{\mathbb{R}} e^{2\pi i\lambda(\Phi_\tau(x,\omega)-y\cdot (\omega,1) ) } \frac{ \lambda^{n-1}a_k(x,\lambda(\omega,1)) \psi_{x,\omega_D}(\omega)  }{\psi_{x}(\omega)}  d\lambda d\omega \right\Vert\lesssim 2^{-\varepsilon k}.
 \end{align}To simplify the previous estimate we will consider the following change of variables
 \begin{align*}
    T(\zeta)= \omega=\omega_D+2^{-k/2n}Q(x,\omega_D)^{-1/2}[(2^{ k}+2^{(1-\rho)100 k}+2^{\rho k})^{-C}\zeta_j] ,
 \end{align*} with a constant $C>1.$  Note that
 \begin{align}\label{determinant:Q}     |\textnormal{det}T(\zeta)|=2^{-(n-1)k/2}\det(Q(x,\omega_D))^{-1/2}(2^{ k}+2^{(1-\rho)100 k}+2^{\rho k})^{-C(n-1)}.
 \end{align} In what follows, let us denote
  $$\alpha:= (2^{ k}+2^{(1-\rho)100 k}+2^{\rho k})^{-C},$$
and 
$$\chi_{x,\omega_D}(\zeta):=  \psi(2^{k}Q(x,\omega_D)^{1/2}(\alpha \zeta,\alpha\zeta))(2^{ k}+2^{(1-\rho)100 k}+2^{\rho k})^{-C(n-1)}.$$ 
Then, in view of \eqref{phi:x:omega:D} and \eqref{determinant:Q} we have that
    \begin{align*}
    &  \smallint\limits_{\mathbb{R}^{n-1}}  \smallint\limits_{\mathbb{R}} e^{2\pi i\lambda(\Phi_\tau(x,\omega)-y\cdot (\omega,1) ) } \frac{ \lambda^{n-1}a_k(x,\lambda(\omega,1)) \psi_{x,\omega_D}(\omega)  }{\psi_{x}(\omega)}  d\lambda d\omega\\
     &= \smallint\limits_{\mathbb{R}^{n-1}}  \smallint\limits_{\mathbb{R}} e^{2\pi i\lambda(\Phi_\tau(x,T(\zeta))-y\cdot (T(\zeta),1) ) } \frac{ \lambda^{n-1}a_k(x,\lambda(T(\zeta),1)) \psi_{x,\omega_D}(T(\zeta))  }{\psi_{x}(T(\zeta))}  |\det{T(\zeta)}|d\lambda d\zeta\\
     &= \smallint\limits_{\mathbb{R}^{n-1}}  \smallint\limits_{\mathbb{R}} e^{2\pi i\lambda(\Phi_\tau(x,T(\zeta))-y\cdot (T(\zeta),1) ) } \frac{ \lambda^{n-1}a_k(x,\lambda(T(\zeta),1)) 2^{(n-1)k/2} }{\psi_{x}(T(\zeta))}\det(Q(x,\omega_D))^{\frac{1}{2}} \\
     &\times  \psi(2^{k}Q(x,\omega_D)^{1/2}(\alpha \zeta,\alpha\zeta)) 2^{-(n-1)k/2}\det(Q(x,\omega_D))^{-1/2}(2^{ k}+2^{(1-\rho)100 k}+2^{\rho k})^{-C(n-1)}d\lambda d\zeta\\
      & =\smallint\limits_{\mathbb{R}^{n-1}}  \smallint\limits_{\mathbb{R}} e^{2\pi i\lambda(\Phi_\tau(x,T(\zeta))-y\cdot (T(\zeta),1) ) } \frac{ \lambda^{n-1}a_k(x,\lambda(T(\zeta),1))  }{\psi_{x}(T(\zeta))} \\
     &\hspace{2cm} \times \psi(2^{k}Q(x,\omega_D)^{1/2}(\alpha \zeta,\alpha\zeta))(2^{ k}+2^{(1-\rho)100 k}+2^{\rho k})^{-C(n-1)}d\lambda d\zeta\\
      & =\smallint\limits_{\mathbb{R}^{n-1}}  \smallint\limits_{\mathbb{R}} e^{2\pi i\lambda(\Phi_\tau(x,T(\zeta))-y\cdot (T(\zeta),1) ) } \frac{ \lambda^{n-1}a_k(x,\lambda(T(\zeta),1)) \chi_{x,\omega_D}(\zeta)  }{\psi_{x}(T(\zeta))} d\lambda d\zeta.
 \end{align*}
 Then, by multiplying in both sides of \eqref{to:prove:lambda} by the factor $2^{-(n-1)k/2}$ we can re-write \eqref{to:prove:lambda} in terms of $\omega=T(\zeta)$ as follows
 \begin{align}\label{Final:reduction}
     \left\Vert   \smallint\limits_{\mathbb{R}^{n-1}}  \smallint\limits_{\mathbb{R}} e^{2\pi i\lambda(\Phi_\tau(x,T(\zeta))-y\cdot (T(\zeta),1) ) } b_{k,\omega_D,x}(\lambda,\zeta)  d\lambda d\zeta \right\Vert_{L^1(\mathbb{R}^n_x)}\lesssim 2^{-(n-1)k/2}2^{-\varepsilon k} 
 \end{align}
where
\begin{align*}
b_{k,\omega_D,x}(\lambda,\zeta):=    \frac{2^{-k(n-1)/2} \lambda^{n-1}a_k(x,\lambda(T(\zeta),1)) \chi_{x,\omega_D}(T(\zeta))  }{\psi_{x}(T(\zeta))}  .
\end{align*}So, to finish the proof we will show the estimate in \eqref{Final:reduction}. The advantage in proving \eqref{Final:reduction} is that the symbol $b_{k,\omega_D,x}(\lambda,\zeta)$ allows the use of the phase stationary theorem since it does satisfy the estimates
\begin{align}\label{the:symbol:b}
    |\partial_\lambda^{\beta}\partial_{\zeta}^{\delta}b_{k,\omega_D,x}(\lambda,\zeta)|\lesssim_{\delta} 2^{-k\rho|\beta|},
\end{align}
Note that, when $|\delta|=1,$ the chain rule gives
\begin{align*}
    |\partial_{\zeta}^\delta \chi_{x,\omega_D}(x,\omega_D)|&=|\partial_{\zeta}^\delta\left( \psi(2^{k}Q(x,\omega_D))^{1/2}(\alpha \zeta,\alpha\zeta)\right)|(2^{ k}+2^{(1-\rho)100 k}+2^{\rho k})^{-C(n-1)}\\
    &=|\partial_{\zeta}^\delta\left( \psi(2^{k}Q(x,\omega_D)^{1/2}(\alpha \zeta,\alpha \zeta))\right)|(2^{ k}+2^{(1-\rho)100 k}+2^{\rho k})^{-C(n-1)}\\
    &\lesssim \Vert(\partial^{\delta}\psi)\Vert_{L^\infty}|\partial_{\zeta}^\delta(\psi(2^{k}Q(x,\omega_D)^{1/2}(\alpha \zeta,\alpha \zeta))^{1/2})|\\
    &\lesssim 2^{k|\delta|}(2^{-k\varepsilon}+|J_\tau(x,\omega)|)^{1/2}\alpha^{|\delta|}\\
     &\lesssim 2^{k|\delta|}2^{-k\varepsilon}\alpha^{|\delta|}\\
     &\lesssim 2^{k|\delta|}2^{-k\varepsilon}(2^{ k}+2^{(1-\rho)100 k}+2^{\rho k})^{-C|\delta|}\\
     &\lesssim_\delta 1,
\end{align*} where we have used that on the support of $a_k,$ $|J_{\tau}(x,\xi)|\lesssim 2^{-\varepsilon k}, $ see \eqref{region:support:a:k}. Then, using mathematical induction we can prove the estimate
$$ |\partial_{\zeta}^\delta \chi_{x,\omega_D}(x,\omega_D)|\lesssim_{C,\delta} 1,$$ where the previous estimate depends on $C>\delta$ but it is independent of $\delta$ if $C>0$ is large enough and we take a finite number of derivatives in $\zeta$. This will be the case because we are going to apply the method of stationary phase just taking derivatives $\partial_\zeta^\delta$ in $\zeta$ with $|\delta|=1.$
Now, using that the derivatives of $\chi_{x,\omega_D}(x,\omega_D)$ are uniformly bounded, in order to prove that the derivatives in $\zeta$ satisfy \eqref{the:symbol:b} we proceed as follows for $|\delta|=1,$
\begin{align*}
    |\partial_\zeta^{\delta}b_{k,\omega_D,x}(\lambda,\zeta)|&\sim 2^{-k(n-1)/2}2^{k(n-1)}|\partial_\zeta^{\delta}\left(\frac{a_k(x,\lambda(T(\zeta),1)) \psi_{x,\omega_D}(T(\zeta))  }{\psi_{x}(T(\zeta))}\right)|\\
    &\lesssim 2^{k(n-1)/2}\sum_{\delta_1+\delta_2=\delta}|\partial_{\zeta}^{\delta_1}(a_k(x,\lambda(T(\zeta),1)) \psi_{x,\omega_D}(T(\zeta)))  |\cdot |\partial_{\zeta}^{\delta_2}(1/\psi_{x}(T(\zeta)))|\\
    &\lesssim 2^{k(n-1)/2}\sum_{\delta_1+\delta_2=\delta}|\partial_{\zeta}^{\delta_1}(a_k(x,\lambda(T(\zeta),1))  )  |\\
    &\hspace{4cm}\times \sum_{\delta_2^1+\cdots +\delta_2^k=\delta_2} \frac{ |\partial_{\zeta}^{\delta_2^{1}}\cdots \partial_{\zeta}^{\delta_2^{k}}\psi_{x}(T(\zeta))|    }{|\psi_{x}(T(\zeta)))|^{|\delta_2|+1}}\\
    &\lesssim 2^{k(n-1)/2}\sum_{\delta_1+\delta_2=\delta}|\partial_{\omega}^{\delta_1}(a_k(x,\lambda(T(\zeta),1))  )\lambda^{|\delta_1|}\partial_\zeta^{\delta_1} T(z)  |\\
    &\hspace{4cm}\times \sum_{\delta_2^1+\cdots +\delta_2^k=\delta_2}  | \partial_{\omega}^{(\delta_2^{1},\cdots,\delta_2^{k}) } \psi_{x}(T(\zeta)) \partial_{\zeta}^{ (\delta_2^{1},\cdots,\delta_2^{k}) }(T(\zeta))    |    
\end{align*}since $|\psi_{x}(T(\zeta)))|\sim 1.$ Note that we have the estimate
\begin{align*}
    |\partial_\zeta^\delta T(z)|&=|\partial_{\zeta}^\delta(\omega_D+2^{-k/2n}Q(x,\omega_D)^{-1/2}[(2^{ k}+2^{(1-\rho)100 k}+2^{\rho k})^{-C}\zeta_j])|\\
    &=|\partial_{\zeta}^{\delta}\left(2^{-k/2n}\sum_{j=1}^{n-1}(2^{-2\varepsilon k}+\lambda_j(x,\omega_D)^{2})^{-\frac{1}{4}}(2^{ k}+2^{(1-\rho)100 k}+2^{\rho k})^{-2C}|\zeta_j|^2\right)|\\
     &\lesssim 2^{-k/2n}2^{\varepsilon k/4}(2^{ k}+2^{(1-\rho)100 k}+2^{\rho k})^{-2C}.
\end{align*} Moreover, since $T(\zeta)$ is a polynomial of order $2,$ when $|\delta|\geq 3,$ $\partial_\zeta^\delta T(z)=0,$ and  we trivially have  the bound
$$  |\partial_\zeta^\delta T(z)|\leq C_{\delta } 2^{-k/2n} 2^{\varepsilon k/4}(2^{ k}+2^{(1-\rho)100 k}+2^{\rho k})^{-2C-|\delta|}.$$
In consequence, since $\lambda^{|\delta_1|}\sim 2^{k|\delta_1|},$ we can estimate
\begin{align*}
    &2^{k(n-1)/2}\sum_{\delta_1+\delta_2=\delta}|\partial_{\omega}^{\delta_1}(a_k(x,\lambda(T(\zeta),1))  )\lambda^{|\delta_1|}\partial_\zeta^{\delta_1} T(z)  |\\
    &\hspace{4cm}\times \sum_{\delta_2^1+\cdots +\delta_2^k=\delta_2}  | \partial_{\omega}^{(\delta_2^{1},\cdots,\delta_2^{k}) } \psi_{x}(T(\zeta)) \partial_{\zeta}^{ (\delta_2^{1},\cdots,\delta_2^{k}) }(T(\zeta))    | \\
    &\lesssim 2^{k(n-1)/2} \sum_{\delta_1+\delta_2=\delta} 2^{-\frac{(n-1)k}{2}}  2^{(1-\rho)|\delta_1|k}2^{\varepsilon k/4}(2^{ k}+2^{(1-\rho)100 k}+2^{\rho k})^{-2C-|\delta_1|}2^{k|\delta_1|}\\
    &\hspace{2cm} \times \sum_{\delta_2^1+\cdots +\delta_2^k=\delta_2}2^{-k/2n} 2^{C\varepsilon k|(\delta_2^{1},\cdots,\delta_2^{k})| }2^{\varepsilon k/4}(2^{ k}+2^{(1-\rho)100 k}+2^{\rho k})^{-2C-|(\delta_2^{1},\cdots,\delta_2^{k})|}\\
    &\lesssim \sum_{\delta_1+\delta_2=\delta} \sum_{\delta_2^1+\cdots +\delta_2^k=\delta_2}2^{-k/2n} 2^{(1-\rho)|\delta_1|k}   2^{\varepsilon k/4}  2^{-C k}2^{-50(1-\rho)|\delta_1|)}2^{C\varepsilon k|\delta_2|}2^{\varepsilon k/4}2^{-(C+|\delta_2|/2)k}\\
    &\lesssim 1.
\end{align*}On the other hand, to complete the proof of \eqref{the:symbol:b} we just have to use the estimate 
$$  | \partial_{\lambda}^\gamma a_{k}(x,\xi) |\lesssim_{\beta,\gamma} 2^{-\frac{(n-1)k}{2}}2^{-\rho|\gamma|k} .$$ Since $\lambda\sim 2^k,$  we can estimate 
\begin{align*} 
|\partial_{\lambda}^\beta b_{k,\omega_D,x}(\lambda,\zeta)| &=|\partial_{\lambda}^\beta \left(   \frac{2^{-k(n-1)/2} \lambda^{n-1}a_k(x,\lambda(T(\zeta),1)) \chi_{x,\omega_D}(T(\zeta))  }{\psi_{x}(T(\zeta))}\right)| \\
&\sim \left| \left(   \frac{2^{-k(n-1)/2} 2^{k(n-1)}\partial_{\lambda}^\beta (a_k(x,\lambda(T(\zeta),1))) \chi_{x,\omega_D}(T(\zeta))  }{\psi_{x}(T(\zeta))}\right)\right|\\
&\sim \left|   {2^{k(n-1)/2}\partial_{\lambda}^\beta (a_k(x,\lambda(T(\zeta),1))) \chi_{x,\omega_D}(T(\zeta))  }\right|\\
&\lesssim 2^{k(n-1)/2}  2^{-\frac{(n-1)k}{2}}2^{-\rho|\beta|k}\\
&=2^{-k\rho|\beta|},
\end{align*}as claimed. Indeed, note that the derivatives $\delta_\lambda^{\beta_{i}}(\lambda^{n-1})$ with $|\beta_i|\leq \beta,$ are bounded from above by $\lesssim 2^{k(n-1)},$ when $\lambda\sim 2^{k}.$
On the other hand, let us analyse the phase using its Taylor expansion as follows taking into account that $\omega=T(\zeta)$ 
\begin{align*}
   & \Phi_\tau(x,\omega)-y(\omega,1) \\
   &= \Phi_\tau(x,\omega_D)-y\cdot(\omega_D,1)\\
    &+\nabla_\omega(\Phi_\tau(x,\omega)-y(\omega,1))|_{\omega=\omega_D}\cdot (\omega-\omega_D)+e_{k,x,y,\omega_D}(\zeta)\\
    &= \Phi_\tau(x,\omega_D)-y\cdot(\omega_D,1)\\
    &+(\nabla_\omega\Phi_\tau(x,\omega_D)-\overline{y})\cdot (2^{-k/2n}Q(x,\omega_D)^{-1/2}(\alpha \zeta))+e_{k,x,y,\omega_D}(\zeta)\\
    &= \Phi_\tau(x,\omega_D)-y\cdot(\omega_D,1)\\
    &+(2^{ k}+2^{(1-\rho)100 k}+2^{\rho k})^{-C}2^{-k/2n}(\nabla_\omega\Phi_\tau(x,\omega_D)-\overline{y})\cdot Q(x,\omega_D)^{-1/2}(\zeta)\\
    &+e_{k,x,y,\omega_D}(\zeta)\\
    &= \Phi_\tau(x,\omega_D)-y\cdot(\omega_D,1)\\
    &+(2^{ k}+2^{(1-\rho)100 k}+2^{\rho k})^{-C}2^{-k/2n}Q(x,\omega_D)^{-1/2}(\nabla_\omega\Phi_\tau(x,\omega_D)-\overline{y})\cdot \zeta\\
    &+e_{k,x,y,\omega_D}(\zeta).
\end{align*}
Since the phase function
\begin{align*} F(\zeta) &= \Phi_\tau(x,\omega)-y(\omega,1)\\
 &= \Phi_\tau(x,\omega_D)-y\cdot(\omega_D,1)\\
    &+(2^{ k}+2^{(1-\rho)100 k}+2^{\rho k})^{-C}2^{-k/2n}Q(x,\omega_D)^{-1/2}(\nabla_\omega\Phi_\tau(x,\omega_D)-\overline{y})\cdot \zeta\\
    &+e_{k,x,y,\omega_D}(\zeta),
\end{align*} has  gradient 
\begin{align*}
    \nabla_{\zeta}F(\zeta)=\alpha 2^{-k/2n}Q(x,\omega_D)^{-1/2}(\nabla_\omega \Phi_\tau(x,\omega_D)-\overline{y})+O(2^{-k}),
\end{align*} we have that
\begin{align*}
    |\nabla_{\zeta}F(\zeta)| &= \alpha 2^{-k/2n}|\textnormal{diag}[(\lambda_{Q,i}^{\omega})^{-1/2}]_{1\leq i\leq n-1}(\nabla_\omega \Phi_\tau(x,\omega_D)-\overline{y})|+O(2^{-k})\\
    &= \alpha 2^{-k/2n}|[(2^{-\varepsilon k}I+\nabla_{\omega}^2\Phi_\tau(x,\omega)^2]^{-1/4}(\nabla_\omega \Phi_\tau(x,\omega_D)-\overline{y})|+O(2^{-k})\\
    &=O(\alpha 2^{-k/2n}2^{\varepsilon k/4})+O(2^{-k}).
\end{align*}Then, with $\beta=\alpha 2^{-k/2n}2^{\varepsilon k/4},$ one has the estimate
\begin{equation}\label{F:beta}
    |\nabla_{\zeta}\beta^{-1}F(\zeta)|=O(1) . 
\end{equation}
For our further analysis note that 
\begin{align*}
      \Phi_\tau(x,\omega)-y(\omega,1) = \Phi_\tau(x,\omega_D)-y\cdot(\omega_D,1)+Z(\zeta),\, |Z(\zeta)|=O(\beta). 
\end{align*}
So, in order to apply the non-stationary phase principle\footnote{See Proposition \ref{Non-stationary:phase}. We are going to apply this proposition to the phase $F_\beta:=\beta^{-1}F(\zeta),$ making use of the estimate \eqref{F:beta}.}
to the integral
$$ \smallint\limits_{\mathbb{R}^{n-1}}  \smallint\limits_{\mathbb{R}} e^{2\pi i\lambda(\Phi_\tau(x,T(\zeta))-y\cdot (T(\zeta),1) ) } b_{k,\omega_D,x}(\lambda,\zeta)  d\lambda d\zeta$$
we re-write it as follows
\begin{align*}
    &\left|  \smallint\limits_{\mathbb{R}^{n-1}}  \smallint\limits_{\mathbb{R}} e^{2\pi i\lambda(\Phi_\tau(x,T(\zeta))-y\cdot (T(\zeta),1) ) } b_{k,\omega_D,x}(\lambda,\zeta) d\lambda d\zeta  \right|\\
    &\leq   \smallint\limits_{\mathbb{R}}\left|  \smallint\limits_{\mathbb{R}^{n-1}}   e^{2\pi i\lambda(\Phi_\tau(x,T(\zeta))-y\cdot (T(\zeta),1) ) } b_{k,\omega_D,x}(\lambda,\zeta)   d\zeta  \right| d\lambda\\
    &\sim   \smallint\limits_{\lambda\sim 2^{k}}\left|  \smallint\limits_{\mathbb{R}^{n-1}}   e^{2\pi i\lambda F(\zeta) } b_{k,\omega_D,x}(\lambda,\zeta)   d\zeta  \right| d\lambda\\
    &\lesssim 2^{k}\sup_{\lambda\sim 2^k}  \left|  \smallint\limits_{\mathbb{R}^{n-1}}   e^{2\pi i\lambda F(\zeta) } b_{k,\omega_D,x}(\lambda,\zeta)   d\zeta  \right| d\lambda\\
    &= 2^{k}\sup_{\lambda\sim 2^k}  \left|  \smallint\limits_{\mathbb{R}^{n-1}}   e^{2\pi i\lambda\beta \beta^{-1}F(\zeta) } b_{k,\omega_D,x}(\lambda,\zeta)  d\zeta  \right| d\lambda\\
    &\lesssim 2^{k}\sup_{\lambda\sim 2^k} (\lambda\beta)^{-1}\sim \beta^{-1}.
\end{align*} Consequently,
\begin{align*}
    & \left\Vert   \smallint\limits_{\mathbb{R}^{n-1}}  \smallint\limits_{\mathbb{R}} e^{2\pi i\lambda(\Phi_\tau(x,T(\zeta))-y\cdot (T(\zeta),1) ) } b_{k,\omega_D,x}(\lambda,\zeta)  d\lambda d\zeta \right\Vert_{L^1(\mathbb{R}^n_x)}\\
    &=  \left\Vert   \smallint\limits_{\mathbb{R}^{n-1}}  \smallint\limits_{\mathbb{R}} e^{2\pi i\lambda(\Phi_\tau(x,T(\zeta))-y\cdot (T(\zeta),1) ) } b_{k,\omega_D,x}(\lambda,\zeta)  d\lambda d\zeta \right\Vert_{L^1(D)}\\
    &\leq \beta^{-1}\textnormal{Vol}(D),
 \end{align*} where the eccentric disc $D$ is given by
$$ D=\left\{x: \Phi_\tau(x,\omega_D)=y(\omega_D,1)+O(\beta);\nabla_\omega \Phi_\tau(x,\omega_D)=\overline{y}+Q(x,\omega_D)^{1/2}Z \right\},$$ with $|Z|=O(\beta ) .$
Taking into account that $\textnormal{det}(Q(x,\omega_D))\lesssim 2^{-\varepsilon k}$ one can estimate the volume of $D$ as follows, see \cite[Page 8]{Tao} 
\begin{equation*}
    \textnormal{Vol}(D)\lesssim 2^{-\varepsilon k} \beta^{n+1}.
\end{equation*}So, we can estimate
\begin{align*}
     &\left\Vert   \smallint\limits_{\mathbb{R}^{n-1}}  \smallint\limits_{\mathbb{R}} e^{2\pi i\lambda(\Phi_\tau(x,T(\zeta))-y\cdot (T(\zeta),1) ) } b_{k,\omega_D,x}(\lambda,\zeta)  d\lambda d\zeta \right\Vert_{L^1(\mathbb{R}^n_x)}\leq \beta^{-1}\textnormal{Vol}(D)\\
     &\lesssim 2^{-\varepsilon k}\beta^{n}\\
     & \lesssim 2^{-\varepsilon k} (2^{ k}+2^{(1-\rho)100 k}+2^{\rho k})^{-Cn}2^{-k/2}2^{\varepsilon n k/4}\\
&\leq   2^{-\varepsilon k} 2^{-k/2}(2^{-Cn k}2^{\varepsilon n k/4}).
 \end{align*} Taking $\varepsilon>0$ small enough and $C>0$ large enough 
we have the estimate $2^{-k/2}(2^{-Cn k}2^{\varepsilon n k/4})\leq 2^{-(n-1)k/2}.$ Therefore, we have that
$$\left\Vert   \smallint\limits_{\mathbb{R}^{n-1}}  \smallint\limits_{\mathbb{R}} e^{2\pi i\lambda(\Phi_\tau(x,T(\zeta))-y\cdot (T(\zeta),1) ) } b_{k,\omega_D,x}(\lambda,\zeta)  d\lambda d\zeta \right\Vert_{L^1(\mathbb{R}^n_x)}\lesssim 2^{-\varepsilon k} 2^{-(n-1)k/2}$$
as desired. We have proved \eqref{Final:reduction}. In consequence, the proof of Lemma \ref{Fundamental:deg:lemma} is complete. 
\end{proof}
 \begin{proof}[Proof of Proposition \ref{Proposition:deg}] Let us consider the   kernel
\begin{equation*}
    K_{\textnormal{deg},k}(x,y):=\smallint\limits_{\mathbb{R}^n}e^{2\pi i\Phi_\tau(x,\xi)}a(x,\xi)e^{-(1+i\tau)\textnormal{Im}(\Phi(x,\xi))}\phi_{-\varepsilon k}(J_{\tau}(x,\omega))\eta_{k}(\xi)e^{-2\pi i y\cdot \xi}d\xi.
\end{equation*}Since $a(x,\xi)\in S^{-(n-1)/2}_{1,\mathscr{C}},$ and $e^{-(1+i\tau)\textnormal{Im}(\Phi(x,\xi))}\in S^{0}_{1/2,1/2} ,$ the H\"ormander calculus implies that 
$$\tilde{a}(x,\xi):=a(x,\xi)e^{-(1+i\tau)\textnormal{Im}(\Phi(x,\xi))}\in S^{-(n-1)/2}_{1/2,\mathscr{C}} .$$
 By applying Lemma \ref{Fundamental:deg:lemma} to  $\tilde{a}(x,\xi)\in S^{-(n-1)/2}_{\rho,\mathscr{C}} ,$ with $\rho=1/2,$    the symbol 
\begin{equation}
   \tilde{ a}_{k}(x,\xi)= \tilde{a}(x,\xi)\phi_{-\varepsilon k}(J_{\tau}(x,\omega))\eta_{k}(\xi),
\end{equation}satisfies the inequalities
\begin{equation}
    |\partial_\omega^\beta \partial_{\lambda}^\gamma a_{k}(x,\xi) |\lesssim_{\beta,\gamma} 2^{-\frac{(n-1)k}{2}}2^{-\frac{1}{2}|\gamma|k+\frac{1}{2}|\beta|k} 2^{C\varepsilon k |\beta|},
\end{equation} and the   kernel $ K_{\textnormal{deg},k}$ satisfies the estimate
\begin{equation}
    \left\Vert \smallint\limits_{\mathbb{R}^n}e^{2\pi i\Phi_\tau(x,\xi)}a_k(x,\xi)e^{-2\pi i y\cdot \xi} d\xi   \right\Vert_{L^1(\mathbb{R}^n_x)}\lesssim 2^{-\varepsilon k},
\end{equation} uniformly in $y\in \mathbb{R}^n,$ and in $\varepsilon>0$ small enough.
In view of Remark \ref{remark:final:part:proof} the proof of  Proposition \ref{Proposition:deg} is complete.    
\end{proof}

\subsection{Boundedness of the non-degenerate component} Now, we are going to prove the weak (1,1) boundedness of the non-degenerate component $T_{\textnormal{nondeg}} $ in \eqref{non:de:c}. To do this, we follow the approach in \cite{Tao}, by making the factorisation $T_{\textnormal{nondeg}}=SA+ E $
modulo an error operator $E=T_{\textnormal{nondeg}}-SA$ bounded on $L^1.$ Here, we are going to construct a pseudo-differential operator $S$ of order zero and an averaging operator $A,$ which is bounded on $L^1.$

\subsubsection{Construction of the operators $A$ and  $S$}\label{Construction of the operatorsA:S} Let us consider the projection $\pi_{1}:\textnormal{supp}(a)\rightarrow \mathbb{R}^n$ in the first component $(x,\xi)\mapsto x,$ and let $V$ be an open subset, such that $\overline{V}$ is compact with $\pi_{1}(\textnormal{supp}(a))\subset V.$ Let $1_{\overline{V}}$ be the characteristic function of $\overline{V}.$
The operator $A$ will be defined as follows
$$  Af(x)$$
\begin{align*}
  := \sum_{k\gg 1}\smallint\smallint e^{2\pi i\xi'\cdot \nabla_\xi\Phi_\tau(x,\omega)}\psi(x,\omega)e^{-\mu(x,\omega)\pi i/4} &(1-\phi_{-\varepsilon k}(J(x,\omega)))\eta_{k}(\xi')e^{-(1+i\tau)\textnormal{Im}(\Phi(x,\xi'))}
\end{align*}
$$ \times 1_{\overline{V}}(x) |J(x,\omega)|^{1/2}\widehat{f}(\xi')d\xi'd\omega,$$
where $\nabla_\xi\Phi_\tau(x,\omega)$ denotes the value of $\nabla_\xi\Phi_\tau(x,\xi)$ at $\xi=(\omega,1),$ and
 $\psi:=\psi(x,\omega)$ is a bump function to be chosen later. The integer quantities $\mu(x,\omega)$ are such that they satisfy, according to the principle of stationary phase, the asymptotic expansion (see Proposition \ref{Stationar:phase})
    \begin{align*}
        \smallint e^{2\pi i\xi'\cdot \nabla_\xi\Phi_\tau(x,\omega)}\psi(x,\omega)d\omega=e^{2\pi i\Phi_\tau(x,\xi')}\psi(x,\omega')e^{\mu(x,\omega')\pi i/4}(\lambda')^{-(n-1)/2}|J(x,\omega')|^{-1/2}+\cdots,
    \end{align*} when $|\xi'|\gg 1.$ Since $\mu(x,\omega)$ are smooth functions in the region of integration they are constant on each connected component of this region.    

\begin{remark}
 When  $\textnormal{Im}(\Phi(x,\xi))\equiv 0,$ that is if $\Phi$ is real-valued, the definition of $A$ purely involves the phase function $\Phi_\tau$ and the {\it curvature term } $J(x,\omega)=\det(\nabla_\omega^2\Phi_\tau(x,\omega)),$ the $L^1$-boundedness of $A$ follows from the argument for real-valued phase functions in Tao \cite[Pages 14-15]{Tao}. 
\end{remark}

\begin{lemma}\label{boundedness:A}
    $A$ extends to a bounded operator on $L^1(\mathbb{R}^n)$ provided that there exists $\varkappa_0>0,$ such that for all $(x,\theta)\in \overline{V}\times \overline{\{\theta: |\theta|\sim 1\}},$ the lower bound $$\textnormal{Im}(\Phi(x,\theta))>\varkappa_0,$$ holds.
\end{lemma}
\begin{proof}First, assume that $\widehat{f}(\xi)$ is smooth and compactly supported and that $f$ vanishes when $|\xi|\leq 1.$ Note that the support of $\eta_{k}(\xi')$ is contained in the region $|\xi'|\sim 2^ k.$
By applying summation by parts, because of the properties of the support of $f,$ boundary terms vanish and we can write
\begin{align*}
  &Af(x)\\
  &:= \sum_{k\gg 1}\smallint\smallint_{|\xi'|\sim 2^ k} e^{2\pi i\xi'\cdot \nabla_\xi\Phi_\tau(x,\omega)}\psi(x,\omega)e^{-\mu(x,\omega)\pi i/4} (1-\phi_{-\varepsilon k}(J(x,\omega)))\eta_{k}(\xi')e^{-(1+i\tau)\textnormal{Im}(\Phi(x,\xi'))}\\
&\times |J(x,\omega)|^{1/2}\widehat{f}(\xi')d\xi'd\omega\\
&=- \sum_{k\gg 1}\smallint\smallint_{|\xi'|\sim 2^ k} e^{2\pi i\xi'\cdot \nabla_\xi\Phi_\tau(x,\omega)}\psi(x,\omega)e^{-\mu(x,\omega)\pi i/4} (\phi_{-\varepsilon (k+1)}(J(x,\omega))-\phi_{-\varepsilon k}(J(x,\omega)))\\
&\times \phi_{k}(\xi')e^{-(1+i\tau)\textnormal{Im}(\Phi(x,\xi'))} 1_{\overline{V}}(x)|J(x,\omega)|^{1/2}\widehat{f}(\xi')d\xi'd\omega.
\end{align*}
In view of the triangle inequality, it suffices to prove  the estimate
\begin{equation}
    \Vert A_k f\Vert_{L^1(\mathbb{R}^n)}\lesssim 2^{-\varepsilon k}\Vert f\Vert_{L^1(\mathbb{R}^n)},
\end{equation} where 
\begin{align*}
    A_kf(x)&=- \smallint\smallint_{|\xi'|\sim 2^ k} e^{2\pi i\xi'\cdot \nabla_\xi\Phi_\tau(x,\omega)}\psi(x,\omega)e^{-\mu(x,\omega)\pi i/4}) (\phi_{-\varepsilon (k+1)}(J(x,\omega))-\phi_{-\varepsilon k}(J(x,\omega)))\\
&\times \phi_{k}(\xi')e^{-(1+i\tau)\textnormal{Im}(\Phi(x,\xi'))} 1_{\overline{V}}(x) |J(x,\omega)|^{1/2}\widehat{f}(\xi')d\xi'd\omega. 
\end{align*}
Let us denote by $$\phi_k^\bullet(\xi'):=\phi_k(\xi')\chi_{ \{\xi': |\xi'|\sim 2^k\} }(\xi').$$
Then,
\begin{align*}
    A_kf(x)&=- \smallint\smallint e^{2\pi i\xi'\cdot \nabla_\xi\Phi_\tau(x,\omega)}\psi(x,\omega)e^{-\mu(x,\omega)\pi i/4}) (\phi_{-\varepsilon (k+1)}(J(x,\omega))-\phi_{-\varepsilon k}(J(x,\omega)))\\
&\times \phi_{k}^\bullet(\xi')e^{-(1+i\tau)\textnormal{Im}(\Phi(x,\xi'))} 1_{\overline{V}}(x) |J(x,\omega)|^{1/2}\widehat{f}(\xi')d\xi'd\omega. 
\end{align*}
Using the Fourier inversion formula, we have that
\begin{align*}
   & \smallint e^{2\pi i\xi'\cdot \nabla_\xi\Phi_\tau(x,\omega)} \phi_{k}^\bullet(\xi') e^{-(1+i\tau)\textnormal{Im}(\Phi(x,\xi'))}1_{\overline{V}}(x) \widehat{f}(\xi')d\xi'\\
   &=1_{\overline{V}}(x)[e^{-(1+i\tau)\textnormal{Im}(\Phi(x,D))}\phi_{k}^\bullet(D)f](\nabla_\xi\Phi_\tau(x,\omega)).
\end{align*} Since $\textnormal{supp}[(\phi_{-\varepsilon (k+1)}(J(x,\omega))]$ is contained in the set where $J(x,\omega)\lesssim 2^{-\varepsilon k},$ the difference $\phi_{-\varepsilon (k+1)}(J(x,\omega))-\phi_{-\varepsilon k}(J(x,\omega)),$ is supported on the region $J(x,\omega)\sim 2^{-k\varepsilon}.$ Since $|\phi_k(t)|\lesssim \Vert\phi\Vert_{L^{\infty}},$ $\forall t>0,$ we have that $$|(\phi_{-\varepsilon (k+1)}(J(x,\omega))-\phi_{-\varepsilon k}(J(x,\omega))||J(x,\omega)|^{1/2}\lesssim 2^{-\varepsilon k/2},$$
and consequently
\begin{align*}
    \Vert A_k f\Vert_{L^1(\mathbb{R}^n)}  &=\smallint|A_k f(x)|dx\\
    &\lesssim\smallint\smallint |1_{\overline{V}}(x)[e^{-(1+i\tau)\textnormal{Im}(\Phi(x,D))}\phi_{k}^\bullet (D)f](\nabla_\xi\Phi_\tau(x,\omega))|2^{-\varepsilon k/2}|\phi(x,\omega)|dx d\omega.
\end{align*}    Note that the operator 
\begin{equation*}
   \Omega(x,D)= 1_{\overline{V}}(x)e^{-(1+i\tau)\textnormal{Im}(\Phi(x,D))}\phi_{k}^\bullet(D)
\end{equation*}is uniformly bounded in $k$ on $L^1(\mathbb{R}^n).$ Indeed, its kernel $\Omega(x,D)\delta(y),$ given by  
$$ \Omega(x,D)\delta(y)=\smallint e^{2\pi i(x-y)\xi}1_{\overline{V}}(x)\phi_{k}^\bullet(\xi)e^{-(1+i\tau)\textnormal{Im}(\Phi(x,\xi))}d\xi  $$
belongs to $L^1(\mathbb{R}_x^n)$ uniformly in $y\in \mathbb{R}^n.$
Since  the phase function $\Phi$ is of strictly positive type, that is, $\textnormal{Im}(\phi)> 0,$ and $|\phi_{k}(\xi)|\lesssim 1,$ note that
\begin{align*}
    \smallint|\smallint e^{2\pi i(x-y)\xi}1_{\overline{V}}(x) &\phi_{k}^\bullet(\xi)e^{-(1+i\tau)\textnormal{Im}(\Phi(x,\xi))}d\xi|dx \\
   = \smallint|\smallint_{|\xi|\sim 2^k} e^{2\pi i(x-y)\xi}1_{\overline{V}}(x) &\phi_{k}(\xi)e^{-(1+i\tau)\textnormal{Im}(\Phi(x,\xi))}d\xi|dx \\
    &\leq \smallint   1_{\overline{V}}(x)\sup_{|\xi'|\sim 2^{k}}e^{-\textnormal{Im}(\Phi(x,\xi'))} dx\smallint_{|\xi|\sim 2^{k}}|\phi_{k}(\xi)| d\xi\\
    &\lesssim |\overline{V}|\sup_{x\in \overline{V}} \sup_{|\xi'|\sim 2^{k}}e^{-\textnormal{Im}(\Phi(x,\xi'))}\smallint_{|\xi|\sim 2^{k}}d\xi\\
    &\lesssim |\overline{V}|\sup_{x\in \overline{V}} \sup_{|\xi'|\sim 2^{k}}e^{-2^{k}\textnormal{Im}(\Phi(x,2^{-k}\xi'))}2^{kn}\\
    & \lesssim |\overline{V}|\sup_{x\in \overline{V}} \sup_{|\theta|\sim 1}e^{-2^{k}\textnormal{Im}(\Phi(x,\theta))}2^{kn}<\infty.
\end{align*}Indeed, since $\textnormal{Im}(\phi)> 0,$ when $|\xi|\neq 0,$ the compactness of the closed annulus $\overline{\{\theta: |\theta|\sim 1\}},$ and of $\overline{V},$ imply that there exists $\varkappa_0>0$ such that for all $(x,\omega)\in \overline{V}\times \overline{\{\theta: |\theta|\sim 1\}},$ one has the lower bound $$\textnormal{Im}(\Phi(x,\theta))>\varkappa_0.$$ Therefore, 
$$\mathscr{K}:=\sup_{k>0}\sup_{x\in \overline{V}} \sup_{|\theta|\sim 1}e^{-2^{k}\textnormal{Im}(\Phi(x,\theta))}2^{kn}\leq \sup_{k>0} e^{-2^{k}\varkappa_0}2^{kn}<\infty.$$
Thus, we have proved that  $\sup_{y\in \mathbb{R}^n}\Vert  \Omega(x,D)\delta(y) \Vert_{L^1(\mathbb{R}^n_x)}<\infty,$ and Schur's lemma gives the boundedness of $\Omega(x,D)$ on $L^1(\mathbb{R}^n).$ Note also, that if $z=F(x)$ is an smooth change of coordinates on $\overline{V}$, the operator $$\Omega(F^{-1}(z),D)=1_{\overline{V}}(F^{-1}(z))\phi_{k}^\bullet(D)e^{-(1+i\tau)\textnormal{Im}(\Phi(F^{-1}(z),D))}$$ is also bounded on $L^1.$ Moreover, since $1_{\overline{V}}\circ F^{-1}=1_{F(\overline{V})},$ by following the previous argument, we have that
\begin{align*}
   \sup_{y}  \Vert \Omega(F^{-1}(z),D)\delta(y)\Vert_{L^1(\mathbb{R}^n_z)}\lesssim  |F(\overline{V})|\sup_{z\in F(\overline{V})} \sup_{|\theta|\sim 1}e^{-2^{k}\textnormal{Im}(\Phi(F^{-1}(z),\theta))}2^{kn}<\infty.
\end{align*}Now, since for any $\omega,$ $z=F_\omega(x)=\nabla_\xi\Phi_\tau(x,\omega)$ is a diffeomorphism on $\overline{V},$ taking $R>0$ large enough such that $\{\omega: (x,\omega)\in \textnormal{supp}(\phi(x,\omega))\}\subset\{|\omega|\leq R\},$ we have that
\begin{align*}
    &\Vert A_k f\Vert_{L^1(\mathbb{R}^n)}  \\
    &\lesssim\smallint_{|\omega|\leq R}\smallint |1_{\overline{V}}(x)[\phi_{k}^\bullet(D)e^{-(1+i\tau)\textnormal{Im}(\Phi(x,D))}f](\nabla_\xi\Phi_\tau(x,\omega))|2^{-\varepsilon k/2}|\phi(x,\omega)|dx d\omega\\
    &\lesssim \Vert\phi(x,\omega)\Vert_{L^{\infty}} \smallint_{|\omega|\leq R}\smallint |\Omega(x,D)f(F_{\omega}(x))|2^{-\varepsilon k/2}dx d\omega\\
    &= \Vert\phi(x,\omega)\Vert_{L^{\infty}} \smallint_{|\omega|\leq R}\smallint |\Omega(F^{-1}_\omega(z),D)f(z)|2^{-\varepsilon k/2}|\det(DF_{\omega}^{-1}(z))|dz d\omega\\
    &\lesssim 2^{-\varepsilon k/2} \smallint_{|\omega|\leq R}\Vert\Omega(F^{-1}_\omega(z),D)f\Vert_{L^1(\mathbb{R}^n_z)}d\omega\\
    &\lesssim 2^{-\varepsilon k/2} \sup_{|\omega'|\leq R}\Vert\Omega(F^{-1}_{\omega'}(z),D)f\Vert_{L^1(\mathbb{R}^n_z)} \smallint_{|\omega|\leq R}d\omega\\
    & \lesssim 2^{-\varepsilon k/2} \sup_{|\omega'|\leq R} |F_{\omega'}(\overline{V})|\sup_{z\in F_{\omega'}(\overline{V})} \sup_{|\theta|\sim 1}e^{-2^{k}\textnormal{Im}(\Phi(F_{\omega'}^{-1}(z),\theta))}2^{kn}\Vert f\Vert_{L^1(\mathbb{R}^n)}\\
    &=2^{-\varepsilon k/2} \sup_{|\omega'|\leq R} |F_{\omega'}(\overline{V})|\sup_{x\in \overline{V}} \sup_{|\theta|\sim 1}e^{-2^{k}\textnormal{Im}(\Phi(x,\theta))}2^{kn}\Vert f\Vert_{L^1(\mathbb{R}^n)}\\
    &\lesssim_{\mathscr{K}} 2^{-\varepsilon k/2} \sup_{|\omega'|\leq R} |F_{\omega'}(\overline{V})|\Vert f\Vert_{L^1(\mathbb{R}^n)}.
\end{align*} Since $F_{\omega}$ is $C^1$ in $\omega,$ and $\overline{V}$ is compact, $ \sup_{|\omega'|\leq R} |F_{\omega'}(\overline{V})|<\infty.$ So we have the estimate
\begin{equation}
    \Vert A_k f\Vert_{L^1(\mathbb{R}^n)}\lesssim 2^{-k\varepsilon/2}\Vert f\Vert_{L^1(\mathbb{R}^n)}
\end{equation}as desired proving the $L^1$-boundedness of $A.$ The proof of Lemma \ref{boundedness:A} is complete.
\end{proof} 

\begin{definition}
 Using that the mapping $$(x,\xi)\mapsto (x,\theta):=(x, \nabla_x\Phi_\tau(x,\xi))$$ is a diffeomorphism on the support of the symbol $a:=a(x,\xi),$ we define the pseudo-differential operator $S$ associated to the symbol\begin{equation}\label{symbol:s}
   s:= s(x, \theta )=\frac{\lambda^{\frac{n-1}{2}}a(x,\xi)}{\psi(x,\omega)}1_{\overline{V}},
\end{equation}   where, $\forall (x,\omega),$ $\psi(x,\omega)\neq 0,$   is a bounded function away from zero and smooth for all $x$ and for all $\omega$. 
\end{definition}
\begin{remark}
    Since $V$ is an open neighbourhood of the set $\{x\in \mathbb{R}^n:(x,\xi)\in\textnormal{supp}[a]\},$ we have that $a(x,\xi)=0$ on $\overline{V}\setminus V.$ In consequence  $s(x, \theta )=\frac{\lambda^{\frac{n-1}{2}}a(x,\xi)}{\psi(x,\omega)}1_{\overline{V}}=\frac{\lambda^{\frac{n-1}{2}}a(x,\xi)}{\psi(x,\omega)},$ is smooth in $(x,\omega)$ and $\{x\in \mathbb{R}^n:(x,\theta)\in \textnormal{supp}[s] \}\subset V.$ 
\end{remark}
Now we prove the following fundamental property of the operator $S.$
\begin{lemma}
    The pseudo-differential operator $S$ defined by
    \begin{equation}
        Sf(x)=\smallint e^{2\pi i x\cdot \xi}s(x,\theta(x,\xi))\widehat{f}(\xi)d\xi,
    \end{equation}
\end{lemma}is of weak (1,1) type.
\begin{proof} Since $\lambda^{ \frac{n-1}{2} }=\xi_{n}^{ \frac{n-1}{2}}$ and $a_\tau:=a_\tau(x,\xi)\in S^{-\frac{n-1}{2}}_{1,0}$ we have that $s:=s(x,\theta(x,\xi))$ is symbol of order zero. Then, the  pseudo-differential operator $S=s(x,\nabla_x\Phi_\tau(x,D))$ is essentially a Calder\'on-Zygund operator and then it satisfies the weak (1,1) inequality, namely, $S:L^1\rightarrow L^{1,\infty}$ is bounded. 
\end{proof}

\subsection{Control of the error operator}\label{L1:boundedness:remaninder} In this subsection we prove that the error operator $E$ is bounded on $L^{1}.$ The idea behind the proof of this fact can be understood from the calculus of Fourier integral operators. The operator $A$ is essentially a Fourier integral operator with the same phase function that $T,$ in the sense that its phase $2\pi i\Phi_\tau(x,\omega)$ evaluated at the stationary point $\omega=\omega'$ is $2\pi i\Phi_\tau(x,\xi'),$ see \cite[Page 19, Eq. (16)]{Tao}. So, intuitively, if $A$ would have  had  the form in \eqref{FIO:Quantisation}, namely, $Af(x)=\smallint\limits_{\mathbb{R}^n}e^{i\phi_A(x,\xi)}\sigma_A(x,\xi)\widehat{f}(\xi)d\xi,$  one could compute an asymptotic expansion for the symbol $\sigma(x,\xi)$ of $SA$ as follows 
\begin{equation}
    \sigma(x,\xi)=\smallint\smallint \sigma_A(x,\xi)s(x,\zeta)e^{2\pi i(\phi_A(y,\xi)-\phi_A(x,\xi)+(x-y)\cdot \zeta) }d\zeta dy,
\end{equation} and satisfying the asymptotic expansion, see e.g. \cite[Page 17]{Staubach},
\begin{equation}
   \sigma(x,\xi)=s(x,\nabla\phi_A(x,\xi))\sigma_A(x,\xi)+\sum_{0<|\alpha|<N}\frac{1}{\alpha!}\sigma_{\alpha}(x,\xi)+r_{N}(x,\xi), 
\end{equation}and the symbols $\sigma_{\alpha}$ satisfying the estimates (under suitable conditions on the phase $\phi_A$)
$$|\partial_{\xi}^{\beta}\partial_x^{\gamma}\sigma_\alpha(x,\xi)|\lesssim_{\varepsilon'} (1+|\xi|)^{-\frac{n-1}{2}-(\frac{1}{2}-\varepsilon')|\alpha|-\frac{|\beta|}{2} +\frac{1}{2}|\gamma|},$$ for any $0<\varepsilon'<1/2.$
Below we will prove that essentially $T_{\textnormal{nondeg}}$ and $SA$ have the same principal symbol $s(x,\nabla\phi_A(x,\xi))\sigma_A(x,\xi),$ and then the order $\tilde{\mu}$ of the error operator $E=T_{\textnormal{nondeg}}-SA$  should be, at least heuristically,  the order of the $\sigma_{\alpha}$'s when $|\alpha|=1,$ that is $\tilde{\mu}\leq -\frac{n-1}{2}-\frac{1}{2}+\varepsilon'. $ This order $\tilde{\mu}$ is good enough to deduce the $L^1$-boundedness of $E.$ However, since the definition of $A$  also involves integration with respect to $\omega,$ the structure of this operator is not as the one in \eqref{FIO:Quantisation}, and one has to use the stationary phase method to prove that each dyadic component of the error $E,$ has a kernel with a  $L^1_x$-norm satisfying uniformly in $y,$ a decay proportional to $2^{C\varepsilon k}2^{-k/2}.$ Then, if proved this we are allowed to use Proposition 6.1 in \cite[Page 17]{Tao} to deduce the boundedness of $E,$ even under the presence of the oscillating term $e_{\Phi}(x,\xi)=e^{-(1+i\tau)\Phi(x,\xi)}.$

First, let us calculate the kernel $E\delta_z(x)=T_{\textnormal{nondeg}}\delta_{z}(x)-SA\delta_z(x)$ of the error operator $E.$ Note that
\begin{align*}
     &A\delta_z(x)\\
  &= \sum_{k\gg 1}\smallint\smallint e^{2\pi i\xi'\cdot \nabla_\xi\Phi_\tau(x,\omega)}\psi(x,\omega)e^{-\mu(x,\omega)\pi i/4}) (1-\phi_{-\varepsilon k}(J(x,\omega)))\eta_{k}(\xi')e^{-(1+i\tau)\textnormal{Im}(\Phi(x,\xi'))}\\
&\times |J(x,\omega)|^{1/2}\widehat{\delta_z}(\xi')d\xi'd\omega\\
&= \sum_{k\gg 1}\smallint\smallint\smallint e^{2\pi i\xi'\cdot \nabla_\xi\Phi_\tau(x,\omega)-2\pi iy\cdot \xi'}\psi(x,\omega)e^{-\mu(x,\omega)\pi i/4}) (1-\phi_{-\varepsilon k}(J(x,\omega)))\\
&\times \eta_{k}(\xi')e^{-(1+i\tau)\textnormal{Im}(\Phi(x,\xi'))} |J(x,\omega)|^{1/2}\delta_z(y)d\xi'd\omega dy\\
&= \sum_{k\gg 1}\smallint\smallint e^{2\pi i\xi'\cdot \nabla_\xi\Phi_\tau(x,\omega)-2\pi iz\cdot \xi'}\psi(x,\omega)e^{-\mu(x,\omega)\pi i/4} (1-\phi_{-\varepsilon k}(J(x,\omega)))\\
&\times \eta_{k}(\xi')e^{-(1+i\tau)\textnormal{Im}(\Phi(x,\xi'))} |J(x,\omega)|^{1/2}d\xi'd\omega .
\end{align*}
Consequently,
\begin{align*}
    SA\delta_{z}(x) &=\smallint e^{2\pi i x\cdot \zeta}s(x,\zeta)\widehat{A\delta_z}(\zeta)d\zeta\\
    &=\smallint\smallint e^{2\pi i (x-y)\cdot \zeta}s(x,\zeta)A\delta_z(y)dyd\zeta\\
    &=       \sum_{k\gg 1}\smallint\smallint\smallint\smallint e^{2\pi i\xi'\cdot \nabla_\xi\Phi_\tau(y,\omega)-2\pi iz\cdot \xi'+2\pi i (x-y)\cdot \zeta}s(x,\zeta)\psi(y,\omega)e^{-\mu(y,\omega)\pi i/4} \\
&\times (1-\phi_{-\varepsilon k}(J(y,\omega)))\eta_{k}(\xi')e^{-(1+i\tau)\textnormal{Im}(\Phi(y,\xi'))} |J(y,\omega)|^{1/2}d\xi'd\omega dy d\zeta\\
&=       \sum_{k\gg 1}\smallint\smallint\smallint\smallint e^{2\pi i[\xi'\cdot( \nabla_\xi\Phi_\tau(y,\omega)-z)+ (x-y)\cdot \zeta]}s(x,\zeta)\psi(y,\omega)e^{-\mu(y,\omega)\pi i/4} \\
&\times (1-\phi_{-\varepsilon k}(J(y,\omega)))\eta_{k}(\xi')e^{-(1+i\tau)\textnormal{Im}(\Phi(y,\xi'))} |J(y,\omega)|^{1/2}d\xi'd\omega dy d\zeta. 
\end{align*}
On the other hand, the kernel of $T_{\textnormal{nondeg}}$ is given by
\begin{align*}
     & T_{\textnormal{nondeg}} \delta_z(x)\\
     &= \sum_{k\gg 1}\smallint e^{2\pi i[\Phi_\tau(x,\xi') -\xi'\cdot z]}a(x,\xi')e^{-(1+i\tau)\textnormal{Im}(\Phi(x,\xi'))}(1-\phi_{-\varepsilon k}(J(x,\omega')))\eta_{k}(\xi')d\xi'\\
     &= \sum_{k\gg 1} \smallint e^{2\pi i[\Phi_\tau(x,\xi') -\xi'\cdot z]}s(x,\nabla_x\Phi_\tau(x,\xi'))(\lambda')^{-\frac{n-1}{2}}e^{-(1+i\tau)\textnormal{Im}(\Phi(x,\xi'))}\\
     &\times (1-\phi_{-\varepsilon k}(J(x,\omega')))\eta_{k}(\xi')d\xi'.
\end{align*}In order to prove the $L^1$ boundedness of $E$  it suffices to prove the following estimate on each dyadic component
\begin{align*}
   \Vert &\smallint e^{2\pi i[\Phi_\tau(x,\xi') -\xi'\cdot z]}s(x,\nabla_x\Phi_\tau(x,\xi'))(\lambda')^{-\frac{n-1}{2}}e^{-(1+i\tau)\textnormal{Im}(\Phi(x,\xi'))}\\
    &\times (1-\phi_{-\varepsilon k}(J(x,\omega')))\eta_{k}(\xi')d\xi'\\
    &- \smallint\smallint\smallint\smallint e^{2\pi i[\xi'\cdot( \nabla_\xi\Phi_\tau(y,\omega)-z)+ (x-y)\cdot \zeta]}s(x,\zeta)\psi(y,\omega)e^{-\mu(y,\omega)\pi i/4} \\
&\times (1-\phi_{-\varepsilon k}(J(y,\omega)))\eta_{k}(\xi')e^{-(1+i\tau)\textnormal{Im}(\Phi(y,\xi'))} |J(y,\omega)|^{1/2}d\omega dy d\zeta d\xi'\Vert_{L^1(\mathbb{R}^n_x)}\lesssim 2^{C\varepsilon k}2^{-k/2}.
\end{align*}However, the integral inside of this $L^1$-norm, can be written as follows
\begin{align*}
  \Vert &\smallint e^{2\pi i[\Phi_\tau(x,\xi') -\xi'\cdot z]}s(x,\nabla_x\Phi_\tau(x,\xi'))(\lambda')^{-\frac{n-1}{2}}e^{-(1+i\tau)\textnormal{Im}(\Phi(x,\xi'))}\\
    &\times (1-\phi_{-\varepsilon k}(J(x,\omega')))\eta_{k}(\xi')d\xi'\\
    &- \smallint\smallint\smallint\smallint e^{2\pi i[\xi'\cdot( \nabla_\xi\Phi_\tau(y,\omega)-z)+ (x-y)\cdot \zeta]}s(x,\zeta)\psi(y,\omega)e^{-\mu(y,\omega)\pi i/4} \\
&\times (1-\phi_{-\varepsilon k}(J(y,\omega)))\eta_{k}(\xi')e^{-(1+i\tau)\textnormal{Im}(\Phi(y,\xi'))} |J(y,\omega)|^{1/2}d\omega dy d\zeta d\xi'\Vert_{L^1(\mathbb{R}^n_x)}\\
=\Vert & \smallint e^{2\pi i[\Phi_\tau(x,\xi') -\xi'\cdot z]}s(x,\nabla_x\Phi_\tau(x,\xi'))(\lambda')^{-\frac{n-1}{2}}e^{-(1+i\tau)\textnormal{Im}(\Phi(x,\xi'))}\\
    &\times (1-\phi_{-\varepsilon k}(J(x,\omega')))\eta_{k}(\xi')d\xi'\\
    &- \smallint\smallint\smallint\smallint e^{2\pi i[\Phi_\tau(x,\xi')-z\cdot \xi'+[\xi'\cdot \nabla_\xi\Phi_\tau(y,\omega)-\Phi_\tau(x,\xi')]+ (x-y)\cdot \zeta]}s(x,\zeta)\psi(y,\omega)e^{-\mu(y,\omega)\pi i/4} \\
&\times (1-\phi_{-\varepsilon k}(J(y,\omega)))\eta_{k}(\xi')e^{-(1+i\tau)\textnormal{Im}(\Phi(y,\xi'))} |J(y,\omega)|^{1/2}d\omega dy d\zeta d\xi'\Vert_{L^1(\mathbb{R}^n_x)}\\
=\Vert & \smallint e^{2\pi i[\Phi_\tau(x,\xi') -\xi'\cdot z]}\\
    &\times \{s(x,\nabla_x\Phi_\tau(x,\xi'))(\lambda')^{-\frac{n-1}{2}}e^{-(1+i\tau)\textnormal{Im}(\Phi(x,\xi'))} (1-\phi_{-\varepsilon k}(J(x,\omega')))\\
    &- \smallint\smallint\smallint e^{2\pi i[\xi'\cdot \nabla_\xi\Phi_\tau(y,\omega)-\Phi_\tau(x,\xi')]+ (x-y)\cdot \zeta]}s(x,\zeta)\psi(y,\omega)e^{-\mu(y,\omega)\pi i/4} \\
&\times (1-\phi_{-\varepsilon k}(J(y,\omega)))e^{-(1+i\tau)\textnormal{Im}(\Phi(y,\xi'))} |J(y,\omega)|^{1/2}d\omega dy d\zeta\}\eta_{k}(\xi') d\xi'\Vert_{L^1(\mathbb{R}^n_x)}.
\end{align*}So, let us simplify the notation above by defining
\begin{align*}
 W^{0}_{x,k,z}(\xi'):=  s(x,\nabla_x\Phi_\tau(x,\xi'))(\lambda')^{-\frac{n-1}{2}}e^{-(1+i\tau)\textnormal{Im}(\Phi(x,\xi'))} (1-\phi_{-\varepsilon k}(J(x,\omega'))),
\end{align*} and 
\begin{align*}
    W_{x,k,z}(\xi')&:= \smallint\smallint\smallint e^{2\pi i[\xi'\cdot \nabla_\xi\Phi_\tau(y,\omega)-\Phi_\tau(x,\xi')+ (x-y)\cdot \zeta]}s(x,\zeta)\psi(y,\omega)e^{-\mu(y,\omega)\pi i/4} \\
&\times (1-\phi_{-\varepsilon k}(J(y,\omega)))e^{-(1+i\tau)\textnormal{Im}(\Phi(y,\xi'))} |J(y,\omega)|^{1/2}d\omega dy d\zeta.
\end{align*} 
Observe that in view of the Euler homogeneity identity $$\Phi_\tau(x,\xi')=\xi'\cdot \nabla_\xi\Phi_\tau(x,\xi'),$$ we can re-write the phase of the integral $ W_{x,k,z}(\xi')$ as follows
\begin{align*}
 &\xi'\cdot \nabla_\xi\Phi_\tau(y,\omega)-\Phi_\tau(x,\xi')+ (x-y)\cdot \zeta\\
 &= \xi'\cdot( \nabla_\xi\Phi_\tau(y,\omega)- \nabla_\xi\Phi_\tau(x,\xi'))+ (x-y)\cdot \zeta\\
 &=:\Psi_{x,k,z}(\omega, \zeta,y),
\end{align*}and then
\begin{align*}
    W_{x,k,z}(\xi')&:= \smallint\smallint\smallint e^{2\pi i \Psi_{x,k,z}(\omega, \zeta,y) }s(x,\zeta)\psi(y,\omega)e^{-\mu(y,\omega)\pi i/4} \\
&\times (1-\phi_{-\varepsilon k}(J(y,\omega)))e^{-(1+i\tau)\textnormal{Im}(\Phi(y,\xi'))} |J(y,\omega)|^{1/2}d\omega dy d\zeta.
\end{align*} 
The analysis above shows that in order to prove  the $L^1$ boundedness of $E$  it suffices to prove the following estimate 
\begin{equation}\label{W:errors}
    \Vert \smallint e^{2\pi i[\Phi_\tau(x,\xi') -\xi'\cdot z]}  [W_{x,k,z}(\xi')- W^{0}_{x,k,z}(\xi')]\eta_{k}(\xi')d\xi'\Vert_{L^1(\mathbb{R}^n_x)}\lesssim 2^{C\varepsilon k}2^{-k/2}.
\end{equation} 
To do so, we will estimate the integral $W_{x,k,z}(\xi')$ on the set of points  $(\omega,y,\zeta)$ that contribute more information to the integrand. This set of points is determined by the region:
$$ \mathscr{R}=\{(\omega,y,\zeta):\omega-\omega'=y-x=O(2^{\varepsilon k}2^{-k/2})\}.$$
By following the argument in \cite[Page 19]{Tao} one can make the approximation 
$$s(x,\zeta)=s(x,\nabla_x\Phi_\tau(x,\xi'))+O(2^{C\varepsilon k}2^{-k/2})$$
where the error term $O(2^{C\varepsilon k}2^{-k/2})=O_{x,\xi',\zeta}(2^{C\varepsilon k}2^{-k/2})$ has derivatives satisfying the estimate
\begin{equation}\label{remainder:term:final}
    |\partial_{\xi'}^\beta O_{x,\xi',\zeta}(2^{C\varepsilon k}2^{-k/2})|\lesssim O(2^{C|\beta|\varepsilon k}2^{-k/2}2^{-(n-1)k/2}2^{-k|\beta|}).
\end{equation}
 Denoting $\mathscr{R}^c:=\mathbb{R}^n\setminus \mathscr{R},$ we have that
\begin{align*}
 W_{x,k,z}(\xi')&:= \smallint\smallint\smallint_{\mathscr{R}} e^{2\pi i \Psi_{x,k,z}(\omega, \zeta,y) }s(x,\zeta)\psi(y,\omega)e^{-\mu(y,\omega)\pi i/4} \\
&\times (1-\phi_{-\varepsilon k}(J(y,\omega)))e^{-(1+i\tau)\textnormal{Im}(\Phi(y,\xi'))} |J(y,\omega)|^{1/2}d\omega dy d\zeta\\
&+ \smallint\smallint\smallint_{\mathscr{R}^{c}} e^{2\pi i \Psi_{x,k,z}(\omega, \zeta,y) }s(x,\zeta)\psi(y,\omega)e^{-\mu(y,\omega)\pi i/4} \\
&\times (1-\phi_{-\varepsilon k}(J(y,\omega)))e^{-(1+i\tau)\textnormal{Im}(\Phi(y,\xi'))} |J(y,\omega)|^{1/2}d\omega dy d\zeta\\
&=: W_{x,k,z}(\xi')^{\mathscr{R}}+ W_{x,k,z}(\xi')^{\mathscr{R}^c}.
\end{align*}
Here, the main integral to study is $ W_{x,k,z}(\xi')^{\mathscr{R}}.$ Indeed, arguing as in \cite[Page 17]{Tao} one can demonstrate that $W_{x,k,z}(\xi')^{\mathscr{R}^c}$ is a smoothing symbol. Indeed, one can estimate this integral over the region $\mathscr{R}^c$ by inserting the cut-off $1-\phi_{-(1/2-\varepsilon)k}(\omega-\omega')$ in the integrand of $W_{x,k,z}(\xi')^{\mathscr{R}^c}.$ Since on $\mathscr{R}^c,$ $|J(y,\omega)|\gg 2^{-\varepsilon k,}$ and in view of the identity 
$$\nabla_\omega \Phi_{x,k,z}(\omega,\zeta,y)=\lambda'(\omega-\omega')\nabla_\omega^2\Phi_\tau(x,\omega),$$
one has that
$$|\nabla_\omega \Phi_{x,k,z}(\omega,\zeta,y)|\gg 2^{k}|\omega-\omega'|\gg 2^{\varepsilon k}2^{k/2}.$$ Repeated integration by parts in the $\omega$-variable provides the factor $2^{-\varepsilon k}2^{-k/2}$ but the differentiation of the  term $\phi_{-(1/2-\varepsilon)}$ gives the factor $2^{-\varepsilon k}2^{k/2}$ which shows that this portion of the integral is $O(2^{-Ck})$ for any $k.$

Note that
\begin{align*}
   W_{x,k,z}(\xi')^{\mathscr{R}} &=  \smallint\smallint\smallint_{\mathscr{R}} e^{2\pi i \Psi_{x,k,z}(\omega, \zeta,y) }s(x,\zeta)\psi(y,\omega)e^{-\mu(y,\omega)\pi i/4} \\
&\times (1-\phi_{-\varepsilon k}(J(y,\omega)))e^{-(1+i\tau)\textnormal{Im}(\Phi(y,\xi'))} |J(y,\omega)|^{1/2}d\omega dy d\zeta\\
&=  \smallint\smallint\smallint_{\mathscr{R}} e^{2\pi i \Psi_{x,k,z}(\omega, \zeta,y) }s(x,\nabla_{x}\Phi_\tau(x,\xi'))\psi(y,\omega)e^{-\mu(y,\omega)\pi i/4} \\
&\times (1-\phi_{-\varepsilon k}(J(y,\omega)))e^{-(1+i\tau)\textnormal{Im}(\Phi(y,\xi'))} |J(y,\omega)|^{1/2}d\omega dy d\zeta\\
+& \smallint\smallint\smallint_{\mathscr{R}} e^{2\pi i \Psi_{x,k,z}(\omega, \zeta,y) }O(2^{C\varepsilon k}2^{-k/2})\psi(y,\omega)e^{-\mu(y,\omega)\pi i/4} \\
&\times (1-\phi_{-\varepsilon k}(J(y,\omega)))e^{-(1+i\tau)\textnormal{Im}(\Phi(y,\xi'))} |J(y,\omega)|^{1/2}d\omega dy d\zeta\\
&=:   W_{x,k,z}(\xi')^{\mathscr{R}_1}+W_{x,k,z}(\xi')^{\mathscr{R}_2}
\end{align*} where 
\begin{align*}
  W_{x,k,z}(\xi')^{\mathscr{R}_1}  & :=  \smallint\smallint\smallint_{\mathscr{R}} e^{2\pi i \Psi_{x,k,z}(\omega, \zeta,y) }s(x,\nabla_{x}\Phi_\tau(x,\xi'))\psi(y,\omega)e^{-\mu(y,\omega)\pi i/4} \\
&\times (1-\phi_{-\varepsilon k}(J(y,\omega)))e^{-(1+i\tau)\textnormal{Im}(\Phi(y,\xi'))} |J(y,\omega)|^{1/2}d\omega dy d\zeta,  
\end{align*}
and 
\begin{align*}
&W_{x,k,z}(\xi')^{\mathscr{R}_2}\\
&:=  \smallint\smallint\smallint_{\mathscr{R}} e^{2\pi i \Psi_{x,k,z}(\omega, \zeta,y) }O(2^{C\varepsilon k}2^{-k/2})\psi(y,\omega)e^{-\mu(y,\omega)\pi i/4} \\
&\times (1-\phi_{-\varepsilon k}(J(y,\omega)))e^{-(1+i\tau)\textnormal{Im}(\Phi(y,\xi'))} |J(y,\omega)|^{1/2}d\omega dy d\zeta.
\end{align*}
To evaluate $ W_{x,k,z}(\xi')^{\mathscr{R}_1}$ we compute the integral with respect to $\zeta$ using the distributional identity $$
\smallint e^{2\pi i(x-y)\cdot \zeta}d\zeta=\delta(x-y),$$  as follows
\begin{align*}
     W_{x,k,z}(\xi')^{\mathscr{R}_1}  & =  \smallint\smallint\smallint_{\mathscr{R}} e^{2\pi i \Psi_{x,k,z}(\omega, \zeta,y) }s(x,\nabla_{x}\Phi_\tau(x,\xi'))\psi(y,\omega)e^{-\mu(y,\omega)\pi i/4} \\
&\times (1-\phi_{-\varepsilon k}(J(y,\omega)))e^{-(1+i\tau)\textnormal{Im}(\Phi(y,\xi'))} |J(y,\omega)|^{1/2}d\omega dy d\zeta\\
 & =  \smallint\smallint\smallint_{\mathscr{R}} e^{2\pi i [\xi'\cdot( \nabla_\xi\Phi_\tau(y,\omega)- \nabla_\xi\Phi_\tau(x,\xi'))+ (x-y)\cdot \zeta] }s(x,\nabla_{x}\Phi_\tau(x,\xi'))\psi(y,\omega)e^{-\mu(y,\omega)\pi i/4} \\
&\times (1-\phi_{-\varepsilon k}(J(y,\omega)))e^{-(1+i\tau)\textnormal{Im}(\Phi(y,\xi'))} |J(y,\omega)|^{1/2}d\omega dy d\zeta; \\
& =  \smallint\smallint_{\mathscr{R}} e^{2\pi i [\xi'\cdot( \nabla_\xi\Phi_\tau(y,\omega)- \nabla_\xi\Phi_\tau(x,\xi')) ] }s(x,\nabla_{x}\Phi_\tau(x,\xi'))\psi(y,\omega)e^{-\mu(y,\omega)\pi i/4} \\
&\times (1-\phi_{-\varepsilon k}(J(y,\omega)))e^{-(1+i\tau)\textnormal{Im}(\Phi(y,\xi'))} |J(y,\omega)|^{1/2} \left(\smallint e^{2\pi i(x-y)\cdot \zeta}d\zeta\right)  d\omega dy  \\
& =  \smallint\smallint_{\mathscr{R}} e^{2\pi i [\xi'\cdot( \nabla_\xi\Phi_\tau(y,\omega)- \nabla_\xi\Phi_\tau(x,\xi')) ] }s(x,\nabla_{x}\Phi_\tau(x,\xi'))\psi(y,\omega)e^{-\mu(y,\omega)\pi i/4} \\
&\times (1-\phi_{-\varepsilon k}(J(y,\omega)))e^{-(1+i\tau)\textnormal{Im}(\Phi(y,\xi'))} |J(y,\omega)|^{1/2} \delta(x-y) dy  d\omega\\
& =  \smallint_{\mathscr{R}} e^{2\pi i \xi'\cdot( \nabla_\xi\Phi_\tau(x,\omega)- \nabla_\xi\Phi_\tau(x,\xi'))  }s(x,\nabla_{x}\Phi_\tau(x,\xi'))\psi(x,\omega)e^{-\mu(x,\omega)\pi i/4} \\
&\times (1-\phi_{-\varepsilon k}(J(x,\omega)))e^{-(1+i\tau)\textnormal{Im}(\Phi(x,\xi'))} |J(x,\omega)|^{1/2}   d\omega\\
& =  s(x,\nabla_{x}\Phi_\tau(x,\xi'))e^{-(1+i\tau)\textnormal{Im}(\Phi(x,\xi'))}  \\
&\times \smallint_{\omega-\omega'=O(2^{\varepsilon k}2^{-k/2})} e^{2\pi i \xi'\cdot( \nabla_\xi\Phi_\tau(x,\omega)- \nabla_\xi\Phi_\tau(x,\xi'))  }\psi(x,\omega)e^{-\mu(x,\omega)\pi i/4}\\
&\hspace{2cm}(1-\phi_{-\varepsilon k}(J(x,\omega))) |J(x,\omega)|^{1/2}   d\omega.
\end{align*}
Note that the integration with respect to $\omega$ is over the region $\{\omega-\omega'=O(2^{\varepsilon k}2^{-k/2})\}$ and then we can estimate this integral by introducing a cutt-off $\phi_{-(\frac{1}{2}-\varepsilon)k}(\omega-\omega').$
In accordance with the principle of stationary phase, we now look at where the phase is stationary in $\omega.$ From the homogeneity of the phase function we have that
\begin{align*}
    \nabla_\xi \Phi_\tau(x,\omega) &=(\nabla_{\overline{\xi}}\Phi_\tau(x, (\omega,1)), \partial_{\xi_n}\Phi_\tau(x,(\omega,1))   )\\
    &=(\nabla_\omega\Phi_\tau(x,\omega), \Phi_\tau(x,\omega)-\omega\cdot \nabla_\omega\Phi(x,\omega)  ).
\end{align*} Hence
\begin{align*}
\nabla_\omega(\xi'\cdot \nabla_\xi \Phi_\tau(x,\omega)  )&=\nabla_\omega(\lambda'\omega'\cdot \nabla_\omega \Phi_\tau(x,\omega)+\lambda'(\Phi_\tau(x,\omega)-\omega\cdot \nabla_\omega\Phi_\tau(x,\omega)  ) )\\
    &=\lambda'(\omega'-\omega)\nabla_\omega^2\Phi_\tau(x,\omega).
\end{align*}So, the only stationary points occurs when $\omega=\omega',$ and the Hessian at this stationary point is given by
\begin{equation}
    \det( \nabla_\omega^2(\xi'\cdot \nabla_\xi \Phi_\tau(x,\omega)  ) )|_{\omega=\omega'}=\det(-\lambda' \nabla_\omega^2\Phi_\tau(x,\omega) )|_{\omega=\omega'}=(-\lambda')^{n-1}J(x,\omega').
\end{equation}Consequently, on the support of  $\phi_{-(\frac{1}{2}-\varepsilon)k}(\omega-\omega'),$ we have the Taylor series expansion
\begin{align*}
    2\pi i \xi'\cdot( \nabla_\xi\Phi_\tau(x,\omega)- \nabla_\xi\Phi_\tau(x,\xi')) =\frac{1}{2}\lambda'\nabla_\omega^2\Phi(x,\omega')(\omega-\omega',\omega-\omega')+O(2^{C\varepsilon}2^{-k/2}).
\end{align*}
Similarly, in view of the Taylor expansion
\begin{align*}
     &\psi(x,\omega)e^{-\mu(x,\omega)\pi i/4}(1-\phi_{-\varepsilon k}(J(x,\omega))) |J(x,\omega)|^{\frac{1}{2}}\\
    &=\psi(x,\omega')e^{-\mu(x,\omega')\pi i/4}(1-\phi_{-\varepsilon k}(J(x,\omega'))) |J(x,\omega')|^{\frac{1}{2}}+O(2^{C\varepsilon}2^{-k/2}),
\end{align*} we can write
\begin{align*}
    & s(x,\nabla_{x}\Phi_\tau(x,\xi'))e^{-(1+i\tau)\textnormal{Im}(\Phi(x,\xi'))}  \\
&\times \smallint_{\omega-\omega'=O(2^{\varepsilon k}2^{-k/2})} e^{2\pi i \xi'\cdot( \nabla_\xi\Phi_\tau(x,\omega)- \nabla_\xi\Phi_\tau(x,\xi'))  }\psi(x,\omega)e^{-\mu(x,\omega)\pi i/4}\\
&\hspace{2cm}(1-\phi_{-\varepsilon k}(J(x,\omega))) |J(x,\omega)|^{\frac{1}{2}}   d\omega\\
&= s(x,\nabla_{x}\Phi_\tau(x,\xi'))e^{-(1+i\tau)\textnormal{Im}(\Phi(x,\xi'))} \psi(x,\omega')e^{-\mu(x,\omega')\pi i/4}(1-\phi_{-\varepsilon k}(J(x,\omega'))) |J(x,\omega')|^{\frac{1}{2}}  \\
&\times \smallint_{\omega-\omega'=O(2^{\varepsilon k}2^{-k/2})} e^{ \pi i \lambda'\nabla_\omega^2\Phi(x,\omega')(\omega-\omega',\omega-\omega') }  \phi_{-(\frac{1}{2}-\varepsilon)k}(\omega-\omega') d\omega. 
\end{align*} In view of the stationary phase principle 
\begin{align*}
    &\smallint_{\omega-\omega'=O(2^{\varepsilon k}2^{-k/2})} e^{ \pi i \lambda'\nabla_\omega^2\Phi(x,\omega')(\omega-\omega',\omega-\omega') }  \phi_{-(\frac{1}{2}-\varepsilon)k}(\omega-\omega') d\omega\\
    &=e^{\mu(x,\omega')\pi i/4}|J(x,\omega')|^{-\frac{1}{2}}(\lambda')^{-(n-1)/2}+O(2^{-Ck}).
\end{align*} Consequently,
\begin{align*}
     &W_{x,k,z}(\xi')^{\mathscr{R}_1}\\
     & =  \smallint\smallint\smallint_{\mathscr{R}} e^{2\pi i \Psi_{x,k,z}(\omega, \zeta,y) }s(x,\nabla_{x}\Phi_\tau(x,\xi'))\psi(y,\omega)e^{-\mu(y,\omega)\pi i/4} \\
&\times (1-\phi_{-\varepsilon k}(J(y,\omega)))e^{-(1+i\tau)\textnormal{Im}(\Phi(y,\xi'))} |J(y,\omega)|^{1/2}d\omega dy d\zeta\\
    &= s(x,\nabla_{x}\Phi_\tau(x,\xi'))e^{-(1+i\tau)\textnormal{Im}(\Phi(x,\xi'))} \psi(x,\omega')e^{-\mu(x,\omega')\pi i/4}(1-\phi_{-\varepsilon k}(J(x,\omega'))) |J(x,\omega')|^{\frac{1}{2}}\\
    &+O(2^{C\varepsilon}2^{-k/2})\\
    &\times (e^{\mu(x,\omega')\pi i/4}|J(x,\omega')|^{-\frac{1}{2}}(\lambda')^{-(n-1)/2}+O(2^{-Ck}))\\
    &= s(x,\nabla_{x}\Phi_\tau(x,\xi'))e^{-(1+i\tau)\textnormal{Im}(\Phi(x,\xi'))} (\lambda')^{-(n-1)/2} \psi(x,\omega')(1-\phi_{-\varepsilon k}(J(x,\omega')) \\
   &+O(2^{C\varepsilon}2^{-k/2}))\\
   &= W^{0}_{x,k,z}(\xi')+O(2^{C\varepsilon}2^{-k/2})).
\end{align*}
Because $W_{x,k,z}(\xi')^{\mathscr{R}_1}$ captures the first term of the Taylor expansion of the symbol $s(x,\zeta)$ around $\zeta_0=\nabla_{z}\Phi_\tau(x,\xi'),$ we have proved that $W_{x,k,z}(\xi')$ is essentially a symbol of order $-(n-1)/2.$ Moreover, the factor $e^{-(1+i\tau)\textnormal{Im}(\Phi(x,\xi'))}\in  S^{0}_{1/2,1/2}$ in $ W^{0}_{x,k,z}(\xi')$ shows that
$$ W^{0}_{x,k,z}(\xi'),\, W_{x,k,z}(\xi')^{\mathscr{R}_1}\in S^{-\frac{n-1}{2}}_{1/2,1/2}.  $$ In consequence, the difference
$$ W_{x,k,z}(\xi')- W^{0}_{x,k,z}(\xi') \in  S^{\mu}_{1/2,1/2}$$
has order $\mu,$ where $\mu\leq -\frac{n-1}{2}. $ Now, what we want to prove is that $\mu$ satisfies the inequality order $\mu\leq -\frac{n-1}{2}-\frac{1}{2}+2\varepsilon.$ Since 
$$W_{x,k,z}(\xi')=W_{x,k,z}(\xi')^{\mathscr{R}_1}+W_{x,k,z}(\xi')^{\mathscr{R}_2}+W_{x,k,z}(\xi')^{\mathscr{R}^c}$$ 
$$ = W^{0}_{x,k,z}(\xi')+W_{x,k,z}(\xi')^{\mathscr{R}_2}+W_{x,k,z}(\xi')^{\mathscr{R}^c} \textnormal{ mod }S^{-\infty}  $$
$$ = W^{0}_{x,k,z}(\xi')+W_{x,k,z}(\xi')^{\mathscr{R}_2} \textnormal{ mod }S^{-\infty}  $$
since $W_{x,k,z}(\xi')^{\mathscr{R}^c}\in S^{-\infty} $ is a smoothing symbol the order $\mu$ corresponds to the order of the symbol  $W_{x,k,z}(\xi')^{\mathscr{R}_2}.$ However, in view of \eqref{remainder:term:final}, one has the estimates
\begin{align*}
    |\partial_{\xi'}^\beta O_{x,\xi',\zeta}(2^{C\varepsilon k}2^{-k/2})|\lesssim O(2^{C|\beta|\varepsilon k}2^{-k/2}2^{-(n-1)k/2}2^{-k|\beta|}).
\end{align*}Since $W_{x,k,z}(\xi')^{\mathscr{R}_2}$ also contains the term $e^{-(1+i\tau)\textnormal{Im}(\Phi(y,\xi'))}\in  S^{0}_{1/2,1/2},$ the amplitude $a(x,y,\xi',\zeta):=O_{x,\xi',\zeta}(2^{C\varepsilon k}2^{-k/2})e^{-(1+i\tau)\textnormal{Im}(\Phi(y,\xi'))}$ satisfies the estimates
 $$ |\partial_{\xi'}^\beta [O_{x,\xi',\zeta}(2^{C\varepsilon k}2^{-k/2})e^{-(1+i\tau)\textnormal{Im}(\Phi(y,\xi'))}]|\lesssim O(2^{C|\beta|\varepsilon k}2^{-k/2}2^{-(n-1)k/2}2^{-k|\beta|/2}).
$$   The previous inequality can be written as
$$ |\partial_{\xi'}^\beta [O_{x,\xi',\zeta}(2^{C\varepsilon k}2^{-k/2})e^{-(1+i\tau)\textnormal{Im}(\Phi(y,\xi'))}]|\lesssim O(2^{C_{\beta}\varepsilon k}2^{-k/2}2^{-(n-1)k/2}2^{-k|\beta|}),
$$ where $$C_{\beta}=|\beta|+|\beta|/2.$$
In this way, the previous estimate says that, omitting the errors $2^{C_{\beta}\varepsilon k},$ $W^{0}_{x,k,z}(\xi')- W_{x,k,z}(\xi')$ is essentially a symbol of order $-\frac{(n-1)}{2}-\frac{1}{2}.$ 

Now, we have prepared the way in order to use Proposition 6.1 of \cite{Tao} for deducing the estimate \eqref{W:errors}. Indeed, our estimate in  \eqref{W:errors} is the analog of  (19) in \cite[Page 19]{Tao} which is a consequence of \cite[Proposition 6.1]{Tao}. For completeness, we present such a proposition as follows.
\begin{lemma}\label{Final:lemma}
    If $V_{x,k,z}:=V_{x,k,z}(\xi')$ is a symbol satisfying estimates of the type 
    \begin{align}
         | \partial_{\xi'}^{\beta}V_{x,k,z}(\xi')|\lesssim 2^{C_{\beta}\varepsilon k}2^{-k/2-(n-1)k/2-k|\beta|},
    \end{align}when $|\xi'|\sim 2^{k},$ then, we have the kernel estimate 
\begin{equation}\label{W:errors:cor}
    \Vert \smallint e^{2\pi i[\Phi_\tau(x,\xi') -\xi'\cdot z]}  V_{x,k,z}(\xi')\eta_{k}(\xi')d\xi'\Vert_{L^1(\mathbb{R}^n_x)}\lesssim 2^{C\varepsilon k}2^{-k/2}.
\end{equation}    
\end{lemma} So, note that for the validity of \eqref{W:errors}
 we have used Lemma \ref{Final:lemma} applied to $V_{x,k,z}(\xi'):=W^{0}_{x,k,z}(\xi')- W_{x,k,z}(\xi').$ The proof of the $L^1$-boundedness of $E$ is complete.

\begin{proof}[Proof of the main Theorem \ref{Main:theorem:2}] With the notation in Subsection \ref{Degenerate:non:degenerate:components} the operator $T$ has been decomposed into its degenerate and non-degenerate components, $T_{\textnormal{deg}}$ and $T_{\textnormal{non-deg}},$ respectively. In Subsection \ref{Boundedness of the degenerate component} in Proposition \ref{Proposition:deg} we have proved that the degenerate operator  $T_{\textnormal{deg}}$ is bounded on $L^1.$ As for the non-degenerate operator $T_{\textnormal{non-deg}},$ we have constructed in Subsubsection \ref{Construction of the operatorsA:S} a Fourier integral operator $A,$ which is an average operator, of the same order that $T$ and a pseudo-differential operator $S$ of order zero. We have proved, that the factorisation approach introduced in Tao \cite{Tao} also operates in our case, namely, that $A$ is bounded on $L^1,$ (see Proposition \ref{boundedness:A}) $S$ is of weak (1,1) type in view of the standard Calder\'on-Zygmund theory, and that the error operator  $E=T_{\textnormal{nondeg}}-SA$ is bounded on $L^1,$ see Subsection \ref{L1:boundedness:remaninder}. Since $T=T_{\textnormal{deg}}+SA+E,$ then $T$ is of weak (1,1) type. The proof of Theorem \ref{Main:theorem:2} is complete.    
\end{proof}

\section{Final remarks}\label{Final:rem}

\subsection{Applications}
In order to illustrate the applications of Theorem \ref{Main:theorem:2}   to {\it a priori} estimates for the Cauchy problem with complex characteristics we consider the case where the associated operator is a classical pseudo-differential operator of order one. Let us introduce some required notation.  Let $(X,g)$ be a compact Riemmanian manifold of dimension $n$ and let us consider $T>0.$ Denote by $\Delta_g$ the positive Laplacian associated to the metric $g.$
Denote the space $L^{1}_s(X)$ defined by the norm
\begin{equation*}
    \Vert f\Vert_{L^{1}_s(X)}=\Vert(1+\Delta_g)^{\frac{s}{2}}f\Vert_{L^1(X)}.
\end{equation*}We also consider the {\it weak} $L^1_s$-space defined by the seminorm $$\Vert f\Vert_{L^{1,\infty}_s(X)}=\Vert(1+\Delta_g)^{\frac{s}{2}}f\Vert_{L^{1,\infty}(X)},$$ and by $L^{1,\infty}_{s,\textnormal{loc}}(X),$ its local version, namely the family of functions $g$ such that for any compactly supported function $\phi\in C^{\infty}_0(X),$ $[(1+\Delta_g)^{\frac{s}{2}}g]\phi\in L^{1,\infty}(X).$

Denote $\partial_t$ the partial derivative with respect to the time variable $t\in [0,T],$ and let $D_t=-i\partial_t.$ Let us consider the Cauchy problem for a first-order classical pseudo-differential operator $A,$
\begin{equation}\label{Cauchy:problem:order:m:II}\begin{cases}D_tv-A v=0 ,& (t,x)\in [0,T]\times X,
\\ v(0,x)=f(x) .&   \end{cases}
\end{equation} We assume that its principal symbol $a(t,x,\xi),\,(t,x,\xi)\in [0,T]\times (T^{*}X\setminus \{0\}),$ has non-negative imaginary part, namely, $a$ satisfies $$\textnormal{Im}[a(t,x,\xi)]> 0, \,|\xi|\neq 0.$$ The solution operator $U(t)$ of \eqref{Cauchy:problem:order:m:II} can be realised as a Fourier integral operator of order zero with a global complex phase function of positive type, see \cite[Page 117]{Ruzhansky:CWI-book} and also Laptev, Safarov and Vassiliev \cite{Laptev:Safarov:Vassiliev} for details. Applying (essentially) Theorem \ref{Main:theorem:2} to the operator $$(1+\Delta_g)^{\frac{s}{2}}U(t)(1+\Delta_g)^{-\frac{n-1}{4}-\frac{s}{2} }$$ of order $-(n-1)/2,$ we have that for any compactly supported initial datum $f\in L^1_{s+\frac{n-1}{2}}(X)$, and for a fixed time $t\in [0,T],$ the solution $v$ of the Cauchy problem \eqref{Cauchy:problem:order:m:II} satisfies $v(t,\cdot)\in L^{1,\infty}_{s,\textnormal{loc}}(X).$ This analysis proves the following consequence of Theorem \ref{Main:theorem:2}.
\begin{corollary}\label{Cor:1} Let $s\in \mathbb{R},$ and let $f\in L^1_{s+\frac{n-1}{2}}(X)$ be compactly supported. Then, for a fixed $t\in [0,T]$, the solution $v$ of the Cauchy problem \eqref{Cauchy:problem:order:m:II} satisfies $v(t,\cdot)\in L^{1,\infty}_{s,\textnormal{loc}}(X).$    
\end{corollary}
\begin{proof}
    We write $v(x,t)=U(t)f(x),$ namely, the solution $v$ of \eqref{Cauchy:problem:order:m:II} in terms of the solution operator $U(t).$ Since $f\in L^1_{s+\frac{n-1}{2}}(X), $ we have that 
    $$ (1+\Delta_g)^{\frac{n-1}{4}+\frac{s}{2}}f\in L^{1}(X).    $$ We can use the factorisation
    \begin{align*}
        (1+\Delta_g)^{\frac{s}{2}}v= (1+\Delta_g)^{\frac{s}{2}} U(t)(1+\Delta_g)^{-\frac{n-1}{4}-\frac{s}{2} } (1+\Delta_g)^{\frac{n-1}{4}+\frac{s}{2}}f,
    \end{align*} and the fact that $(1+\Delta_g)^{\frac{s}{2}}U(t)(1+\Delta_g)^{-\frac{n-1}{4}-\frac{s}{2} }$ is a Fourier integral operator with complex phase of order $-(n-1)/2$\footnote{See \cite[Page 117]{Ruzhansky:CWI-book} and also Laptev, Safarov and Vassiliev \cite{Laptev:Safarov:Vassiliev} for details.} to prove  that $v(t,\cdot)\in L^{1,\infty}_{s,\textnormal{loc}}(X).$ Indeed, if $\{\phi_i\}_{i=1}^{N}$ denotes a partition of unity subordinated to a finite open covering $\{O_{i}\}_{i=1}^{N}$  of the compact manifold $X,$ then  $\sum_{i=1}^{N}\phi_i=1,$ and $\textnormal{supp}[\phi_i]\subset O_i.$ We can assume that the covering $\{O_{i}\}_{i=1}^{N}$ is an atlas for the manifold $X,$ and then each $O_{i}$ is diffeomorphic to an Euclidean open set $\tilde{O}_i\subset\mathbb{R}^n,$ $n=\dim(X).$   Note that $(1+\Delta_g)^{\frac{s}{2}} U(t)(1+\Delta_g)^{-\frac{n-1}{4}-\frac{s}{2} }M_{\phi_{i}}$ is a Fourier integral operator of order $-(n-1)/2$ that can be microlocalised to a Fourier integral operator whose symbol is compactly supported in the spatial variable. Here, $M_{\phi_{i}}g=\phi_{i}g$ denotes the multiplication operator by the function $\phi_{i}.$  Then, Theorem \ref{Main:theorem:2} imply that for any $\psi\in C^{\infty}_0(X),$ we have that
    \begin{align*}
       & \Vert \psi (1+\Delta_g)^{\frac{s}{2}}v \Vert_{L^{1,\infty}(X)}\\
       &=  \Vert\psi (1+\Delta_g)^{\frac{s}{2}} U(t)(1+\Delta_g)^{-\frac{n-1}{4}-\frac{s}{2} } (1+\Delta_g)^{\frac{n-1}{4}+\frac{s}{2}}f \Vert_{L^{1,\infty}(X)}\\
        &\lesssim_N\sum_{i=1}^{N} \Vert\psi (1+\Delta_g)^{\frac{s}{2}} U(t)(1+\Delta_g)^{-\frac{n-1}{4}-\frac{s}{2} }M_{\phi_i} (1+\Delta_g)^{\frac{n-1}{4}+\frac{s}{2}}f \Vert_{L^{1,\infty}(X)}\\
        &\leq C_{\psi,N}\Vert (1+\Delta_g)^{\frac{n-1}{4}+\frac{s}{2}}f\Vert_{L^{1}(X)}\lesssim_{\phi,N}\Vert f\Vert_{L^{1}_{s+\frac{n-1}{2}}(X)},
    \end{align*}which proves Corollary \ref{Cor:1}.
\end{proof}
We also observe that a similar analysis can be adapted to investigate the mapping properties of the parametrix of other first-order complex partial differential operators, see e.g. \cite[Page 120]{Ruzhansky:CWI-book} for a problem that appears naturally in the study of the oblique derivative problem. This was observed e.g. by Melin and Sj\"ostrand in \cite{Melin:Sjostrand1975}. For other applications of the theory of Fourier integral operators with complex phases to {\it a priori} estimates for the Cauchy problem with complex characteristics we refer to Trèves \cite[Chapter XI]{Treves:1980:Vol2} and to \cite[Chapter 5]{Ruzhansky:CWI-book}.   

\subsection{Conclusions and bibliographical discussion}

In this work we have analysed the weak (1,1) boundedness of Fourier integral operators with complex-valued phase functions. As in the setting of real-valued phase functions regarding the weak (1,1) boundedness due to Tao in \cite{Tao}, our estimate is local in the sense that we have considered the symbol $a$ to be compactly supported in the $x$-variable, following the previous approaches, see Seeger, Sogge and Stein \cite{SSS}, Beals \cite{Beals1982}, and the monograph of the second author about the mapping properties of Fourier integral operators with complex phases \cite{Ruzhansky:CWI-book}. Among other things, the sharpness of the order $-(n-1)/2$ for the weak (1,1) inequality of elliptic Fourier integral operators has been discussed by the authors in \cite{CardonaRuzhansky:2022}, and we also refer to \cite{Ruzhansky1999} for the sharpness of Seeger-Sogge-Stein orders regarding the $L^p$-boundeness of these operators.      

There has been in the last decades a wide interest in the analysis of the global continuity properties for Fourier integral operators. This problem can be traced back to the work of Asada and  Fujiwara \cite{Asada1978} and of Peral \cite{Peral1980} and Miyachi \cite{Miyachi1980}. We refer to \cite{Ruzhansky2006}, \cite{Ruzhansky:Sugimoto} and \cite{Coriasco:Ruzhansky,Coriasco:Toft2016} for global properties on $L^p$-spaces. As for global criteria for Fourier integral operators on H\"ormander classes removing the calculus condition $\delta<\rho,$ and allowing the ranges $0\leq \delta<1$ and $0< \rho\leq 1,$ we refer the reader to Dos Santos Ferreira and Staubach \cite{Dos:Santos2014},  Kenig and Staubach \cite{KenigStaubach} and the works of Staubach and his collaborators \cite{Staubach2021,Staubach,Rodriguez:Lopez2015,Rodriguez:Lopez}. We also refer to the seminal work of \'Alvarez and Hounie \cite{Hounie} for the boundedness of pseudo-differential operators without the condition $\delta<\rho,$ which was the starting point for the aforementioned results in the setting of Fourier integral operators. For the global parametrization of Lagrangian distributions we refer the reader to Laptev, Safarov, and Vassiliev \cite{Laptev:Safarov:Vassiliev} and to \cite{Ruzhansky:CWI-book}. We also refer to the work \cite{CardonaRuzhansky:MSJ:Japan} for the global definition of Fourier integral operators on compact Lie groups in terms of the group Fourier transform.

The theory of Fourier integral operators with complex phases was developed by  Melin and Sj\"ostrand in \cite{Melin:Sjostrand1975} motivated by the problem of the construction of parametrices for operators of principal type with non-real principal symbols, see for instance \cite{Melin:Sjostrand1976}. The singularities of the Bergman kernel can be also approximated in terms of asymptotic expansions of certain Fourier integral operators with complex-valued functions, see Boutet de Monvel and Sj\"ostrand \cite{Boutet:de:Monvel82}. For a systematic analysis of the Bergman kernel we refer to Fefferman \cite{Fefferman1974,Fefferman1976}.

There has been also wide activity concerning the smoothing estimates for Fourier integral operators after the proof of Bourgain and Demeter \cite{Bourgain:Demeter:2015} of the decoupling conjecture. The smoothing effect for Fourier integral operators was conjectured by Sogge in \cite{Sogge1991}. It will be difficult to review the literature about this problem here, but in what follows we share some works in this direction. Indeed, we refer the reader to Bourgain \cite{Bourgain1991}, Minicozzi and Sogge \cite{Minicozzi:Sogge}, Guth, Wang and Zhang \cite{GuthWangZhang}, Beltran, Hickman and Sogge \cite{BeltranHickmanSogge}, to Gao, Liu, Miao, and Xi \cite{GaoLiuMiaoXi2023} and references therein for details.
\\

\noindent\textbf{Conflict of interests statement - Data Availability Statements}  The authors state that there is no conflict of interest.  Data sharing does not apply to this article as no datasets were generated or
analysed during the current study.
\subsection{Acknowledgement} The authors thank Terence Tao for suggesting the factorisation approach used in this manuscript.

\bibliographystyle{amsplain}

\begin{thebibliography}{99}

\bibitem{Hounie}  \'Alvarez, J.   Hounie, J. Estimates for the kernel and continuity properties of pseudo-differential operators. {\it Ark. Mat.} 28(1), 1--22, (1990). 

\bibitem{Asada1978} Asada, K. Fujiwara, D. On some oscillatory integral transformations in $L^2(\mathbb{R}^n).$  {\it Japan. J. Math.} (N.S.) 4(2), 299--361, (1978).

\bibitem{Beals1982}  Beals, R. M. $L^p$ boundedness of Fourier integral operators. {\it Mem. Amer. Math. Soc.} 38, no. 264, viii+57 pp, (1982).

\bibitem{BeltranHickmanSogge} Beltran, D.  Hickman, J. Sogge, C. D. Variable coefficient Wolff-type inequalities and sharp local smoothing estimates for wave equations on manifolds, {\it Anal. PDE}, 13(2), 403--433, (2020).

\bibitem{Bourgain1991} Bourgain, J.  $L^p$-estimates for oscillatory integrals in several variables, {\it Geom. Funct. Anal.} 1(4), 321--374, (1991). 

\bibitem{Bourgain:Demeter:2015} Bourgain, J. Demeter, C. The proof of the $l^2$ decoupling conjecture, {\it Ann. of Math.} (2) 182:1, 351--389, (2015). 

\bibitem{Boutet:de:Monvel82} Boutet de Monvel, L. Sj\"ostrand, J.
Sur la singularité des noyaux de Bergman et de Szego. (French) Journées: Équations aux Dérivées Partielles de Rennes (1975), pp. 123--164,
{\it Ast\'erisque,} No. 34--35, Soc. Math. France, Paris, 1976. 

\bibitem{CalderonZygmund1952} Calder\'on, A. P., Zygmund, A. On the existence of certain singular integrals. Acta Math., 88, 85--139
(1952). 

\bibitem{CardonaRuzhansky:2022} Cardona, D., Ruzhansky, M.  Sharpness of Seeger-Sogge-Stein orders for the weak (1,1) boundedness of Fourier integral operators.,  {\it Arch. Math.} 119, 189--198, (2022).  

\bibitem{CardonaRuzhansky:MSJ:Japan} Cardona, D., Ruzhansky, M. Subelliptic pseudo-differential operators and Fourier integral operators on compact Lie groups, to appear in  {\it MSJ Memoirs, Math. Soc. Japan},  Tokyo, 180 Pages.

\bibitem{Cordero2009}  Cordero, E. Nicola, F. Rodino, L. Boundedness of Fourier integral operators on $\mathscr{F}L^p$ spaces. {\it Trans. Amer. Math.} Soc. 361(11), 6049--6071, (2009).

\bibitem{Coriasco:Ruzhansky}  Coriasco, S. Ruzhansky, M. Global $L^p$ continuity of Fourier integral operators. {\it Trans. Amer. Math. Soc.} 366(5), 2575--2596, (2014).

\bibitem{Coriasco:Toft2016} Coriasco, S. Johansson, K. Toft, J. Global wave-front properties for Fourier integral operators and hyperbolic problems. {\it J. Fourier Anal. Appl.} 22(2), 85--333, (2016). 

\bibitem{Dos:Santos2014} Dos Santos Ferreira, D. Staubach, W. Global and local regularity of Fourier integral operators on weighted and unweighted spaces, {\it Mem. Amer. Math. Soc.,} 229 (2014), no. 1074, xiv+65 pp.

\bibitem{Duistermaat-Hormander:FIOs-2}
 Duistermaat, J. J., H{\"o}rmander, L.
\newblock Fourier integral operators. {II}.
\newblock {\em Acta Math.}, 128(3-4), 183--269, (1972).

\bibitem{Eskin1970}  Èskin, G. I. Degenerate elliptic pseudodifferential equations of principal type. (Russian) {\it Mat. Sb.} (N.S.) 82(124), 585--628, (1970).

\bibitem{FeffermanStein1972}  Fefferman, C.   Stein, E. M. $H^p$ spaces of several variables. {\it Acta Math.}, 129, 137--193, (1972).

\bibitem{Fefferman1974} Fefferman, C. The Bergman kernel and biholomorphic mappings of pseudoconvex domains, {\it Invent. Math.} 26, 1--65, (1974).

\bibitem{Fefferman1976} Fefferman, C. L. Monge-Ampère equations, the Bergman kernel, and geometry of pseudoconvex domains. {\it Ann. of Math.} 103(2), 395--416, (1976).

\bibitem{GaoLiuMiaoXi2023} Gao, C. Liu, B. Miao, C. Xi, Y.
Square function estimates and local smoothing for Fourier integral operators. {\it Proc. Lond. Math. Soc.} (3), 126(6), 1923--1960, (2023). 

\bibitem{GuthWangZhang}  Guth, L. Wang, H.  Zhang, R. A sharp square function estimate for the cone in $\mathbb{R}^3$, {\it Ann. of Math.}  192(2),
551--581, (2020).

\bibitem{Hormander1969} H\"ormander, L. Lecture notes at the Nordic Summer School of Mathematics, 1969.

\bibitem{Hormander1971Ac}
H\"ormander, L. Fourier integral operators. I, {\it Acta Math.} 127, 79--183, (1971).

\bibitem{Hormander1985III}   H\"ormander, L. { The analysis of the linear partial differential operators,} Vol. III-IV. Springer-Verlag, (1985).

\bibitem{Ho2}
H{\"o}rmander, L.
\newblock  $L^2$ estimates for Fourier integral operators with complex phase.
\newblock {\it Ark. Mat.}, 21, 283--307, (1983).

\bibitem{Staubach2021} Israelsson, A. Rodríguez-López, S. Staubach, W. Local and global estimates for hyperbolic equations in Besov-Lipschitz and Triebel-Lizorkin spaces. {\it Anal. PDE,} 14(1), 1--44, (2021).

\bibitem{Staubach}  Israelsson, A. Mattsson, T. Staubach, W. Boundedness of Fourier integral operators on classical function spaces. {\it J. Funct. Anal.} 285(5), Paper No. 110018, 64 pp, (2023).

\bibitem{KenigStaubach}  Kenig, C. Staubach, W. $\Psi$-pseudodifferential operators and estimates for maximal oscillatory integrals. {\it Studia Math.} 183(3), 249--258, (2007). 

\bibitem{MMNG2023}
A.~Martini and D.~M{\"u}ller,
\emph{An FIO-based approach to $L^p$-bounds for the wave equation on 2-step Carnot groups: the case of Métivier groups},
Analysis $\&$ PDE (to appear), arXiv:2406.04315.

\bibitem{Laptev:Safarov:Vassiliev}  Laptev, A. Safarov, Y.  Vassiliev, D. On global representation of Lagrangian distributions
and solutions of hyperbolic equations, {\it Comm. Pure Appl. Math.}, 47(11), 1411--1456, (1994).

\bibitem{Melin:Sjostrand1975} Melin, A., Sj\"ostrand, J.  Fourier integral operators with complex-valued phase functions. In: Chazarain, J. (eds) Fourier Integral Operators and Partial Differential Equations. Lecture Notes in Mathematics, vol 459. Springer, Berlin, Heidelberg, (1975).

\bibitem{Melin:Sjostrand1976} Melin, A., Sj\"ostrand, J.  Fourier integral operators with complex phase functions and parametrix for an interior boundary value problem, {\it Commun. Partial. Differ. Equ.}, 1(4), 313--400,  (1976).

\bibitem{Minicozzi:Sogge} Minicozzi, W. P. Sogge, C. Negative results for Nikodym maximal functions and related oscillatory integrals in curved space, {\it Math. Res. Lett.} 4:2-3, 221--237, (1997).

\bibitem{Miyachi1980}  Miyachi, A.
On some estimates for the wave operator in
and $L^p$ and $H^p,$
{\it J. Fac. Sci. Univ. Tokyo Sect. IA Math.}, 27, 331--354, (1980).

\bibitem{Nirenberg71} Nirenberg, L., A proof of the Malgrange preparation theorem. Proc. Liverpool Singularities Symp. I, Dept. pure Math. Univ. Liverpool 1969--1970, 97--105, (1971).

\bibitem{Peral1980}  Peral, J. C. $L^p$ estimates for the wave equation. {\it J. Funct Anal.} 36(1), 114--145, (1980).

\bibitem{Rodriguez:Lopez2015} Rodríguez-López, S. Rule, D. Staubach, W. On the boundedness of certain bilinear oscillatory integral operators. {\it Trans. Amer. Math. Soc.} 367(10), 6971--6995, (2015).

\bibitem{Rodriguez:Lopez}  Rodríguez-López, S. Rule, D. Staubach, W. A Seeger-Sogge-Stein theorem for bilinear Fourier integral operators. {\it Adv. Math.} 264, 1--54, (2014). 


\bibitem{Ruzhansky1999} 
Ruzhansky, M.
\newblock On the sharpness of Seeger-Sogge-Stein orders. 
\newblock {\em Hokkaido Math. J.}, 28(2), 357--362, (1999).

\bibitem{Ruzhansky:CWI-book}
Ruzhansky, M.
\newblock { Regularity theory of {F}ourier integral operators with complex
  phases and singularities of affine fibrations}, volume 131 of {\em CWI
  Tract}.
\newblock {\em Stichting Mathematisch Centrum, Centrum voor Wiskunde en Informatica},
  Amsterdam, 2001.

\bibitem{Ruzhansky2006} Ruzhansky, M. Sugimoto, M. Global L2-boundedness theorems for a class of Fourier integral operators. {\it Commun. Partial. Differ. Equ.,} 31(4-6), 547--569, (2006).

\bibitem{Ruzhansky:Sugimoto} Ruzhansky, M. Sugimoto, M. Weighted Sobolev L2 estimates for a class of Fourier integral operators. {\it Math. Nachr.} 284(13), 1715--1738, (2011).

\bibitem{SteinIdentity}  Stein, E. M. Topics in harmonic analysis related to the Littlewood-Paley theory, Vol. 63 of Annals of Mathematics Studies, Princeton University Press, 1970.

\bibitem{SteinBook1993} {Stein, E. M.},
    {Harmonic analysis: real-variable methods, orthogonality, and
              oscillatory integrals},
     {Princeton Mathematical Series},
    {43},
  {Princeton University Press, Princeton, NJ},
    {1993},
      {xiv+695}.

\bibitem{SSS}
Seeger, A.  Sogge, C. D,  Stein, E. M.
\newblock Regularity properties of {F}ourier integral operators.
\newblock {\em Ann. of Math.}, 134(2), 231--251, (1991).

\bibitem{Sogge1991}  Sogge, C. D. Propagation of singularities and maximal functions in the plane, {\it Invent. Math.} 104(2), 349--376, (1991). 


\bibitem{Tao}  Tao, T. The weak-type (1,1) of Fourier integral operators of order $ -(n-1)/2.$ {\it  J. Aust. Math. Soc.} 76(1), 1--21, (2004).

\bibitem{Treves:1980:Vol2}  Trèves, F. Introduction to pseudodifferential and Fourier integral operators. Vol. 2. Fourier integral operators. University Series in Mathematics. Plenum Press, New York-London, 1980. xiv+301–649+xi pp. 

\bibitem{Wolf2003} Wolff, T. Lectures on harmonic analysis. With a foreword by Charles Fefferman
and preface by Izabella Łaba, Edited by Łaba and Carol Shubin. University
Lecture Series, 29, Amer. Math. Soc., Providence, RI, 2003.

\end{thebibliography}

\end{document}